\documentclass[12pt, reqno]{amsart}
\usepackage{amssymb}
\usepackage{amsmath}
\usepackage[dvips]{graphicx}
\usepackage{amsthm}
\usepackage{hyperref}
\usepackage[normalem]{ulem}
\usepackage[capitalize,nameinlink,noabbrev,nosort]{cleveref}
\usepackage[dvipsnames]{xcolor}
\usepackage{mathtools}
\usepackage{enumitem}
\usepackage[headings]{fullpage}
\usepackage{ytableau}
\usepackage{yhmath}
\usepackage{mathdots}
\usepackage{calc}
\usepackage{ytableau}

\usepackage {tikz}
\usetikzlibrary {positioning, calc}
\definecolor {processblue}{cmyk}{0.96,0,0,0}

\usepackage{accents}

\usepackage{fullpage}

\usepackage{tikz}
\usetikzlibrary{arrows}
\usetikzlibrary{decorations.markings}
\usetikzlibrary{tikzmark,decorations.pathreplacing}
\usetikzlibrary{shapes.geometric}
\numberwithin{equation}{section}
\newcommand{\colsSep}{~\mathbin{\raisebox{-.18em}{$\vdots$}}~}
\newcommand{\matrixCols}[2]{#1 \colsSep #2}
\newlength{\firstlen}
\newlength{\secondlen}
\newlength{\maxlen}
\newcommand{\matrixRows}[2]{%
\begin{tikzpicture}
    \settowidth{\firstlen}{$#1$}
    \settowidth{\secondlen}{$#2$}

    \ifdim\firstlen>\secondlen
        \setlength{\maxlen}{\firstlen}
    \else
        \setlength{\maxlen}{\secondlen}
    \fi

    \node (first) at (0,0) {$#1$};
    \node (dots) at (0,.75em) {\customdotfill{\maxlen}};
    \node (second) at (0,1.5em) {$#2$};
\end{tikzpicture}%
}
\makeatletter
\newcommand{\customdotfill}[1]{%
  \leavevmode
  \leaders\hbox to .44em{\hss.\hss}\hskip #1\kern0pt
}
\makeatother

\renewcommand{\sp}{sp}
\newcommand{\oo}{so}

\renewcommand{\oe}{o^{\textup{even}}}
\newcommand{\Z}{\mathbb{Z}}

\DeclareMathOperator{\sgn}{sgn}
\DeclareMathOperator{\inv}{inv}

\newcommand{\x}{\bar{x}}

\newtheorem{thm}{Theorem}[section]

\newtheorem{cor}[thm]{Corollary}
\newtheorem{lem}[thm]{Lemma}
\newtheorem{defn}[thm]{Definition}
\newtheorem{rem}[thm]{Remark}

\crefname{thm}{Theorem}{Theorems}

\title[Affine Jacobi--Trudi identities]{Character expansions and\\ affine Jacobi--Trudi identities}

\author[]{Ilse Fischer}

\author[]{Moritz Gangl}

\author[]{{\'A}lvaro Guti\'errez}

\author[]{Nishu Kumari}

\address{University of Vienna}
\email[I.~Fischer]{ilse.fischer@univie.ac.at}
\email[M.~Gangl]{moritz.gangl@univie.ac.at}
\email[N.~Kumari]{nishu.kumari@univie.ac.at}
\address{University of Bristol}
\email[Á.~Gutiérrez]{a.gutierrezcaceres@bristol.ac.uk}

\thanks{IF, MG and NK acknowledge the financial support from the Austrian Science Fund (FWF) grant 10.55776/F1002. MG also acknowledges the financial support from the Austrian Science Fund (FWF) grant 10.55776/P34931. ÁG was supported by a University of Bristol Research Training Support Grant, and is grateful for the hospitality of the University of Vienna during the research visit where this project took place.}

\begin{document}
\begin{abstract}
    We provide character expansions of certain specialisations of multiparameter Hall--Littlewood polynomials of types $\mathrm{B}_n$, $\mathrm{C}_n$, and $\mathrm{BC}_n$ for rectangular shapes. We use these expansions to establish the six  Jacobi--Trudi-type identities conjectured by Ole Warnaar in~2025.
\end{abstract}
\maketitle

\section{Introduction}\label{sec:introduction}\label{sec:main results}

The classical Jacobi--Trudi identities \cite{Jacobi1841,Trudi1864} express
Schur polynomials as determinants whose entries are complete homogeneous
or elementary symmetric polynomials. For a partition $\lambda$ and an
alphabet $X=(x_1,\ldots,x_n)$, they take the form
\begin{equation}\label{eq: dual JT type A}
 s_\lambda(X)
 =
 \det_{1\leq i,j\leq\ell(\lambda)}
 \bigl(h_{\lambda_i-i+j}(X)\bigr)
 =
 \det_{1\leq i,j\leq\ell(\lambda')}
 \bigl(e_{\lambda'_i-i+j}(X)\bigr),
\end{equation}
where $\lambda'$ is the conjugate partition,
$h_r(X)=s_{(r)}(X)$, and $e_r(X)=s_{(1^r)}(X)$.
We use the conventions $h_0=e_0=1$ and $h_r=e_r=0$ for $r<0$.
The second equality in \eqref{eq: dual JT type A} is usually called
the dual Jacobi--Trudi identity.

Schur polynomials indexed by partitions of length at most $n$ are the
characters of the irreducible polynomial representations of
$\mathrm{GL}_n$. Analogous determinant formulas are
available for the symplectic and orthogonal characters
\cite{KoikeTerada,FulHar91,FulmekKrattenthaler}.
We denote the odd orthogonal, symplectic, and even orthogonal characters
by $\oo_\lambda$, $\sp_\lambda$, and $\oe_\lambda$, respectively, and set
$\overline X=(x_1^{-1},\ldots,x_n^{-1})$ and $\dot e_r=e_r(X,\overline X)$.
For $k\geq1$ and $\lambda\subseteq(k^n)$, the dual Jacobi--Trudi
identities are
\begin{subequations}\label{eq:intro-classical-characters}
\begin{align}
    \oo_\lambda(X) &= \det\limits_{1\le i,j\le k}(\dot e_{\lambda'_i-i+j}+ \dot e_{\lambda'_i-i-j+1}),\tag{dJT $\operatorname{B}_n$}\label{eq: dual JT type B}\\
    \sp_\lambda(X) &= \det\limits_{1\le i,j\le k}(\dot e_{\lambda'_i-i+j}-\dot  e_{\lambda'_i-i-j}),\tag{dJT $\operatorname{C}_n$}\label{eq: dual JT type C}\\
    \oe_\lambda(X) &= \frac{1}{2}\det\limits_{1\le i,j\le k}(\dot e_{\lambda'_i-i+j} +\dot e_{\lambda'_i-i-j+2}).\tag{dJT $\operatorname{D}_n$}\label{eq: dual JT type D}
\end{align}
\end{subequations}
Here $\dot e_0=1$ and $\dot e_r=0$ for $r<0$ or $r>2n$.
These identities provide a useful connection between symmetric functions,
classical characters, and determinant evaluations. Their combinatorial
proofs use tableaux and nonintersecting lattice paths; see
\cite{FulmekKrattenthaler,AFHS}.

In this paper, we establish six generalisations of these formulas for
specialisations of the multiparameter Hall--Littlewood polynomials of
types $\mathrm{B}_n$, $\mathrm{C}_n$, and $\mathrm{BC}_n$, indexed by
rectangular partitions $(k^n)$. Hall--Littlewood polynomials associated
with root systems generalise the corresponding classical characters
\cite{Macdonald00}. The three families needed here can all be obtained
from the $\mathrm{BC}_n$ Hall--Littlewood polynomials.

For a partition $\lambda=(\lambda_1,\ldots,\lambda_n)$, where zero parts
are allowed, let $m_r(\lambda)$ be the number of parts equal to $r$, so
that $m_0(\lambda)=n-\ell(\lambda)$, and put
\[
 v_\lambda(t)
 =
 \prod_{r\geq0}\prod_{j=1}^{m_r(\lambda)}
 \frac{1-t^j}{1-t}.
\]
Let $\mathbb Z_2\wr S_n$ act by permuting the variables and replacing any
subset of them by their inverses. The $\mathrm{BC}_n$ Hall--Littlewood
polynomial is given by \cite{Venkateswaran,Macdonald00,ole}
\begin{multline}\label{eq:intro-bc-definition}
 P_\lambda^{\mathrm{BC}_n}(X;t,s_1,s_2)
 =
 \frac{1}
 {v_\lambda(t)\prod_{i=1}^{n-\ell(\lambda)}
                  (1-s_1s_2t^{i-1})}
 \\
 {}\times
 \sum_{\sigma\in\mathbb Z_2\wr S_n}
 \sigma.\!\left(
   \prod_{i=1}^n
   x_i^{-\lambda_i}
   \frac{(1-s_1x_i)(1-s_2x_i)}{1-x_i^2}
   \prod_{1\leq i<j\leq n}
   \frac{(tx_i-x_j)(1-tx_ix_j)}
        {(x_i-x_j)(1-x_ix_j)}
 \right).
\end{multline}
The type $\mathrm{B}_n$ and $\mathrm{C}_n$ polynomials are the
specialisations
\begin{equation}\label{eq:intro-bc-specialisations}
\begin{aligned}
 P_\lambda^{\mathrm{B}_n}(X;t,s)
 &=P_\lambda^{\mathrm{BC}_n}(X;t,s,-1),\\
 P_\lambda^{\mathrm{C}_n}(X;t,s)
 &=P_\lambda^{\mathrm{BC}_n}(X;t,s^{1/2},-s^{1/2}).
\end{aligned}
\end{equation}
In particular, the classical characters are recovered by
\[
 P_\lambda^{\mathrm{B}_n}(X;0,0)=\oo_\lambda(X),
 \qquad
 P_\lambda^{\mathrm{C}_n}(X;0,0)=\sp_\lambda(X),
 \qquad
 P_\lambda^{\mathrm{B}_n}(X;0,1)=\oe_\lambda(X).
\]
Our convention for $\oe_\lambda$ is the even orthogonal character
appearing in \eqref{eq: dual JT type D}; its Weyl determinant formula,
together with those for $\oo_\lambda$ and $\sp_\lambda$, is recalled in
\S\ref{sec:preliminaries}.

Jacobi--Trudi-type formulas for parameter generalisations of Schur
polynomials have also been studied for rectangular Jack polynomials
\cite{Matsumoto} and near-rectangular Jack polynomials \cite{BBL}, where
hyperdeterminants replace determinants.
For the ordinary Hall--Littlewood polynomials, Warnaar
\cite[Theorem~1.1]{ole} proved the following identity for a rectangular
partition:
\begin{equation}\label{eq:intro-warnaar-a}
 P_{(k^r)}(X;t)
 =
 \sum_{\substack{y\in\mathbb Z^k\\y_1+\cdots+y_k=0}}
 \det_{1\leq i,j\leq k}
 \left(
   t^{k\binom{y_i}{2}+iy_i}
   e_{r-i+j-ky_i}(X)
 \right),
\end{equation}
where $k\geq1$ and $r\geq0$.
He further conjectured six identities of a similar form for the
specialisations of the $\mathrm{B}_n$, $\mathrm{C}_n$, and
$\mathrm{BC}_n$ Hall--Littlewood polynomials considered here, and proved
them for $k=1$.
In each case, the right-hand side is a sum over integer vectors of
determinants whose entries involve elementary symmetric polynomials and
quadratic powers of $t$.
We prove these conjectures in Corollaries~\ref{cor-b},
\ref{cor-type-c}, and~\ref{cor-type-bc}.
The character expansions in Theorems~\ref{thm-b-char},
\ref{main-type-c}, and~\ref{thm-type-bc-0-char-all} are the main
intermediate results and provide a common approach to the six proofs.

We next introduce the notation needed to state these results.
For $y=(y_1,\ldots,y_k)\in\mathbb Z^k$, write
$|y|=\sum_{i=1}^k y_i$. We use Iverson notation:
$[\text{statement}]$ is $1$ if the statement is true and $0$ otherwise.
For a vector $\nu=(\nu_1,\ldots,\nu_k)$, let $\mathrm{sort}(\nu)$ denote
its decreasing rearrangement and let
\[
 \inv(\nu)
 =
 \bigl|\{(i,j):1\leq j<i\leq k,\ \nu_i<\nu_j\}\bigr|
\]
be its number of inversions.
If the entries of $\nu$ are distinct elements of
$\{0,\ldots,n+k-1\}$, define $\lambda(\nu)\subseteq(k^n)$ by
\[
 \lambda(\nu)'_i
 =
 n-\mathrm{sort}(\nu)_{k+1-i}+i-1,
 \qquad 1\leq i\leq k.
\]
Equivalently, using the notation of \S\ref{sec:lemmas.complements},
\[
 \lambda(\nu)
 =
 \left(
   \bigl(\mathrm{sort}(\nu)-\rho_k\bigr)^{\square n^k}
 \right)',
\]
where $\rho_k=(k-1,k-2,\ldots,0)$ and
$\mu^{\square n^k}=(n-\mu_k,\ldots,n-\mu_1)$ for
$\mu\subseteq(n^k)$.
All sums in the character expansions below are restricted to integer
vectors $y$ for which the associated entries of $\nu$ are distinct and
belong to $\{0,\ldots,n+k-1\}$, in addition to any restriction explicitly
displayed. Thus these sums are finite. The determinant sums in the
corollaries are finite as well, since $\dot e_r=0$ outside $0\leq r\leq2n$.
The relation between these two descriptions of the summation ranges is
proved in \S\ref{sec: Handling the summation range}.

We begin with the three specialisations of type $\mathrm{B}_n$.

\begin{thm}\label{thm-b-char}
Let $n,k\geq1$. Then
\begin{subequations}
\begin{equation}
    \label{thm-type-b-0-char}
    P_{(k^n)}^{\mathrm{B}_n}(X;t,0)=\sum_{y \in \mathbb{Z}^k}
(-1)^{|y| + \inv(\nu)}
t^{\sum_{j=1}^k \left( (k+\frac{1}{2}) y_j^2 + (j-\frac{1}{2})y_j \right)} 
\oo_{\lambda(\nu)}(X),
\end{equation}
where $K=2k+1$ and $\nu_a = |Ky_a+a|-[y_a\ge0]$ for $a=1,\dots,k$;
\begin{equation}
\label{thm-type-b-t-char}
P_{(k^n)}^{\mathrm{B}_n}(X;t,-t^{1/2})=\sum_{y\in\Z^k}
(-1)^{\inv(\nu)}
t^{\sum_{j=1}^k \left( ky_j^2 + (j-\frac{1}{2})y_j \right)}
\oo_{\lambda(\nu)}(X),
\end{equation}
where $K=2k$ and $\nu_a = |Ky_a + a| - [y_a\ge0]$ for $a=1,\dots,k$;
and
    \begin{equation}
\label{thm-type-b-1-char}
     P_{(k^n)}^{\mathrm{B}_n}(X;t,1)=
\sum_{y \in \mathbb{Z}^k}
(-1)^{|y|+\mathrm{inv}(\nu)}
t^{\sum_{j=1}^k \left(k y_j^2 + (j-1)y_j \right)}
\oe_{\lambda(\nu)}(X),
    \end{equation}
\end{subequations}
where $K=2k$ and $\nu_a = |Ky_a + a-1|$ for $a = 1,\ldots,k$.
In each sum, the entries of $\nu$ are required to be distinct and to lie in $\{0,\ldots,n+k-1\}$.
\end{thm}

Applying the odd and even orthogonal dual Jacobi--Trudi identities to Theorem~\ref{thm-b-char} gives the following three formulas, conjectured in \cite[Conjecture~1.2]{ole}.

\begin{cor}[{\cite[Conjecture 1.2]{ole}}] \label{cor-b}
Let $n,k\geq1$. Then
\begin{subequations}
\begin{equation}
    \label{thm-b-0}
    P_{(k^n)}^{\mathrm{B}_n}(X;t,0)=\sum_{y \in \mathbb{Z}^k}
\det_{1 \le i,j \le k}
\left(
(-1)^{y_i}
t^{\frac{1}{2} K y_i^2 - (j-\frac{1}{2})y_i}
\left(
\dot{e}_{\,n-i+j-K y_i}
+
\dot{e}_{\,n+i+j- Ky_i-1}
\right)
\right),
\end{equation}
where $K=2k+1$, as well as
\begin{equation}
\label{thm-type-b-t}
 P_{(k^n)}^{\mathrm{B}_n}(X;t,-t^{1/2})=\sum_{y \in \mathbb{Z}^k}
\det_{1 \le i,j \le k}
\left(
t^{\frac{1}{2} K y_i^2 - (j-\frac{1}{2})y_i}
\left(
\dot{e}_{\,n-i+j-K y_i}
+
\dot{e}_{\,n+i+j- Ky_i-1}
\right)
\right)\,,
\end{equation}
and
    \begin{equation}
\label{thm-type-b-1}
     P_{(k^n)}^{\mathrm{B}_n}(X;t,1)=
\frac{1}{2}
\sum_{y \in \mathbb{Z}^k}
\det_{1 \le i,j \le k}
\left(
(-1)^{y_i}
t^{\frac{1}{2}K y_i^2 - (j-1)y_i}
(
\dot{e}_{\,n-i+j-K y_i}
+
\dot{e}_{\,n+i+j- Ky_i-2})
\right),
    \end{equation}
\end{subequations}
where $K=2k$.
\end{cor}

The next two character expansions concern type $\mathrm{C}_n$. Their Jacobi--Trudi consequences settle \cite[Conjecture~1.3]{ole}.

\begin{thm}\label{main-type-c}
Let $n,k\geq1$. Then
\begin{subequations}
\begin{equation}
        \label{thm-type-c-0-char}
    P_{(k^n)}^{\mathrm{C}_n}(X;t,0)=
\sum_{y \in \mathbb{Z}^{k}}
(-1)^{\mathrm{inv}(\nu)+\sum_{j=1}^k [y_j<0]}
t^{\sum_{j=1}^k \left(\frac{1}{2}Ky_j^2 + jy_j \right)}
\sp_{\lambda(\nu)}(X),
    \end{equation}
    where $K=2k+2$ and $\nu_a = |Ky_a + a| - 1$ for $a=1,\dots,k$, and
    \begin{equation}
\label{thm-type-c-t-char}
    P_{(k^n)}^{\mathrm{C}_n}(X;t,t)=
\sum_{\substack{y \in \mathbb{Z}^{k}\\ y_k\ge0}}
(-1)^{|y|+\mathrm{inv}(\nu)+\sum_{j=1}^k [y_j<0]}
t^{\sum_{j=1}^k \left( \frac{1}{2}Ky_j^2 + jy_j \right)}
\sp_{\lambda(\nu)}(X),
\end{equation}
\end{subequations}
where $K=2k$ and $\nu_a = |Ky_a + a|-1$ for $a=1,\dots,k$.
In each sum, the entries of $\nu$ are required to be distinct and to lie in $\{0,\ldots,n+k-1\}$.
\end{thm}

The symplectic dual Jacobi--Trudi identity gives the following corollary. The condition $y_k\geq0$ in the second character expansion corresponds to the condition that $|y|$ is even in the determinant sum.

\begin{cor}[{\cite[Conjecture 1.3]{ole}}]\label{cor-type-c}
Let $n,k\geq1$. Then
\begin{subequations}
\begin{equation}
     \label{cor-type-c-0}
P_{(k^n)}^{\mathrm{C}_n}(X;t,0)=\sum_{y \in \mathbb{Z}^k}
\det_{1 \le i,j \le k}
\left(
t^{\frac{1}{2} K y_i^2 - jy_i}
\left(
\dot{e}_{\,n-i+j-K y_i}
-
\dot{e}_{\,n+i+j- Ky_i}
\right)
\right),
\end{equation}
where $K = 2k+2$.
Furthermore,
\begin{equation}
\label{cor-type-c-t}
P_{(k^n)}^{\mathrm{C}_n}(X;t,t)=
    \sum_{\substack{y \in \mathbb{Z}^k,\\ |y| \text{ even }}}
\det_{1 \le i,j \le k}
\left(
t^{\frac{1}{2} K y_i^2 - jy_i}
\left(
\dot{e}_{\,n-i+j-K y_i}
-
\dot{e}_{\,n+i+j- Ky_i}
\right)
\right),
\end{equation}
\end{subequations}
    where $K=2k$.
\end{cor}

Finally, we consider the $\mathrm{BC}_n$ specialisation. Its expansion is again in symplectic characters.

\begin{thm}
\label{thm-type-bc-0-char-all}
Let $n,k\geq1$. Then
\begin{equation}
\label{thm-type-bc-0-char}
        P_{(k^n)}^{\mathrm{BC}_n}(X;t,-t^{1/2},0)=
\sum_{y \in \mathbb{Z}^{k}}
(-1)^{|y| + \mathrm{inv}(\nu)+ \sum_{j=1}^k [y_j<0]}
t^{\sum_{j=1}^k \left( \frac{1}{2}Ky_j^2 + jy_j \right)}
\sp_{\lambda(\nu)}(X),
    \end{equation}
    where $K=2k+1$ and $\nu_a = |Ky_a+a|-1$ for $a=1,\ldots,k$.
In the sum, the entries of $\nu$ are required to be distinct and to lie in $\{0,\ldots,n+k-1\}$.
\end{thm}

The corresponding determinant identity is the remaining conjecture of Warnaar.

\begin{cor}[{\cite[Conjecture 1.4]{ole}}]
\label{cor-type-bc}
Let $n,k\geq1$. Then
\begin{equation}
\label{cor-type-bc-0}
        P_{(k^n)}^{\mathrm{BC}_n}(X;t,-t^{1/2},0)=
    \sum_{y \in \mathbb{Z}^k}
\det_{1 \le i,j \le k} 
\left(
(-1)^{y_i}
t^{\frac{1}{2} K y_i^2 - jy_i}
\left(
\dot{e}_{\,n-i+j-K y_i}
-
\dot{e}_{\,n+i+j- Ky_i}
\right)
\right),
    \end{equation}
    where $K=2k+1$.
\end{cor}

These results also complete the derivation of the
$q,t$-Rogers--Ramanujan identities from the finite Hall--Littlewood
identities in \cite[\S5.1]{ole}. Warnaar already gave independent proofs
of those Rogers--Ramanujan identities in \cite[\S5.3]{ole}; the results
above establish the conjectured determinant identities used in the
alternative derivation.

We briefly describe the common idea of the proofs.
Starting from \eqref{eq:intro-bc-definition} for $\lambda=(k^n)$, we
first obtain a single determinant expression.
For each of the six parameter specialisations, its multilinear expansion
can be described by orientations of a labelled graph whose components
are paths. Orientations with two edges directed towards the same vertex
give determinants with repeated columns and hence contribute zero.
The remaining orientations are determined by one root in each component.
The roots specify the omitted columns and therefore the partition
indexing the corresponding character, while the edge weights and the
column order determine the power of $t$ and the sign.
Applying the classical dual Jacobi--Trudi identities to the resulting
character expansions and rearranging the sums then gives
Corollaries~\ref{cor-b}, \ref{cor-type-c}, and~\ref{cor-type-bc}.

\medskip
\noindent\textbf{Outline of the paper.}
In \S\ref{sec:preliminaries} we recall the character formulas and
Hall--Littlewood polynomials needed in the proofs.
Section~\ref{sec:type-a} treats the case $r=n$ of
\eqref{eq:intro-warnaar-a} as an elementary example.
The auxiliary results are established in \S\ref{sec:lemmas};
in particular, \S\ref{sec:noodle-doodle} gives the common determinant
expansion and its six specialisations.
The character expansions and their Jacobi--Trudi consequences are proved
in \S\!\S\ref{sec:type-b}, \ref{sec:type-c}, and~\ref{sec:type-bc},
for types $\mathrm{B}_n$, $\mathrm{C}_n$, and $\mathrm{BC}_n$,
respectively.

\section{Preliminaries}\label{sec:preliminaries}
\subsection{Classical, odd orthogonal, symplectic and even orthogonal Schur polynomials}\label{subsec: c,oo,s,eo schur polynomials}
For the classical Lie algebras $\mathfrak{sl}(n+1)$ (type $\mathrm{A}_n$), $\mathfrak{so}(2n+1)$  (type $\mathrm{B}_n$), $\mathfrak{sp}(2n)$ (type $\mathrm{C}_n$), and $\mathfrak{so}(2n)$ (type $\mathrm{D}_n$) their characters can be expressed via \emph{Weyl's character formula}. We point to \cite{Humphreys,FulHar91} for a general reference.

Let $\bar{x}_i$ be a shorthand for $x_i^{-1}$.
The character of the irreducible representation labelled by the partition $\lambda$ satisfies, for each of the four classical types, corresponding to the classical, odd orthogonal, symplectic and even orthogonal Schur polynomials, the following formula:
\begin{align}
    s_\lambda(X) &= \frac{\det\limits_{1\le i,j\le n}(x_j^{\lambda_i+n-i})}{\det\limits_{1\le i,j\le n}(x_j^{n-i})},\tag{Weyl $\operatorname A_n$}\label{eq: Weyl character formula type A}\\
    \oo_\lambda(X) &= \frac{\det\limits_{1\le i,j\le n}(x_j^{\lambda_i+n-i+1/2}-\bar{x}_j^{\lambda_i+n-i+1/2})}{\det\limits_{1\le i,j\le n}(x_j^{n-i+1/2}-\bar{x}_j^{n-i+1/2})},\tag{Weyl $\operatorname B_n$}\label{eq: Weyl character formula type B}\\
    \sp_\lambda(X) &= \frac{\det\limits_{1\le i,j\le n}(x_j^{\lambda_i+n-i+1}-\bar{x}_j^{\lambda_i+n-i+1})}{\det\limits_{1\le i,j\le n}(x_j^{n-i+1}-\bar{x}_j^{n-i+1})},\tag{Weyl $\operatorname C_n$}\label{eq: Weyl character formula type C}\\
    \oe_\lambda(X) &= 
    \frac{{2} \displaystyle\det\limits_{1\le i,j\le n}(x_i^{\lambda_j+n-j}+\x_i^{\lambda_j+n-j})}
    {\displaystyle (1+\delta_{\lambda_n,0})
    \det\limits_{1\le i,j\le n}
    (x_i^{n-j}+\x_i^{n-j})}
    ,\tag{Weyl $\operatorname D_n$}\label{eq: Weyl character formula type D}
\end{align}
where $\delta_{i,j}\coloneqq[i=j]$ is the Kronecker delta. Note that for $\lambda_n>0$ the displayed $ \oe_\lambda(X)$ is the sum of two irreducible $\mathfrak{so}(2n)$ characters.

The denominator, sometimes called \emph{Weyl denominator}, in each of the four cases factorises into a product formula, and we make use of the following formulae multiple times throughout the later chapters.
\begin{align}\det_{1\le i,j\le n}(x_j^{n - i})
&= \prod_{1\le i<j\le n}(x_i-x_j)\tag{$\operatorname{A}_n$}\label{glde}\\
    \det_{1\le i,j\le n}(x_i^{n-j+1/2}-\x_i^{n-j+1/2})
    &= (-1)^{\binom{n}{2}}\prod_{i=1}^n\x_i^{n-1}\prod_{i=1}^n\left(x_i^{1/2}-\x_i^{1/2}\right)\prod_{1\leq i<j\leq n}(1-x_ix_j)(x_i-x_j),\tag{$\operatorname{B}_n$}\label{eq: Denom type B}\\
    \det_{1\le i,j\le n}(x_i^{n-j+1}-\x_i^{n-j+1})
&= (-1)^{\binom{n}{2}}\prod_{i=1}^n\x_i^{n-1}\prod_{i=1}^n(x_i-\x_i)\,\prod_{1\leq i<j\leq n}(1-x_ix_j)(x_i-x_j),\tag{$\operatorname{C}_n$}\label{eq: Denom type C}\\
\det_{1\le i,j\le n}(x_i^{n-j}+\x_i^{n-j})
&= {2}(-1)^{\binom{n}{2}} \prod_{i=1}^n\x_i^{n-1}\prod_{1\leq i<j\leq n}(1-x_ix_j)(x_i-x_j).\tag{$\operatorname{D}_n$}\label{eq: Denom type D}
\end{align}

\begin{rem}
    Notice that 
\begin{equation}\label{too many rows}
s_\lambda(X) = \sp_\lambda(X)  = \oo_\lambda(X) = \oe_\lambda(X) = 0, \quad
\text{if } n < \ell(\lambda).
\end{equation}
We also use the convention that 
\begin{equation}\label{eq: convention not a partiton}
s_\lambda(X) = \sp_\lambda(X)  = \oo_\lambda(X) = \oe_\lambda(X) = 0, \quad
\text{if } \lambda \textrm{ is not a partition}.
\end{equation}
\end{rem}

\subsection{Hall--Littlewood polynomials}\label{subsec: Hall Littlewood polynomials}
We let $S_n$ act naturally on polynomials in $X=(x_1,\dots,x_n)$. The \emph{ordinary Hall--Littlewood polynomial} $P_\lambda(X,t)$ is defined as
\begin{equation}
\label{eqn: expansion A}
    P_\lambda(X,t)=\frac{1}{v_\lambda(t)} \sum_{\sigma \in S_n} \sigma.\! \left(
    \prod_{i=1}^nx_i^{\lambda_i}
    \prod_{1\leq i<j\leq n} \frac{x_i-t x_j}{x_i-x_j}
    \right).
\end{equation}
The ordinary Hall--Littlewood polynomials play a central role in algebraic
combinatorics and representation theory: they encode the structure of Hall
algebras, enter the character theory of finite general linear groups, realize
spherical functions for general linear groups over non-archimedean local
fields, and, through the Kostka--Foulkes polynomials, provide $q$-analogues
of weight multiplicities with important combinatorial and geometric
interpretations~\cite{Macdonald1995,Green1955,Lusztig1983,Kato1982}.

\section{An illustrative example}\label{sec:type-a}

In all of our main results, we have $\lambda = (k^n)$, where $n$ is the number of variables in $X$. As an illustrative easy example of some of the techniques to be used, we prove the special case $r=n$ of Warnaar's \cite[Theorem 1.1]{ole}.
\begin{thm}
The type $\mathrm{A}_n$ Hall--Littlewood polynomial indexed by a rectangular partition of length $n$ satisfies the following dual Jacobi--Trudi formula:
\label{thm-type-a}
\[
P_{(k^n)}(X;t) = \sum_{\substack{y\in\Z^k\\|y|=0}} \det_{1\le i,j\le k}(t^{k\binom{y_i}{2}+iy_i} e_{n-i+j-ky_i}(X))\,.
\]
\end{thm}
We note that the polynomial in the above equation is just one monomial.
\begin{proof}
Begin by expanding the Hall--Littlewood polynomial of type $\mathrm{A}_n$,
\[
P_{(k^n)}(X;t) = \frac{1}{[n]_t!}\sum_{\sigma\in S_n} \sigma.\!\left( \prod_{i=1}^n x_i^{\lambda_i}
 \prod_{i<j} \frac{x_i-t x_j}{x_i-x_j}
\right).
\]
When $\lambda = (k^n)$, then note that $\prod_{i=1}^nx_i^k$ is invariant under the action of the symmetric group. We get
\[
P_{(k^n)}(X;t) =\frac{\prod_{i=1}^n x_i^{k}}{[n]_t!}\sum_{\sigma\in S_n} \sigma.\!\left(
 \prod_{i<j} \frac{x_i-t x_j}{x_i-x_j}
\right)= \frac{\prod_{i=1}^nx_i^k}{[n]_t!} [n]_t!= \prod_{i=1}^n x_i^k = s_{(k^n)}(X)\,, 
\]
where we have used \cite[p. 207, (1.4)]{Macdonald1995}.

Now consider the right-hand side of the identity we aim to show, and factor out the power of $t$ to get
\begin{equation}
    \label{eq:power of t A}
\sum_{\substack{y\in\Z^k\\|y|=0}} \det_{1\le i,j\le k}(t^{k\binom{y_i}{2}+iy_i} e_{n-i+j-ky_i}(X))
= 
\sum_{\substack{y\in\Z^k\\|y|=0}} 
\Big( \prod_{i=1}^k t^{k\binom{y_i}{2}+iy_i} \Big)
\det_{1\le i,j\le k}(e_{n-i+j-ky_i}(X))\,.
\end{equation}
If $y\neq 0$ and $|y|=0$ then there exists an $y_i<0$. We then have that
$$n-i+r-ky_i\geq n-i+r+k>n,$$
for every column index $1\leq r\le k$. Thus the $i$-th row of the matrix in the determinant is zero and subsequently the determinant vanishes and only $y=0$ survives.
An application of the 
dual Jacobi--Trudi identity \eqref{eq: dual JT type A} gives that the right-hand side is $s_{(n^k)'}(X)$, as claimed.
\end{proof}

Two critical steps of this proof do not work for general type and require a finer analysis. The first, the step in~\eqref{eq:power of t A} where the power of $t$ from the right-hand side factors out of the determinant. In the general case, this will be accomplished by Lemma~\ref{lem:powery}. The other subtle step is the analysis of which determinants of an infinite sum vanish. The general tool in this case is presented as six lemmas in \S\ref{sec:noodle-doodle}.

We also need to derive an expression for the Hall--Littlewood polynomials of other types which is more amenable to our proof strategy. This is done in Lemma~\ref{magical manipulation}.
Finally, in order to control the partitions labelling the character expansions, we introduce Lemma~\ref{lem:flip}.

\section{A handful of useful lemmas}\label{sec:lemmas}
In this section, we present various auxiliary lemmas which are the main ingredients of the proofs of our main results stated in \S\ref{sec:main results} presented in \S\!\S\ref{sec:type-b}, \ref{sec:type-c} and \ref{sec:type-bc}. In \S\ref{subsec: First Lemma} we provide an alternative expression of the $\mathrm{BC}_n$ Hall--Littlewood polynomials in terms of a determinant. In the special cases of our theorems presented in \S\ref{sec:main results}, this determinant can be expanded as a sum of simpler determinants using a combinatorial argument on certain graphs. This machinery is the core of our proof idea and is presented in \S\ref{sec:noodle-doodle}. To relate the character expansions of the specialisations of the $\mathrm{BC}_n$ Hall--Littlewood polynomials to the Jacobi--Trudi identities conjectured by Warnaar, we need a determinant trick provided in \S\ref{subsec: pulling out of det} and a way of identifying the partitions that come out of our calculations presented in \S\ref{sec:lemmas.complements} as well as a way of handling the respective summation ranges explained in \S\ref{sec: Handling the summation range}. 

\subsection{The $\mathrm{BC}_n$ Hall--Littlewood polynomials as determinants}\label{subsec: First Lemma}

\begin{lem}\label{magical manipulation}
Let $n,k\in\Z_{>0}$. The (multiparameter) Hall--Littlewood polynomial of type $\mathrm{BC}_n$ indexed by a rectangular partition satisfies
    \begin{multline*}
    P_{(k^n)}^{\mathrm{BC}_n}(X;t,s_1,s_2)= \frac{\sqrt{t}^{\binom{n}{2}}}{2}
    \prod_{i=1}^nx_i^{n-1}\prod_{1\leq i<j\leq n}\frac{1}{(1-x_ix_j)(x_i-x_j)}\\\times\det_{1\le i,j\le n}\left(\frac{(1-s_1x_i)(1-s_2x_i)}{1-x_i^2}
    ((\sqrt{t})^{j-1}x_i^{j-1-k}+
    (\sqrt{t})^{-j+1}x_i^{-j+1-k})\right.\\\left.-\frac{(s_1-x_i)(s_2-x_i)}{1-x_i^2}
    ((\sqrt{t})^{j-1}x_i^{-j+1+k}
    +(\sqrt{t})^{-j+1}x_i^{j-1+k})\right).
\end{multline*}
\end{lem}

\begin{proof}
The underlying set of $\Z_2\wr S_n$ is $\Z_2^n\times S_n $.
Split the sum in the expression \eqref{eq:intro-bc-definition}
over the wreath product accordingly,
\begin{multline*}
    P_{(k^n)}^{\mathrm{BC}_n}(X;t,s_1,s_2) = \frac{1}{ v_{(k^n)}(t)} \\
    \times \sum_{\alpha \in \{\pm 1\}^n}\left[\sum_{\sigma\in S_n}\sigma.\!\left(\prod_{i=1}^n\x_i^{k}\frac{(1-s_1x_i)(1-s_2x_i)}{1-x_i^2}\prod_{1\leq i<j\leq n}\frac{(tx_i-x_j)(1-tx_ix_j)}{(x_i-x_j)(1-x_ix_j)}\right)\right]_{X=X^\alpha},
\end{multline*}
where we denote $X^\alpha=(x_1^{\alpha_1},\dots,x_n^{\alpha_n})$ for $\alpha\in \{\pm1\}^n$ and note that $v_{(k^n)}(t)= \prod_{i=1}^n \frac{1-t^i}{1-t}$.
We rewrite the sum over $S_n$ by factoring out terms that are invariant by permutations. The sum becomes
\begin{multline*}
    \prod_{i=1}^n\x_i^{k}\frac{(1-s_1x_i)(1-s_2x_i)}{1-x_i^2}\prod_{1\leq i<j\leq n}\frac{1-tx_ix_j}{1-x_ix_j}\sum_{\sigma\in S_n}\sigma.\!\left(\prod_{1\leq i<j\leq n}\frac{tx_i-x_j}{x_i-x_j}\right)\\
    =
    \prod_{i=1}^n\x_i^{k}\frac{(1-s_1x_i)(1-s_2x_i)}{1-x_i^2}\prod_{1\leq i<j\leq n}\frac{1-tx_ix_j}{1-x_ix_j}v_{(k^n)}(t) \,,
\end{multline*}
where we have used~\cite[p.~207, (1.4)]{Macdonald1995}. The right-most term cancels the prefactor $1/v_{(k^n)}(t)$ and we obtain the following formula for the Hall--Littlewood polynomial:
\[
    P_{(k^n)}^{\mathrm{BC}_n}(X;t,s_1,s_2) = \sum_{\alpha \in \{\pm 1\}^n}\left[\prod_{i=1}^n\x_i^{k}\frac{(1-s_1x_i)(1-s_2x_i)}{1-x_i^2}\prod_{1\leq i<j\leq n}\frac{1-tx_ix_j}{1-x_ix_j} \right]_{X=X^\alpha}.
\]

Now observe that for any $\alpha\in \{\pm1\}^n$ we have
\begin{equation*}
    \left[\prod_{1\leq i<j\leq n}\frac{1}{(x_i-x_j)(1-x_ix_j)}\right]_{X=X^\alpha}=\frac{\prod_{i\in\alpha^-}x_i^{2(n-1)}}{\prod_{1\leq i<j\leq n}(x_i-x_j)(1-x_ix_j)}\,,
\end{equation*}
where $\alpha^-=\{i\in [n]\mid \alpha_i=-1\}$. Dividing and multiplying the Hall--Littlewood polynomial by this expression gives
\begin{align*}
    P_{k^n}^{\mathrm{BC}_n}(X;&t,s_1,s_2)=\frac{1}{\prod_{1\leq i<j\leq n}(x_i-x_j)(1-x_ix_j)}\\
    &\times \sum_{\alpha \in \{\pm 1\}^n}
    \prod_{i\in\alpha^-}x_i^{2(n-1)}
    \left[\prod_{i=1}^n\x_i^{k}\frac{(1-s_1x_i)(1-s_2x_i)}{1-x_i^2}\prod_{1\leq i<j\leq n}
    {(1-tx_ix_j)}{(x_i-x_j)}\right]_{X
    =X^\alpha}\\& =
    \frac{\prod_{i=1}^n x_i^{n-1}}{\prod_{1\leq i<j\leq n}(x_i-x_j)(1-x_ix_j)}
    \\&\times\sum_{\alpha \in \{\pm 1\}^n}
    \left[\prod_{i=1}^n\x_i^{k+n-1}
    \frac{(1-s_1x_i)(1-s_2x_i)}{1-x_i^2}
    \prod_{1\leq i<j\leq n}(1-tx_ix_j)(x_i-x_j)\right]_{X=X^\alpha}.
\end{align*}
In the last equality we have used the following simplification,
\begin{multline*}
    \prod_{i\in\alpha^-} x_i^{2(n-1)}
\left[
\prod_{i=1}^n\x_i^{k}
\right]_{X=X^\alpha}
= \left(\prod_{i=1}^nx_i^{n-1}\prod_{i=1}^nx_i^{-n+1}\right)\prod_{i\in\alpha^-} x_i^{n-1}\prod_{i\in\alpha^-} x_i^{n-1}\left[
\prod_{i=1}^n\x_i^{k}
\right]_{X=X^\alpha}\\= \prod_{i=1}^nx_i^{n-1}\left(\prod_{i=1}^nx_i^{-n+1}\prod_{i\in\alpha^-} x_i^{n-1}\right)\prod_{i\in\alpha^-} x_i^{n-1}\left[
\prod_{i=1}^n\x_i^{k}
\right]_{X=X^\alpha}\\=\prod_{i=1}^nx_i^{n-1}\prod_{i\in\alpha^+} x_i^{-n+1}\prod_{i\in\alpha^-} x_i^{n-1}\left[
\prod_{i=1}^n\x_i^{k}
\right]_{X=X^\alpha}=\prod_{i=1}^n x_i^{n-1} \left[\prod_{i=1}^n\x_i^{k+n-1}\right]_{X=X^\alpha},
\end{multline*}
where we first multiply the equation with $1=\prod_{i=1}^nx_i^{n-1}\prod_{i=1}^nx_i^{-n+1}$, then regroup terms and use that $\alpha^{-}$ and $\alpha^+$ partition $[n]=\alpha^-\sqcup \alpha^+$ to get that $\prod_{i=1}^nx_i^{-n+1}\prod_{i\in\alpha^-} x_i^{n-1}=\prod_{i\in\alpha^+} x_i^{-n+1}$ and then use it once again to distribute the two factors over the expression in the square brackets in the last equality.

Setting $x_i=x_i\sqrt{t}$ in the Weyl denominator of type \ref{eq: Denom type D} we obtain
\begin{equation}
   \det_{1\leq i,j\leq n}\left(
    (x_i\sqrt{t})^{j-1}+  (\x_i/\sqrt{t})^{j-1}\right)= 2 (\sqrt{t})^{-\binom{n}{2}}
    \prod_{i=1}^n\x_i^{n-1}\prod_{1\leq i<j\leq n}(1-tx_ix_j)(x_i-x_j).
\end{equation}
We use this factorisation to find that $P_{k^n}^{\mathrm{BC}_n}(X;t,s_1,s_2)$ equals
\begin{multline*}
\frac{\sqrt{t}^{\binom{n}{2}}}{2}
    \prod_{i=1}^nx_i^{n-1}\prod_{1\leq i<j\leq n}\frac{1}{(1-x_ix_j)(x_i-x_j)}\\\times
\sum_{\alpha \in \{\pm 1\}^n}
    \left[\prod_{i=1}^n\x_i^{k}
    \frac{(1-s_1x_i)(1-s_2x_i)}{1-x_i^2}
    \det_{1\le i,j\le n}\left(
    (x_i\sqrt{t})^{j-1}+  (\x_i/\sqrt{t})^{j-1}\right)\right]_{X=X^\alpha}.
\end{multline*}
By linearity of the determinant this is equal to 
\begin{multline*}
 \frac{\sqrt{t}^{\binom{n}{2}}}{2}
    \prod_{i=1}^nx_i^{n-1}\prod_{1\leq i<j\leq n}\frac{1}{(1-x_ix_j)(x_i-x_j)}\\\times
    \det_{1\le i,j\le n}\left(
    \frac{(1-s_1x_i)(1-s_2x_i)}{1-x_i^2}
    (x_i^{j-1-k}(\sqrt{t})^{j-1}+  
    \x_i^{j-1+k}(\sqrt{t})^{-j+1})\right.
    \\\left.-
    \frac{(x_i-s_1)(x_i-s_2)}{1-x_i^2}
    (\x_i^{j-1-k}(\sqrt{t})^{j-1}+  
    x_i^{j-1+k}(\sqrt{t})^{-j+1})\right)
\end{multline*}
and thus we are finished.
\end{proof}

\subsection{A determinant trick}\label{subsec: pulling out of det}
The following lemma shows that $t$ can be factored out of the determinant in Corollaries \ref{cor-b}, \ref{cor-type-c} and \ref{cor-type-bc}.

\begin{lem}
\label{lem:powery}
Let $A = (a_{ij}(x))_{1\le i,j\le k}$ be a matrix of univariate functions.
Let $R\subseteq\Z^k$ be a finite $S_k$-invariant subset of $\Z^k$.
Then,
   \[
    \sum_{y \in R} \det_{1 \le i,j \le k}
    (
    t^{jy_i} a_{i,j}(y_i)
    ) =  \sum_{y \in R} t^{\sum_{j=1}^k j y_j} \det_{1 \le i,j \le k} (
    a_{i,j}(y_j)).
    \]
\end{lem}
\begin{proof}
    Expand the left-hand side according to the Leibniz formula of the determinant,
    \[
     \sum_{y \in R} \sum_{\sigma \in S_k}  \mathrm{sgn}(\sigma)  \left( \prod_{j=1}^k
    t^{jy_{\sigma(j)}} a_{\sigma(j),j}(y_{\sigma(j)})
    \right).
    \]
Let $\mathrm{sort}(R)=\{y\in R\mid y_1\geq \dots\geq y_k\}$ and define an action of $S_k$ on $R$ by $(\mu\cdot y)_i=y_{\mu(i)}$. Observe that any $y\in R$ can be written as $y=\mu\cdot x$ for some $x\in \mathrm{sort}(R) $ and some $\mu\in S_k$. Conversely from one $x\in \mathrm{sort}(R) $ we can obtain the same $y\in R$ multiple times by acting with different $\mu$'s. The number of times this can happen is given by the size of the stabilizer of $x$ under this action and denoted by $m(x)=\vert\mathbf{Stab}(x)\vert$. Using this observation, we can rewrite the above sum as
    \[
     \sum_{x \in \mathrm{sort}(R)}\frac{1}{m(x)}\sum_{\mu\in S_k} \sum_{\sigma \in S_k}  \mathrm{sgn}(\sigma)  \left( \prod_{j=1}^k
    t^{jx_{\mu(\sigma(j))}} a_{\sigma(j),j}(x_{\mu(\sigma(j))})
    \right).
    \]
Summing over all $\rho=\mu\circ \sigma$ instead of $\mu\in S_k$, we obtain
    \[
     \sum_{x \in \mathrm{sort}(R)}\frac{1}{m(x)}\sum_{\rho\in S_k} \sum_{\sigma \in S_k}  \mathrm{sgn}(\sigma)  \left( \prod_{j=1}^k
    t^{jx_{\rho(j)}} a_{\sigma(j),j}(x_{\rho(j)})
    \right).
    \]
The inner sum now clearly is a determinant and this is equal to
\begin{multline*}
     \sum_{x \in \mathrm{sort}(R)}\frac{1}{m(x)}\sum_{\rho\in S_k} \prod_{j=1}^k
    t^{jx_{\rho(j)}} \det_{1\leq i,j\leq k}\left( a_{i,j}(x_{\rho(j)})
    \right)
    =\sum_{y\in R}\prod_{j=1}^k
    t^{jy_j} \det_{1\leq i,j\leq k}\left( a_{i,j}(y_j)
    \right)\\=\sum_{y \in R} t^{\sum_{j=1}^k j y_j} \det_{1 \le i,j \le k} (
    a_{i,j}(y_j)),
\end{multline*}
    where the first equality again follows from the definition of the group action. This proves the lemma.
\end{proof}

\subsection{Determinant expansions}\label{sec:noodle-doodle}

In this section, we establish six determinant expansion formulae that will
be used in the proofs of Theorems~\ref{thm-b-char}, \ref{main-type-c}, and
\ref{thm-type-bc-0-char-all}. Since the proofs follow essentially the same
strategy and differ only in the labelling of the underlying graphs, we
first prove a general expansion and then give the six specialisations.
Throughout the graph argument, $n,m,k$ are integers with $k\geq1$ and
$0\leq m\leq n$. The graphs with $m=0$ are allowed. All matrices and functions take values in a field of
characteristic zero. We use $[a,b]=\{a,a+1,\ldots,b\}$ when $b-a$ is a
nonnegative integer, and $[a,b]=\varnothing$ when $b<a$. Products over an
empty set are~$1$.

In its simplest form, the graph has vertex set $[0,k+m]$ and $m$ labelled
edges: $k+j$ is adjacent to $k-j$ for
$j=1,\ldots,\min(m,k-1)$, while $k+j$ is adjacent to $-k+j$ for
$j=k,\ldots,m$. The edges arising in this way are labelled with~$j$.
Note that the second type of edges only exists if $k\leq m$.
The graph is denoted by $\Lambda_{m,k,0,0}$ and illustrated in Figure~\ref{fig:example of Lambda graph}. Its connected components are
paths, each containing exactly one vertex in $[0,k]$, which we call its
\emph{center}.
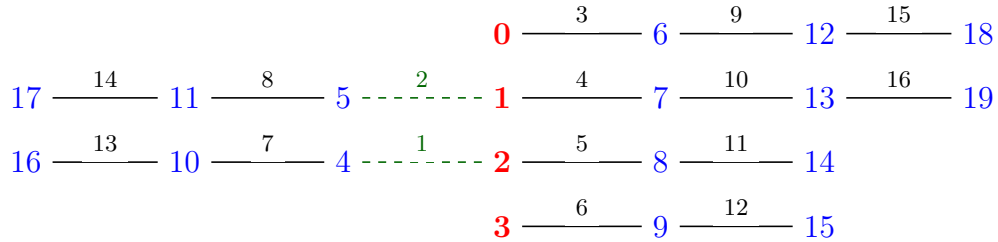
\begin{figure}[h]
    \centering
      \begin{tikzpicture}[- ,auto ,node distance =0.85 cm and 2.1cm ,on grid ,
semithick ,
state/.style ={text=blue , minimum width =0.25 cm},
center/.style = {red, minimum width =0.25 cm}]
\node[center] (A) {$\mathbf{0}$};
\node[center] (B) [below =of A] {$\mathbf{1}$};
\node[center] (C) [below =of B] {$\mathbf{2}$};
\node[center] (D) [below =of C] {$\mathbf{3}$};
\node[state] (E) [right =of A] {$6$};
\node[state] (F) [left =of B] {$5$};
\node[state] (G) [left =of C] {$4$};
\node[state] (H) [right =of B] {$7$};
\node[state] (I) [right =of C] {$8$};
\node[state] (J) [right =of D] {$9$};
\node[state] (L) [left =of F] {$11$};
\node[state] (M) [left =of G] {$10$};
\node[state] (N) [right =of E] {$12$};
\node[state] (O) [right =of H] {$13$};
\node[state] (P) [right =of I] {$14$};
\node[state] (Q) [right =of J] {$15$};
\node[state] (R) [left =of M] {$16$};
\node[state] (S) [left =of L] {$17$};
\node[state] (T) [right =of N] {$18$};
\node[state] (U) [right =of O] {$19$};

\draw (A) -- (E) node [midway, fill=white, anchor=south] {\scriptsize $3$};
\draw[color=green!40!black, dashed, anchor=south] (B) -- (F) node [midway, fill=white] {\scriptsize $2$};
\draw[color=green!40!black, dashed, anchor=south] (C) -- (G) node [midway, fill=white] {\scriptsize $1$};
\draw (B) -- (H) node [midway, fill=white, anchor=south] {\scriptsize $4$};
\draw (C) -- (I) node [midway, fill=white, anchor=south] {\scriptsize $5$};
\draw (D) -- (J) node [midway, fill=white, anchor=south] {\scriptsize $6$};
\draw (F) -- (L) node [midway, fill=white, anchor=south] {\scriptsize $8$};
\draw (G) -- (M) node [midway, fill=white, anchor=south] {\scriptsize $7$};
\draw (E) -- (N) node [midway, fill=white, anchor=south] {\scriptsize $9$};
\draw (H) -- (O) node [midway, fill=white, anchor=south] {\scriptsize $10$};
\draw (I) -- (P) node [midway, fill=white, anchor=south] {\scriptsize $11$};
\draw (J) -- (Q) node [midway, fill=white, anchor=south] {\scriptsize $12$};
\draw (M) -- (R) node [midway, fill=white, anchor=south] {\scriptsize $13$};
\draw (L) -- (S) node [midway, fill=white, anchor=south] {\scriptsize $14$};
\draw (N) -- (T) node [midway, fill=white, anchor=south] {\scriptsize $15$};
\draw (O) -- (U) node [midway, fill=white, anchor=south] {\scriptsize $16$};
\end{tikzpicture} 
    \caption{The graph $\Lambda_{16,3,0,0}$. To help the reader, type-1 edges are green and dashed, and centers of the connected components are in red and bold.}
    \label{fig:example of Lambda graph}
\end{figure}

In its general form, the graph has two additional parameters
$h^+\in\{0,\frac12\}$ and $h^-\in\{0,-\frac12\}$.

\begin{defn}[The labelled graph]\label{def: labelled graph}
Set $\bar k=k+\lfloor h^++h^-\rfloor$ and $K=2(k+h^+)$.
The graph $\Lambda_{m,k,h^+,h^-}$ has vertex set
\[
 V=[0,\bar k+m]+|h^++h^-|
\]
and edge labels $E=[1,m]$. The edge labelled $j$ joins
\[
 k+j+h^++h^-\qquad\text{and}\qquad |k-j+h^+-h^-|.
\]
It is called a \emph{type-$1$ edge} if
$j<k+\lceil h^+-h^-\rceil$, and a \emph{type-$2$ edge} otherwise.
The vertices
\[
 c_j=j+|h^++h^-|,\qquad 0\leq j\leq\bar k,
\]
are the centers.
\end{defn}
    The graphs have a very simple structure and in order to see this, we first consider the subgraph of $\Lambda_{m,k,h^{+},h^{-}}$ consisting only of the type-$2$ edges. This is a collection of $K$ disjoint paths, the vertices of which are comprised by integers that have the same remainder modulo $K$.
In other words: the $j$-th path connects vertices in the order 
$${\color{orange} j}+{\color{orange} |h^{+}+h^{-}|},~
{\color{violet} K}+{\color{orange} j}+{\color{orange} |h^{+}+h^{-}|},~
{\color{violet} 2K}+{\color{orange} j}+{\color{orange} |h^{+}+h^{-}|},~
{\color{violet} 3K}+{\color{orange} j}+{\color{orange} |h^{+}+h^{-}|},\ldots$$ 
with edge labels
$${\color{orange} j+k}+{\color{orange} \lceil h^+-h^- \rceil},~
{\color{violet} K}+{\color{orange} j+k}+ {\color{orange} \lceil h^+-h^- \rceil},~
{\color{violet} 2K}+{\color{orange} j+k}+ {\color{orange} \lceil h^+-h^- \rceil},\ldots$$
for 
$0 \le j \le K-1$, that is the edge 
$$\{{{\color{violet} K l} +\color{orange} j}+  {\color{orange} |h^{+}+h^{-}|},~{\color{violet} K(l+1)}+{\color{orange} j}+
{\color{orange} |h^{+}+h^{-}|} \}$$ 
is labelled with ${\color{violet} Kl}+{\color{orange}j+ k} + {\color{orange} \lceil h^+-h^- \rceil}$. 

The role of the type-$1$ edges is as follows. First, note that, if $m\leq k+\lceil h^{+}-h^{-} \rceil-1$, there only exists 
type-$1$ edges. Otherwise, the role of the type-$1$ edges is that they connect two of the paths comprised of 
type-$2$ edges. Concretely, using the numbering of paths from above, we have that the $j$-th path is connected to the 
$(K-j-2|h^++h^-|)$-th path, for $j=1- \lceil |h^{+}+h^{-}| \rceil,\ldots,k-1+\lfloor h^{+}-h^{-} \rfloor$, with 
connecting edge 
$$
\{j+|h^{+}+h^{-}|,~K-|h^{+}+h^{-}|-j\},
$$
the label of which is $k-j+ \lfloor h^{+}-h^{-} \rfloor$. All type-$1$ edges are adjacent to a center.

From this, it is not difficult to check that $\Lambda_{m,k,h^{+},h^{-}}$ consists of $\bar k+1$ components, each of which contains precisely one center $c_j$, and this $j$ coincides with the numbering of components from above. We draw the centers $c_0,\ldots,c_{\bar k}$ centrally justified, from top to bottom. The respective components are drawn horizontally, with the type-$1$ edge of the centers placed to the left of the respective center. Centers are always incident with a type-$1$ edge, except for $c_0$ in the cases $(h^{+},h^{-})=(0,0), \left(\frac{1}{2},-\frac{1}{2}\right)$ and $c_{\bar k}$ in the cases
$(h^{+},h^{-})=(0,0),\left(\frac{1}{2},0\right)$. 
The four graphs for $m=8$ and $k=3$ are shown in
Figure~\ref{fig:graphs for various parameters}.

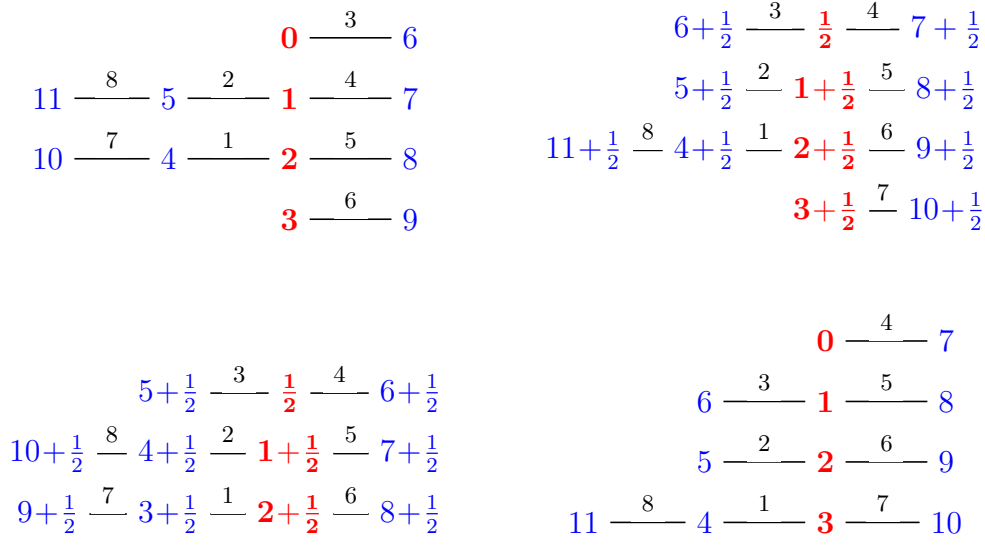
\begin{figure}[h]
        \centering
        \begin{tabular}{ccc}
            \begin{tikzpicture}[- ,auto ,node distance =0.8 cm and 1.6cm ,on grid ,
semithick ,
state/.style ={text=blue , minimum width =0.25 cm}]
\node[state, red] (A) {$\mathbf{0}$};
\node[state, red] (B) [below =of A] {$\mathbf{1}$};
\node[state, red] (C) [below =of B] {$\mathbf{2}$};
\node[state, red] (D) [below =of C] {$\mathbf{3}$};
\node[state] (E) [right =of A] {$6$};
\node[state] (F) [left =of B] {$5$};
\node[state] (G) [left =of C] {$4$};
\node[state] (H) [right =of B] {$7$};
\node[state] (I) [right =of C] {$8$};
\node[state] (J) [right =of D] {$9$};
\node[state] (L) [left =of F] {$11$};
\node[state] (M) [left =of G] {$10$};

\draw (A) -- (E) node [midway, fill=white, anchor=south] {\scriptsize $3$};
\draw (B) -- (F) node [midway, fill=white, anchor=south] {\scriptsize $2$};
\draw (C) -- (G) node [midway, fill=white, anchor=south] {\scriptsize $1$};
\draw (B) -- (H) node [midway, fill=white, anchor=south] {\scriptsize $4$};
\draw (C) -- (I) node [midway, fill=white, anchor=south] {\scriptsize $5$};
\draw (D) -- (J) node [midway, fill=white, anchor=south] {\scriptsize $6$};
\draw (F) -- (L) node [midway, fill=white, anchor=south] {\scriptsize $8$};
\draw (G) -- (M) node [midway, fill=white, anchor=south] {\scriptsize $7$};
\end{tikzpicture}  & 
\hspace{0.4cm} &
\begin{tikzpicture}[- ,auto ,node distance =0.8 cm and 1.6cm ,on grid ,
semithick, state/.style ={text=blue , minimum width =0.25 cm}]
\node[state, red] (A) {$\mathbf{\frac{1}{2}}$};
\node[state] (a) [right =of A] {$7+\frac{1}{2}$};
\node[state, red] (B) [below =of A] {$\mathbf{1\!+\!\frac{1}{2}}$};
\node[state, red] (C) [below =of B] {$\mathbf{2\!+\!\frac{1}{2}}$};
\node[state, red] (D) [below =of C] {$\mathbf{3\!+\!\frac{1}{2}}$};
\node[state] (E) [left =of A] {$6\!+\!\frac{1}{2}$};
\node[state] (F) [left =of B] {$5\!+\!\frac{1}{2}$};
\node[state] (G) [left =of C] {$4\!+\!\frac{1}{2}$};
\node[state] (H) [right =of B] {$8\!+\!\frac{1}{2}$};
\node[state] (I) [right =of C] {$9\!+\!\frac{1}{2}$};
\node[state] (J) [right =of D] {$10\!+\!\frac{1}{2}$};
\node[state] (K) [left =of G] {$11\!+\!\frac{1}{2}$};
\draw (A) -- (a) node [midway, fill=white, anchor=south] {\scriptsize $4$};
\draw (A) -- (E) node [midway, fill=white, anchor=south] {\scriptsize $3$};
\draw (B) -- (F) node [midway, fill=white, anchor=south] {\scriptsize $2$};
\draw (C) -- (G) node [midway, fill=white, anchor=south] {\scriptsize $1$};
\draw (B) -- (H) node [midway, fill=white, anchor=south] {\scriptsize $5$};
\draw (C) -- (I) node [midway, fill=white, anchor=south] {\scriptsize $6$};
\draw (D) -- (J) node [midway, fill=white, anchor=south] {\scriptsize $7$};
\draw (K) -- (G) node [midway, fill=white, anchor=south] {\scriptsize $8$};
\end{tikzpicture}\\[0.7cm]
\begin{tikzpicture}[- ,auto ,node distance =0.8 cm and 1.6cm ,on grid ,
semithick, state/.style ={text=blue , minimum width =0.25 cm}]
\node[state,red] (A) {$\mathbf{\frac{1}{2}}$};
\node[state] (a) [left =of A] {$5\!+\!\frac{1}{2}$};
\node[state,red] (B) [below =of A] {$\mathbf{1\!+\!\frac{1}{2}}$};
\node[state,red] (C) [below =of B] {$\mathbf{2\!+\!\frac{1}{2}}$};
\node[state] (E) [right =of A] {$6\!+\!\frac{1}{2}$};
\node[state] (F) [right =of B] {$7\!+\!\frac{1}{2}$};
\node[state] (G) [right =of C] {$8\!+\!\frac{1}{2}$};
\node[state] (H) [left =of B] {$4\!+\!\frac{1}{2}$};
\node[state] (I) [left =of C] {$3\!+\!\frac{1}{2}$};
\node[state] (N) [left =of H] {$10\!+\!\frac{1}{2}$};
\node[state] (O) [left =of I] {$9\!+\!\frac{1}{2}$};

\draw (A) -- (a) node [midway, fill=white, anchor=south] {\scriptsize $3$};
\draw (A) -- (E) node [midway, fill=white, anchor=south] {\scriptsize $4$};
\draw (B) -- (F) node [midway, fill=white, anchor=south] {\scriptsize $5$};
\draw (C) -- (G) node [midway, fill=white, anchor=south] {\scriptsize $6$};
\draw (B) -- (H) node [midway, fill=white, anchor=south] {\scriptsize $2$};
\draw (C) -- (I) node [midway, fill=white, anchor=south] {\scriptsize $1$};
\draw (H) -- (N) node [midway, fill=white, anchor=south] {\scriptsize $8$};
\draw (I) -- (O) node [midway, fill=white, anchor=south] {\scriptsize $7$};
\end{tikzpicture} &
\hspace{0.4cm}
&
\begin{tikzpicture}[- ,auto ,node distance =0.8 cm and 1.6cm ,on grid ,
semithick ,state/.style ={text=blue , minimum width =0.25 cm}]
\node[state, red] (A) {$\mathbf{0}$};
\node[state, red] (B) [below =of A] {$\mathbf{1}$};
\node[state, red] (C) [below =of B] {$\mathbf{2}$};
\node[state, red] (D) [below =of C] {$\mathbf{3}$};
\node[state] (a) [right =of A] {$7$};
\node[state] (E) [left =of B] {$6$};
\node[state] (F) [left =of C] {$5$};
\node[state] (G) [left =of D] {$4$};
\node[state] (H) [right =of B] {$8$};
\node[state] (I) [right =of C] {$9$};
\node[state] (J) [right =of D] {$10$};
\node[state] (M) [left =of G] {$11$};

\draw (A) -- (a) node [midway, fill=white, anchor=south] {\scriptsize $4$};
\draw (B) -- (E) node [midway, fill=white, anchor=south] {\scriptsize $3$};
\draw (C) -- (F) node [midway, fill=white, anchor=south] {\scriptsize $2$};
\draw (D) -- (G) node [midway, fill=white, anchor=south] {\scriptsize $1$};
\draw (B) -- (H) node [midway, fill=white, anchor=south] {\scriptsize $5$};
\draw (C) -- (I) node [midway, fill=white, anchor=south] {\scriptsize $6$};
\draw (D) -- (J) node [midway, fill=white, anchor=south] {\scriptsize $7$};
\draw (G) -- (M) node [midway, fill=white, anchor=south] {\scriptsize $8$};
\end{tikzpicture}     
        \end{tabular}  
        \caption{The graphs $\Lambda_{8,3,h^{+},h^{-}}$ for 
        $(h^{+},h^{-})=\left(0,0\right),\left(\frac{1}{2},0\right), \left(0,-\frac{1}{2}\right),\left(\frac{1}{2},-\frac{1}{2}\right)$ with the centers in bold and red. The parameter pairs are listed from left to right, first in the top row and then in the bottom row. }
        \label{fig:graphs for various parameters}
\end{figure}

To formulate the determinant expansion, we consider orientations of the
edges in which all edges of each component are directed away from one
chosen vertex, its \emph{root}. We denote the set of these orientations by
$\mathcal A(\Lambda_{m,k,h^+,h^-})$. An orientation in this set is uniquely
determined by the roots $z_0,\ldots,z_{\bar k}$, where $z_j$ lies in the
component containing~$c_j$. This includes an isolated vertex, which is
necessarily its own root.

For $A\in\mathcal A(\Lambda_{m,k,h^+,h^-})$, let $E_A(c_j,z_j)$ be the set
of edges between $c_j$ and $z_j$, and let $E_A^2(c_j,z_j)$ be its subset
of type-$2$ edges. We write
\[
 E_A=\bigcup_{j=0}^{\bar k}E_A(c_j,z_j),\qquad
 E_A^2=\bigcup_{j=0}^{\bar k}E_A^2(c_j,z_j).
\]
For any orientation $O$, let $t_O(j)$ and $s_O(j)$ be the target and source
of edge~$j$. The \emph{natural orientation} $\mathcal N$ is the one with
roots $z_j=c_j$. Its edges are directed towards their larger endpoints.
In the following determinants, the block $M$ is placed first, and the
remaining columns are ordered by increasing edge label or vertex label,
as indicated by the index below the determinant.

\begin{lem}[Expansion over orientations]\label{lem: Noodle-Doodle expansion first form}
Let $M$ be an $n\times(n-m)$ matrix, let $f_1,\ldots,f_n$ be functions
on $V$, and let $p_s,p_t$ be functions on $E$. Then
\begin{multline}\label{eq: to show pure graph}
 \det_{\substack{1\leq i\leq n\\ j\in E}}
 \bigl(\matrixCols{M}{p_t(j)f_i(t_{\mathcal N}(j))
                         +p_s(j)f_i(s_{\mathcal N}(j))}\bigr)
 \\
 =\sum_{A\in\mathcal A(\Lambda_{m,k,h^+,h^-})}
   \prod_{j\in E\setminus E_A}p_t(j)
   \prod_{j\in E_A}p_s(j)
   \det_{\substack{1\leq i\leq n\\ j\in E}}
     \bigl(\matrixCols{M}{f_i(t_A(j))}\bigr).
\end{multline}
\end{lem}
\begin{proof}
Using multilinearity in columns $n-m+1,\ldots,n$, we expand the
left-hand side by choosing one endpoint of each edge. Orienting each edge
towards the chosen endpoint identifies the resulting terms with
orientations of the graph.

By the antisymmetry of the determinant, an orientation with a vertex of
indegree two gives a zero summand, since two columns are equal. In every
remaining orientation, a component with $r$ vertices has $r-1$ vertices
of indegree one and precisely one of indegree zero. Its edges are directed
away from the latter vertex. Hence it suffices to sum over
$\mathcal A(\Lambda_{m,k,h^+,h^-})$.

Changing the root from $c_j$ to $z_j$ reverses precisely the edges between
these two vertices. Thus the edges in $E_A$ contribute $p_s$, and all other
edges contribute $p_t$, proving~\eqref{eq: to show pure graph}.
\end{proof}

We next parametrise the roots. Each root can be written uniquely as
\begin{equation}\label{eq: parametrisation of roots}
 z_j=z(y_j)=\alpha_j(Ky_j+c_j),\qquad
 \alpha_j=\begin{cases}1,&y_j\geq0,\\-1,&y_j<0,\end{cases}
\end{equation}
\nopagebreak[4]
provided that
\begin{equation}\label{eq: admissible roots}
 \begin{gathered}
 (y_0,\ldots,y_{\bar k})\in\Z^{\bar k+1},\qquad
 z(y_j)\in V\quad(0\leq j\leq\bar k),\\
 y_j\geq0\quad\text{if }c_j\equiv-c_j\pmod K.
 \end{gathered}
\end{equation}
Indeed, on the two sides of a center the successive vertices are
$c_j+K,c_j+2K,\ldots$ and $K-c_j,2K-c_j,\ldots$, respectively, truncated
to~$V$. The second side is absent if $c_j\equiv-c_j\pmod K$; in that case
allowing negative $y_j$ would count vertices twice. These observations also
show that~\eqref{eq: admissible roots} gives precisely one choice of a root
in each component. In particular, the roots are distinct and the set of
admissible tuples is finite.

\begin{lem}\label{lem: EA and tA(j) are given by}
For an orientation $A\in \mathcal A(\Lambda_{m,k,h^+,h^-})$ determined by the roots $z_0,\dots,z_{\bar k}$ parametrised as in~\eqref{eq: parametrisation of roots}
and~\eqref{eq: admissible roots}, we have
\begin{equation}\label{eq: EA is given by}
 E_A=\bigcup_{i=0}^{\bar k}\Big\{K\ell+k+\alpha_i i+\lceil h^+-h^-\rceil
       -[y_i<0]\lceil|h^++h^-|\rceil~:~     0\leq\ell<|y_i|\Big\},
\end{equation}
where $\ell$ is an integer. Moreover,
\begin{equation}\label{eq: tA(j) is equal to}
 t_A(j)=\begin{cases}
 |k-j+h^+-h^-|,&j\in E_A,\\
 k+j+h^++h^-,&j\notin E_A,
 \end{cases}
 \qquad
 |E_A^2|=\sum_{i=0}^{\bar k}\bigl(|y_i|-[y_i<0]\bigr).
\end{equation}
\end{lem}
\begin{proof}
    We claim that \begin{equation*}
        E_A(c_i,z_i)=\left\{2(k+h^+)\ell+k+\alpha_i i+\lceil h^+-h^-\rceil-[y_i< 0]\lceil |h^++h^-|\rceil~:~
    \ell=0,\dots,|y_i|-1\right\},
    \end{equation*}
    from which \eqref{eq: EA is given by} immediately follows. 
    
    Note that if $z_i$ is located weakly to the right side of $c_i$ then the $|y_i|$ edges between $c_i$ and $z_i$ are labelled with $$2(k+h^+)\ell+k+i+\lceil h^+-h^-\rceil,$$ for $\ell=0, \dots, |y_i|-1$ as claimed. 
    
Similarly, if $z_i$ is located strictly to the left side of $c_i$ then the $|y_i|$ edges between $c_i$ and $z_i$ consist of the one type-$1$ edge labelled by $$k-i+\lfloor h^+-h^-\rfloor$$ and the $|y_i|-1$ type-$2$ edges labelled with $$2(k+h^+)\ell+k-i^*+\lceil h^+-h^-\rceil$$ for $\ell=1,\dots,|y_i|-1$. Here we denote by $i^*=i$ if $(h^+,h^-)=(0,0),\left(\frac{1}{2},-\frac{1}{2}\right)$ and $i^*=i+1$ if $(h^+,h^-)=\left(\frac{1}{2},0\right),\left(0,-\frac{1}{2}\right)$.
Note that $\lceil |h^++h^-|\rceil=1$ if and only if $(h^+,h^-)=(\frac{1}{2},0),(0,-\frac{1}{2})$ and this proves the first claim. Note that this also immediately proves the formula for $|E_A^2|$ as it shows that $E_A$ contains $|y_i|$ type-$2$ edges if $y_i\geq 0$ and precisely one type-$1$ edge and $|y_i|-1$ type $2$ edges if $y_i<0$. 

An edge in $E_A$ is directed towards the center and therefore towards its
smaller endpoint. Every other edge is directed towards its larger
endpoint. This proves the formula for~$t_A(j)$.

\end{proof}

In our applications, the determinants on the right-hand side of the
expansion are to be recognised as character numerators. We must therefore
reorder their columns so that the vertex labels are increasing. The
following lemma gives the required sign.

\begin{lem}[Column order]\label{lem: column order}
For an orientation $A\in\mathcal A(\Lambda_{m,k,h^+,h^-})$ with roots $z_0,\ldots,z_{\bar k}$, set
\[
 b_j=\bigl|\{i:i>j\text{ and }z_i>z_j\}\bigr|.
\]
Then
\begin{equation}\label{eq: sign of column order}
 (-1)^{\inv(t_A(1),\ldots,t_A(m))}=(-1)^{\binom{\bar k+1}{2}
       +\sum_{j=0}^{\bar k}\bigl((K-1)y_j
            +2|h^++h^-|[y_j<0]+b_j\bigr)},
\end{equation}
where $\inv(a_1,\dots,a_n)\coloneqq|\{(i,j)\in[1,n]^2\mid i>j\textrm{ and }a_i<a_j\}|$ is the number of \emph{inversions} of the array $(a_1,\dots,a_n)$.
\end{lem}
\begin{proof}
We first count inversions in a list containing all
vertices, and then subtract the inversions involving roots.

Consider $(z_0,\ldots,z_{\bar k}\mid t_A(1),\ldots,t_A(m))$.
The vertical bar merely separates the roots from the targets. The roots
remain in component order, and the targets remain in edge-label order.
Every vertex occurs exactly once: roots have indegree zero, and every
other vertex is the target of precisely one edge.

For the natural orientation, this list is $ (c_0,\ldots,c_{\bar k}\mid
 t_{\mathcal N}(1),\ldots,t_{\mathcal N}(m))$,
which is the increasing list of all vertices. Its inversion number is
therefore zero.

Now move a root along the path from its center to its prescribed position.
Suppose one step moves the current root from $u$ to an adjacent vertex $v$,
across edge $\ell$. Before this move, the edge points from $u$ to $v$, so
the root entry is $u$ and the entry corresponding to edge $\ell$ is $v$.
Afterwards, the root entry is $v$, while the reversed edge has target $u$.
Thus this move exchanges exactly two entries of the list.
All other entries remain unchanged.

Exchanging two entries reverses the parity of the inversion number, even
when their positions are not adjacent. Consequently, after moving all roots,
\begin{equation}\label{eq: sign of augmented list}
 \inv(z_0,\ldots,z_{\bar k},t_A(1),\ldots,t_A(m))
 \equiv |E_A|
 =\sum_{j=0}^{\bar k}|y_j|\pmod2.
\end{equation}
This is a parity statement: one move need not change the inversion number
by exactly one.

We now separate the inversions of this full list into three kinds:
\begin{itemize}
 \item inversions between two roots;
 \item inversions between a root and a target;
 \item inversions between two targets.
\end{itemize}
The last kind is precisely what the lemma seeks to count.

For the first kind, recall that $b_j=\bigl|\{i:i>j,\ z_i>z_j\}\bigr|$.
Thus $b_j$ counts the non-inversions whose first entry is $z_j$
in the root list. Since the roots are distinct, each of the
$\binom{\bar k+1}{2}$ pairs is either an inversion or a non-inversion.
Hence
\[
 \inv(z_0,\ldots,z_{\bar k})
 =\binom{\bar k+1}{2}-\sum_{j=0}^{\bar k}b_j.
\]

For the second kind, every root appears before every target. Therefore a
root $z_j$ and a target form an inversion exactly when the target is smaller
than $z_j$.
Since the vertices are consecutive with smallest vertex $|h^++h^-|$,
there are exactly $z_j-|h^++h^-|$
vertices smaller than $z_j$. Some are roots, so the number of targets
smaller than $z_j$ is $ z_j-|h^++h^-|-\bigl|\{i:z_i<z_j\}\bigr|$. Here, the last set includes all smaller roots, regardless of their
indices as these are never targets of an edge.
When we sum this expression over the components $j$, each unordered pair of roots contributes exactly once to the subtracted term: its smaller member is counted for its larger member.
Thus $ \sum_{j=0}^{\bar k}\bigl|\{i:z_i<z_j\}\bigr|
 =\binom{\bar k+1}{2}.$
The number of root--target inversions is consequently
\[
 \sum_{j=0}^{\bar k}(z_j-|h^++h^-|)-\binom{\bar k+1}{2}.
\]

Combining the three contributions gives the exact equality
\begin{equation*}
    \inv(z_0,\ldots,z_{\bar k},t_A(1),\ldots,t_A(m))=\inv(t_A(1),\ldots,t_A(m))
 +\sum_{j=0}^{\bar k}\bigl(z_j-|h^++h^-|-b_j\bigr).
\end{equation*}
The two binomial coefficients have cancelled. Reducing modulo $2$, and
using the root-move argument, yields
\[
 \inv(t_A(1),\ldots,t_A(m))
 \equiv\sum_{j=0}^{\bar k}
 \bigl(|y_j|+z_j-|h^++h^-|+b_j\bigr)\pmod2.
\]
The remaining step in the lemma is just substitution. From the definition
of $z_j$,
\[
 |y_j|+z_j-|h^++h^-|
 =\begin{cases}
 (K+1)y_j+j,&y_j\geq0,\\
 -(K+1)y_j-j-2|h^++h^-|,&y_j<0.
 \end{cases}
\]
Modulo $2$, this becomes
\[
 (K-1)y_j+j+2|h^++h^-|[y_j<0].
\]
Substituting into the above and using
$\sum_{j=0}^{\bar k}j=\binom{\bar k+1}{2}$ gives the asserted sign.
\end{proof}
We are ready to state the main result of this section. 
\begin{lem}[Determinant expansion]\label{lem: noodle-doodle expansion}
\label{thm: Noodle-Doodle expansion}
Let $M$ be an $n\times(n-m)$ matrix and let $\xi_1,\ldots,\xi_n$ be
functions on $\Z+|h^++h^-|$ satisfying
$\xi_i(a)=\epsilon\xi_i(-a)$ for a fixed $\epsilon\in\{\pm1\}$.
For arbitrary functions $p_+,p_-$ on $E$, we have
\begin{multline}\label{eq: general determinant expansion}
 \det_{\substack{1\leq i\leq n\\1\leq j\leq m}}
 \bigl(\matrixCols{M}{p_+(j)\xi_i(k+j+h^++h^-)
                         +p_-(j)\xi_i(k-j+h^+-h^-)}\bigr)\\
 =\sum_{(y_0,\ldots,y_{\bar k})}
 \epsilon^{\sum_{j=0}^{\bar k}(|y_j|-[y_j<0])}
 (-1)^{\binom{\bar k+1}{2}+
     \sum_{j=0}^{\bar k}((K-1)y_j+2|h^++h^-|[y_j<0]+b_j)}\\
 \times\prod_{j\in E\setminus E_A}p_+(j)\prod_{j\in E_A}p_-(j)
 \det_{\substack{1\leq i\leq n\\v\in V\setminus\{z_0,\ldots,z_{\bar k}\}}}
       \bigl(\matrixCols{M}{\xi_i(v)}\bigr).
\end{multline}
Here the sum ranges precisely over~\eqref{eq: admissible roots}, $E_A$ is
given by~\eqref{eq: EA is given by}, and $b_j$ is as in
Lemma~\ref{lem: column order}. The columns indexed by $v$ are in increasing
order.
\end{lem}
\begin{proof}
Apply Lemma~\ref{lem: Noodle-Doodle expansion first form} with
\[
 f_i=\xi_i,\qquad p_t(j)=p_+(j),\qquad
 p_s(j)=\epsilon^{[j\text{ is type }2]}p_-(j).
\]
This is valid also when the reflected argument is zero, since
$\xi_i(0)=\epsilon\xi_i(0)$. The factor contributed by the reflected
arguments is $\epsilon^{|E_A^2|}$, with $|E_A^2|$ given
by~\eqref{eq: tA(j) is equal to}. Every non-root vertex has indegree one,
so the set of targets is $V\setminus\{z_0,\ldots,z_{\bar k}\}$.
Parametrise the roots by~\eqref{eq: parametrisation of roots} and apply
Lemma~\ref{lem: column order} to put the columns in increasing order.
\end{proof}
Using Lemma \ref{lem: EA and tA(j) are given by} the weight can equivalently be written in the following form used in the applications below:
\begin{multline}\label{eq: product of edge weights}
 \prod_{j\in E\setminus E_A}p_+(j)\prod_{j\in E_A}p_-(j)\\=\prod_{j=1}^m p_+(j)
 \prod_{i=0}^{\bar k}\prod_{\ell=0}^{|y_i|-1}
 \frac{p_-(K\ell+k+\alpha_i i+\lceil h^+-h^-\rceil
                      -[y_i<0]\lceil|h^++h^-|\rceil)}
      {p_+(K\ell+k+\alpha_i i+\lceil h^+-h^-\rceil
                      -[y_i<0]\lceil|h^++h^-|\rceil)}.
\end{multline}

With the general determinant expansion established, we now state the six
specialisations needed in the subsequent sections. Table~\ref{tab: six expansions}
gives the graph and the fixed columns in each case; ``empty'' means that
$M$ has no columns.
\begin{table}[h]
\centering
\renewcommand{\arraystretch}{1.3}
\begin{tabular}{c c c c c}
Expansion&Graph&$\xi_i$&$\epsilon$&$M$\\\hline
$1.6a$&$\Lambda_{n-1,k,1/2,0}$&$\eta_i$&$-1$&$(\eta_i(k+1/2))_{i=1}^n$\\
$1.6b$&$\Lambda_{n,k,0,-1/2}$&$\eta_i$&$-1$&empty\\
$1.6c$&$\Lambda_{n-1,k,0,0}$&$\zeta_i$&$1$&$(\zeta_i(k))_{i=1}^n$\\
$1.8a$&$\Lambda_{n-1,k+1,0,0}$&$\eta_i$&$-1$&$(\eta_i(k+1))_{i=1}^n$\\
$1.8b$&$\Lambda_{n,k,0,0}$&$\eta_i$&$-1$&empty\\
$1.10a$&$\Lambda_{n,k,1/2,-1/2}$&$\eta_i$&$-1$&empty\\
\end{tabular}\\\,
\caption{The six specialisations of Lemma~\ref{lem: noodle-doodle expansion}.}
\label{tab: six expansions}
\end{table}

In the following six lemmas, $n,k\geq1$, and $p_+,p_-$ are functions on
$[1,n]$. Only values on the edge labels of the specified
graph are used. In every sum the roots must lie in the indicated vertex
interval, in addition to the endpoint restrictions displayed under the
summation sign. These restrictions make each sum finite. This condition can also be formulated as the vanishing of the determinant of a rectangular matrix.

\begin{lem}[Expansion 1.6a]\label{lem: noodle-doodle expansion 2.1a}
    Let $\eta_1,\dots,\eta_n$ be functions on $\frac{1}{2}\Z$ with $\eta_i(a) = -\eta_i(-a)$. 
    Then
    \begin{multline*}
        \det_{\substack{1\le i \le n\\ 1\leq j\leq n-1}}\big(\matrixCols{\eta_i(k+1/2)}{p_+(j)\eta_i(k+1/2+j)+p_-(j)\eta_i(k+1/2-j)}\big)\\=\sum_{\substack{(y_0,\dots,y_{ k}) \in\Z^{k + 1} \\ y_k \geq 0}}
        (-1)^{\binom{k+1}{2}
        +\sum_{j=0}^{k}
        (y_j+b_j)} \prod_{j=1}^{n-1}p_+(j) 
        \prod_{i=0}^{k}
        \prod_{\ell=0}^{\alpha_iy_i-1}
        \frac{p_-(K\ell+k+\alpha_i i+[y_i\geq 0])}{p_+(K\ell+k+\alpha_i i+[y_i\geq 0])}\\\times
         \det_{\substack{1\le i\le n \\0 \le j\le k+n-1 \\ j \not=z(y_r)-1/2, r=0,\dots,k}}(\matrixCols{\eta_i(k+1/2)}{\eta_i(j+1/2)}),
    \end{multline*}
    where $K=2k+1$ and $z(y_j)=\alpha_j(Ky_j+j+1/2)$ 
    with $\alpha_j=1$ if $y_j \geq 0$ and $-1$ otherwise, and $b_j=\vert\{i| z(y_i)> z(y_j)\textrm{ and }i>j\}\vert$.
    The sum is restricted by $\tfrac12\leq z(y_r)\leq k+n-\tfrac12$ for $r=0,\ldots,k$.
\end{lem}

\begin{lem}[Expansion 1.6b]\label{lem: noodle-doodle expansion 2.1b}
    Let $\eta_1,\dots,\eta_n$ be functions on $\frac{1}{2}\Z$ with $\eta_i(a) = -\eta_i(-a)$. 
    Then
    \begin{multline*}
        \det_{1\le i,j \le n}\big({p_+(j)\eta_i(k+j-1/2)+p_-(j)\eta_i(k-j+1/2)}\big)\\=\sum_{(y_0,\dots,y_{ k-1}) \in\Z^{k}}
        (-1)^{\binom{k}{2}
        +\sum_{j=0}^{k-1}b_j} \prod_{j=1}^{n}p_+(j) 
        \prod_{i=0}^{k-1}
        \prod_{\ell=0}^{\alpha_iy_i-1}
        \frac{p_-(K\ell+k+\alpha_i i+[y_i\geq 0])}{p_+(K\ell+k+\alpha_i i+[y_i\geq 0])}\\\times
         \det_{\substack{1\le i\le n \\0 \le j\le k+n-1 \\ j \not=z(y_r)-1/2, r=0,\dots,k-1}}({\eta_i(j+1/2)}),
    \end{multline*}
    where $K=2k$ and $z(y_j)=\alpha_j(Ky_j+j+1/2)$ 
    with $\alpha_j=1$ if $y_j \geq 0$ and $-1$ otherwise, and where we let $b_j=\vert\{i| z(y_i)> z(y_j)\textrm{ and }i>j\}\vert$.
    The sum is restricted by $\tfrac12\leq z(y_r)\leq k+n-\tfrac12$ for $r=0,\ldots,k-1$.
\end{lem}

\begin{lem}[Expansion 1.6c]\label{lem: noodle-doodle expansion 2.1c}
    Let $\zeta_1,\dots,\zeta_n$ be functions on $\Z$ with $\zeta_i(a) =\zeta_i(-a)$. 
    Then
    \begin{multline*}
        \det_{\substack{1\le i \le n\\ 1\leq j\leq n-1}}\big(\matrixCols{\zeta_i(k)}{p_+(j)\zeta_i(k+j)+p_-(j)\zeta_i(k-j)}\big)=\sum_{\substack{(y_0,\dots,y_{ k}) \in\Z^{k +1} \\ y_0,y_k \geq 0}}
        (-1)^{\binom{k+1}{2}
        +\sum_{j=0}^{k}
        (b_j-y_j)}\\\times\prod_{j=1}^{n-1}p_+(j) 
        \prod_{i=0}^{k}
        \prod_{\ell=0}^{\alpha_iy_i-1}
        \frac{p_-(K\ell+k+\alpha_i i)}{p_+(K\ell+k+\alpha_i i)}
         \det_{\substack{1\le i\le n \\0 \le j\le k+n-1 \\ j \not=z(y_r), r=0,\dots,k}}(\matrixCols{\zeta_i(k)}{\zeta_i(j)}),
    \end{multline*}
    where $K=2k$ and $z(y_j)=\alpha_j(Ky_j+j)$ 
    with $\alpha_j=1$ if $y_j \geq 0$ and $-1$ otherwise, and where we let  $b_j=\vert\{i| z(y_i)> z(y_j)\textrm{ and }i>j\}\vert$.
    The sum is restricted by $0\leq z(y_r)\leq k+n-1$ for $r=0,\ldots,k$.
\end{lem}

\begin{lem}[Expansion 1.8a]\label{lem: noodle-doodle expansion 2.3a}
    Let $\eta_1,\dots,\eta_n$ be functions on $\Z$ with $\eta_i(a) =-\eta_i(-a)$. 
    Then
    \begin{multline*}
        \det_{\substack{1\le i \le n\\ 1\leq j\leq n-1}}\big(\matrixCols{\eta_i(k+1)}{p_+(j)\eta_i(k+1+j)+p_-(j)\eta_i(k+1-j)}\big)
        \\=\sum_{\substack{(y_0,\dots,y_{k+1}) \in\Z^{k +2} \\ y_0,y_{k+1} \geq 0}}
        (-1)^{\binom{k+2}{2}
        +\sum_{j=0}^{k+1}
        (b_j-[y_j<0])}\prod_{j=1}^{n-1}p_+(j) 
        \prod_{i=0}^{k+1}
        \prod_{\ell=0}^{\alpha_iy_i-1}
        \frac{p_-(K\ell+k+1+\alpha_i i)}{p_+(K\ell+k+1+\alpha_i i)}
         \\\times\det_{\substack{1\le i\le n \\0 \le j\le k+n \\ j \not=z(y_r), r=0,\dots,k+1}}(\matrixCols{\eta_i(k+1)}{\eta_i(j)}),
    \end{multline*}
    where $K=2k+2$ and $z(y_j)=\alpha_j(Ky_j+j)$ 
    with $\alpha_j=1$ if $y_j \geq 0$ and $-1$ otherwise, and where we let $b_j=\vert\{i| z(y_i)> z(y_j)\textrm{ and }i>j\}\vert$.
    The sum is restricted by $0\leq z(y_r)\leq k+n$ for $r=0,\ldots,k+1$.
\end{lem}

\begin{lem}[Expansion 1.8b]\label{lem: noodle-doodle expansion 2.3b}
    Let $\eta_1,\dots,\eta_n$ be functions on $\Z$ with $\eta_i(a) = -\eta_i(-a)$. 
    Then
    \begin{multline*}
        \det_{1\le i,j \le n}\big({p_+(j)\eta_i(k+j)+p_-(j)\eta_i(k-j)}\big)=\sum_{\substack{(y_0,\dots,y_{ k}) \in\Z^{k + 1} \\y_0,y_k \geq 0}}
        (-1)^{\binom{k+1}{2}
        +\sum_{j=0}^{k}
        (b_j-[y_j<0])} \\
        \times \prod_{j=1}^np_+(j) 
        \prod_{i=0}^{k}
        \prod_{\ell=0}^{\alpha_iy_i-1}
        \frac{p_-(K\ell+k+\alpha_i i)}{p_+(K\ell+k+\alpha_i i)}
         \det_{\substack{1\le i\le n \\0 \le j\le k+n \\ j \not=z(y_r), r=0,\dots,k}}({\eta_i(j)}),
    \end{multline*}
    where $K=2k$ and $z(y_j)=\alpha_j(Ky_j+j)$ 
    with $\alpha_j=1$ if $y_j \geq 0$ and $-1$ otherwise, and where we let $b_j=\vert\{i| z(y_i)> z(y_j)\textrm{ and }i>j\}\vert$.
    The sum is restricted by $0\leq z(y_r)\leq k+n$ for $r=0,\ldots,k$.
\end{lem}

\begin{lem}[Expansion 1.10a]\label{lem: noodle-doodle expansion 2.5a}
    Let $\eta_1,\dots,\eta_n$ be functions on $\Z$ with $\eta_i(a) = -\eta_i(-a)$. 
    Then
    \begin{multline*}
        \det_{1\le i,j \le n}\big({p_+(j)\eta_i(k+j)+p_-(j)\eta_i(k+1-j)}\big)\\=\sum_{\substack{(y_0,\dots,y_{ k}) \in\Z^{k+1}\\ y_0\geq 0}}
        (-1)^{\binom{k+1}{2}
        +\sum_{j=0}^{k}(y_j+b_j-[y_j<0])} \prod_{j=1}^{n}p_+(j) 
        \prod_{i=0}^{k}
        \prod_{\ell=0}^{\alpha_iy_i-1}
        \frac{p_-(K\ell+k+\alpha_i i+1)}{p_+(K\ell+k+\alpha_i i+1)}\\\times
         \det_{\substack{1\le i\le n \\0 \le j\le k+n \\ j \not=z(y_r), r=0,\dots,k}}({\eta_i(j)}),
    \end{multline*}
    where $K=2k+1$ and $z(y_j)=\alpha_j(Ky_j+j)$ 
    with $\alpha_j=1$ if $y_j \geq 0$ and $-1$ otherwise, and where we let $b_j=\vert\{i| z(y_i)> z(y_j)\textrm{ and }i>j\}\vert$.
    The sum is restricted by $0\leq z(y_r)\leq k+n$ for $r=0,\ldots,k$.
\end{lem}

\subsection{A partition complement lemma}
\label{sec:lemmas.complements}
Given a partition $\lambda$, let $[\lambda]$ denote its Young diagram, thought as a collection of points (cells). Using English notation, the cells are indexed in a matrix-like fashion
\[
[\lambda] = \{(i,j)\in\Z^2 \mid 1\le j \le \lambda_i\}.
\]
In what follows, we think of partitions as embedded in a prescribed rectangular partition, either $(n^k)$ or $(k^n)$. The \emph{flip} of a partition $\mu\subseteq(n^k)$ is the set of cells
\[
\mathrm{flip}_{k^n}\mu = \{
(i,j)\in\Z^2 \mid 1\leq j\leq k\textrm{ and } n-\mu_{k-j+1}<i\le n 
\}\subseteq(k^n).
\]
The \emph{box-complementary} of a partition $\nu\subseteq(n^k)$ is the partition $\nu^{\square n^k}\subseteq(n^k)$ such that $\nu^{\square n^k}_i = n-\nu_{k+1-i}$.
For instance,
\[
\ytableausetup{smalltableaux,centertableaux}
[(3,3,1)] = \ydiagram{4,4,4}*[\circ]{3,3,1}*[*(yellow)]{3,3,1}
\,,\quad
\mathrm{flip}_{4^3}(3,3,1)  = \ydiagram{4,4,4}*[\cdot]{2+2,2+2,1+3}*[*(blue!20)]{2+2,2+2,1+3}\,,
\quad\text{and}\quad
(3,3,1)^{\square 4^3}  = (3,1,1).
\]
Let $\rho_a$ be the partition $(a-1, a-2, \dots, 1)$. Partitions are added entry-wise, and their union is taken as the union of multisets. Part of the following useful lemma appears in~\cite{2-row}.
\begin{lem}\label{lem:flip}
    Given two partitions $\lambda \subseteq (k^n)$ and $\mu \subseteq (n^k)$, the following are equivalent:
    \begin{enumerate}
        \item $(\lambda+\rho_n)\sqcup(\mu + \rho_k) = \rho_{k+n}$,
        \item $[\lambda] \sqcup \mathrm{flip}_{k^n}\mu = (k^n)$,
        \item $\lambda = (\mu^{\square n^k})'$,
        \item $m_{k-i}(\lambda) = \mu_i - \mu_{i+1}$ for all $i=0,\ldots,k-1$, where $\mu_0 = n$,
        \item $m_{k-i}(\lambda) = (\mu + \rho_k)_i - (\mu+\rho_k)_{i+1} - 1$ for all $i=0,\ldots,k-1$, where $(\mu+\rho_k)_0 = n+k$.
    \end{enumerate}
\end{lem}
\begin{defn}\label{de:lambda(nu)}
    Let $\nu$ be a composition and suppose that $\lambda$ and $\mathrm{sort}(\nu) = \mu+\rho_k$ are related in any of the above equivalent ways. Then, we write $\lambda = \lambda(\nu)$. See Figure~\ref{fig:flip}.
\end{defn}
\begin{proof}[Proof of Lemma~\ref{lem:flip}]
    Start by supposing $(\lambda+\rho_n) \sqcup (\mu+\rho_k) = \rho_{n+k}$.
    We are going to construct the following illustration step by step, where $\lambda$ is in yellow $\ytableaushort[*(yellow)]{{\circ}}$ and $\mathrm{flip}_{k^n}\mu$ in blue~$\ytableaushort[*(blue!20)]{{\cdot}}$. Once we are done, we will have established (1) implies (2).
    \[
    \ytableausetup{boxsize=.8em}
    \ydiagram[*(lightgray)]{1+8,2+7,3+6,4+5,5+4,6+3,7+2,8+1}
    \quad \leadsto \cdots \leadsto
    \ydiagram{1+8,2+7,3+6,4+5,5+4,6+3,7+2,8+1}*[*(lightgray)]{0,2+5,3+4,4+3,5+2,6+1}*[\cdot]{8+1,7+2,7+2,7+2}*[*(blue!20)]{8+1,7+2,7+2,7+2}*[\circ]{4+4,4+4}*[*(yellow)]{4+4,4+4}
    \quad \leadsto \cdots \leadsto
    \ydiagram{1+8,2+7,3+6,4+5,5+4,6+3,7+2,8+1}*[\cdot]{8+1,7+2,5+4,4+5}*[*(blue!20)]{8+1,7+2,5+4,4+5}*[\circ]{4+4,4+3,4+1}*[*(yellow)]{4+4,4+3,4+1}
    \]
    We start with a gray staircase $\rho_{n+k}$.

    In the first step, either $\mu_1+k-1 = n+k-1$ or $\lambda_1+n-1 = n+k-1$. In the first case, we have $\mu_1=n$, so we paint the first row with $k-1$ white and $n$ blue $\ytableaushort[*(blue!20)]{{\cdot}}$ cells. Otherwise, $\lambda_1=k$ and we color the last column with $k$ yellow $\ytableaushort[*(yellow)]{{\circ}}$ cells and $n-1$ white cells.

    In the $j$-th step, we have some $m$ rows and some $\ell$ columns with painted cells, and we are left with a gray staircase of size $n+k-m-\ell-1$.
    
    The biggest part of $\mu+\rho_k$ that we haven't used to color some row is $\mu_{m+1}+k-m-1$, and the biggest part of $\lambda+\rho_n$ not used in any column is $\lambda_{\ell+1}+n-\ell-1$. Either of these two quantities is equal to $n+k-m-\ell-1$, since this is the largest part of $\rho_{n+k}$ not still used in any row or column. Paint either the top row or the last column of the gray staircase accordingly. 
    
    After the $(n+k)$-th step, we get a $k\times n$ rectangle tiled by (translates of) $[\lambda]$ and $\mathrm{flip}_{k^n}\mu$.
    \medskip

    To go from (2) to (1), we reverse the algorithm: start with the full picture.
    Think of the boundary between blue $\ytableaushort[*(blue!20)]{{\cdot}}$  and yellow $\ytableaushort[*(yellow)]{{\circ}}$ cells as a lattice path from $(0,0)$ to $(k,n)$.
    Place a bullet at the $0$th coordinate of the lattice path, which is $(0,0)$.
    The first step is either north or east; move the bullet accordingly. 

    In the $j$-th step we have a gray staircase $\rho_j$ whose corner is touching the $j$-th coordinate of the lattice path. If the next step is north, then this means that $\lambda+\rho_n$ has a part of size~$j$; if the next step is east then this means that $\mu+\rho_k$ has a part of size~$j$. In either case, move the bullet point and grow the gray staircase accordingly.
    \[
    \ydiagram{1+8,2+7,3+6,4+5,5+4,6+3,7+2,8+1}*[*(lightgray)]{0,2+5,3+4,4+3,5+2,6+1}*[\cdot]{8+1,7+2,7+2,7+2}*[*(blue!20)]{8+1,7+2,7+2,7+2}*[\circ]{4+4,4+4}*[*(yellow)]{4+4,4+4}
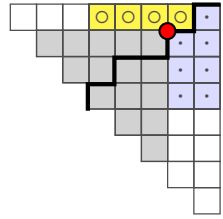
    \begin{tikzpicture}[remember picture, overlay, x=.85em, y=.85em, xshift=0em, yshift=.2em]
        \fill[white, opacity=.3] (-8,-4) rectangle (0.1,4.1);
        \draw[ultra thick] (-5,0) -- (-5,1) -- (-4,1) -- (-4,2) -- (-2,2) -- (-2,3) -- (-1,3) -- (-1,4) -- (0,4);
        \filldraw[fill=red, thick] (-2,3) circle (3pt);
    \end{tikzpicture}
    \]
    When the bullet arrives to $(k,n)$, we have that seen that all the possible sizes of parts $\{1, 2, \ldots, n+k-1\}$ that form $\rho_{n+k}$ are either coming from $\lambda+\rho_n$ or from $\mu+\rho_k$.\medskip

    That (2) and (3) are equivalent is clear from the definitions. Similarly, (4) and (5) are clearly equivalent. It remains thus to show (2) and (4) are equivalent.\medskip

    To show that (2) and (4) are equivalent, we come back to the lattice path defined above, which separates $[\lambda]$ from $\mathrm{flip}_{k^n}\mu$. The multiplicity $m_i(\lambda)$ is given by the number of north steps after the $i$-th east step of this path. But this is also counted by the difference of two consecutive parts of $\mu$, as claimed.
\end{proof}
\begin{figure}[h]
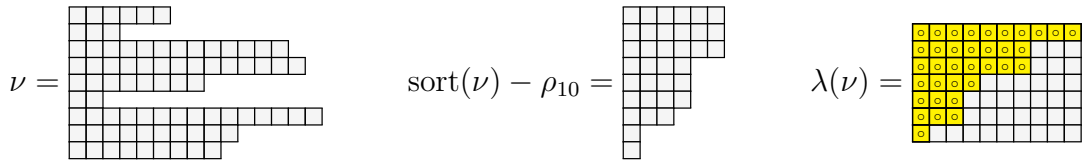

    \centering
    \hfill
    \ytableausetup{boxsize=.5em}
    $\nu = \ydiagram[*(lightgray!15)]{6,3,13,14,8,2,15,10,9}$
    \hfill
    $\mathrm{sort}(\nu) - \rho_{10} = \ydiagram[*(lightgray!15)]{6,6,6,4,4,4,3,1,1}$
    \hfill
    $\lambda(\nu)=
    \ydiagram[{\raisebox{.1em}{\scalebox{.5}{$\circ$}}}]{10,7,7,4,3,3,1}*[*(yellow)]{10,7,7,4,3,3,1}*[*(lightgray!15)]{10,10,10,10,10,10,10}$
    \hfill{\,}
    \caption{Let $n=7$ and $k=10$. Then $\lambda(6,3,13,14,8,2,15,10,9,0) = (10,7^2,4,3^2,1)$.}
    \label{fig:flip}
\end{figure}

The main way we use this construction is via the next corollary.
\begin{cor}\label{cor:useful.complementary}
    Fix $c\in\Z+\frac{1}{2}\Z$.
    Let $\epsilon \in \{\pm1\}$ and consider a family $\{\xi_i(m) \mid i\in[n],~m\in\Z+c\}$ of elements of a ring $R$.
    Let $\lambda = (\lambda_1, \ldots, \lambda_n)$ be a partition. If $\lambda = \lambda(\nu)$ then
    \[
    (-1)^{\binom{n}{2}}\det_{1\le i,j \le n}(\xi_i(\lambda_j + n-j + c)) = \det_{\substack{1\le i\le n \\ j\in([0,n+k-1]+c)\setminus\{\nu_a+c \mid a=1,\dots,k\}}}(\xi_i(j))^{1\le i\le n}.
    \]
    The converse also holds if there are no non-trivial algebraic relations between the $\{\xi_i(m) \mid i\in[n],~m\in\Z+c\}$.
\end{cor}

\begin{proof}
    If $\lambda = \lambda(\nu)$ then the identity is clear after noting that in the left-hand side the arguments are in decreasing order in $j$, whereas the arguments on the right-hand side are in increasing order; reordering columns we obtain a sign $(-1)^{\binom{n}{2}}$.

    Let us now prove the reciprocal.
    Since there are no algebraic relations between the $\xi_i(m)$ the equality of determinants implies the following identity of the arguments on the left and right-hand sides:
    \[
    \{\lambda_j + n - j + c \mid j=1,\ldots,n\} = ([0,n+k-1]+c)\setminus\{\nu_a+c ~\mid~ a=1,\dots,k\}.
    \]
    Cancel the summand of $c$ from the three sets. If $\lambda = (\lambda_1, \ldots, \lambda_n)$ is a partition then the left-hand side becomes the multiset of parts of $\lambda + \rho_n$. 
    On the right-hand side, note that $[0,n+k-1]$ is the set of parts of $\rho_{n+k}$. Hence if $\nu = (\nu_1, \ldots, \nu_k)$ is a composition then we have
    \[
    (\lambda + \rho_n) \sqcup \nu = \rho_{n+k},
    \]
    which is to say $\lambda = \lambda(\nu)$ as claimed.
\end{proof}

\subsection{Handling the summation range}\label{sec: Handling the summation range}
The following lemma presents a subtle argument necessary for deducing Corollaries \ref{cor-b}, \ref{cor-type-c}, and \ref{cor-type-bc} from Theorems \ref{thm-b-char}, \ref{main-type-c}, and \ref{thm-type-bc-0-char-all}, respectively.
\begin{lem}\label{lem: handling the summation range}
Let $n,k\in \Z_{\geq 0}$, $\nu\in\Z_{\geq 0}^k$ then $\lambda(\nu)$ is a partition inside $(k^n)$ if and only if $\nu_1,\dots,\nu_k\in[0,n+k-1]$ are all different elements. For the determinant expressions below, we use $\lambda(\nu)'_i=n-\mathrm{sort}(\nu)_{k+1-i}+i-1$ even if this sequence is not a partition. Furthermore, we have that
\begin{align*}
     \det\limits_{1\le i,j\le k}(\dot e_{\lambda(\nu)'_i-i+j}+ \dot e_{\lambda(\nu)'_i-i-j+1})&=\oo_{\lambda(\nu)}(X)\\
    \det\limits_{1\le i,j\le k}(\dot e_{\lambda(\nu)'_i-i+j}-\dot  e_{\lambda(\nu)'_i-i-j})&=\sp_{\lambda(\nu)}(X)\\
    \frac{1}{2}\det\limits_{1\le i,j\le k}(\dot e_{\lambda(\nu)'_i-i+j} +\dot e_{\lambda(\nu)'_i-i-j+2})&=\oe_{\lambda(\nu)}(X)\qquad(k\geq1).
\end{align*}
if $\nu_1,\dots,\nu_k\in[0,n+k-1]$ are all different elements and otherwise the left hand side of each of these equations is zero. Note that the right hand side is then zero too by \eqref{eq: convention not a partiton} so equality still holds.
\end{lem}
\begin{proof}
    For $k=0$, the first assertion and the first two identities are immediate, with empty determinants equal to $1$. We may therefore assume that $k\geq1$.

    First assume that $\lambda(\nu)$ is a partition inside $(k^n)$, thus $\lambda(\nu)'$ is a partition inside $(n^k)$. By definition we have that $\lambda(\nu)'_i=n-\mathrm{sort}(\nu)_{k+1-i}+i-1\in[0,n],$
    thus $\mathrm{sort}(\nu)_{k+1-i}\in[i-1,n+i-1]$. Therefore, all parts $\nu_1,\dots,\nu_k$ are inside $[0,n+k-1]$. Assume two of them were equal, then there exists $i_1\in[1,k-1]$ such that $\mathrm{sort}(\nu)_{k+1-i_1}=\mathrm{sort}(\nu)_{k-i_1}$. But then we have that $$\lambda(\nu)'_{i_1}=n-\mathrm{sort}(\nu)_{k+1-i_1}+i_1-1<n-\mathrm{sort}(\nu)_{k-i_1}+i_1=\lambda(\nu)'_{i_1+1},$$
    and thus $\lambda(\nu)'$ and subsequently $\lambda(\nu)$ would not be partitions. This contradicts our assumption and thus $\nu_1,\dots\nu_k$ all have to be different elements. 

    Now assume that $\nu_1,\dots,\nu_k\in[0,n+k-1]$ are all different elements. Then $i-1\leq \mathrm{sort}(\nu)_{k+1-i}\leq n+i-1$ and hence $n\geq \lambda(\nu)_i'=n-\mathrm{sort}(\nu)_{k+1-i}+i-1\geq 0$ and is weakly decreasing as 
    $$\lambda(\nu)_i'=n-\mathrm{sort}(\nu)_{k+1-i}+i-1\geq n-\mathrm{sort}(\nu)_{k-i}+i=\lambda(\nu)_{i+1}',$$
    for $1\leq i<k$, as $\mathrm{sort}(\nu)_{k+1-i}+1\leq \mathrm{sort}(\nu)_{k-i}$ because all the $\nu_i$ are different. Therefore $\lambda(\nu)'$ is a partition inside $(n^k)$ and so $\lambda(\nu)$ is a partition inside $(k^n)$.

    To prove the claim about the determinants, first note that if $\nu_1,\dots,\nu_k\in[0,n+k-1]$ are all different elements then $\lambda(\nu)$ is a partition as we just proved and then the formulae are just the dual Jacobi--Trudi identities from \eqref{eq: dual JT type B}, \eqref{eq: dual JT type C} and \eqref{eq: dual JT type D}. If the condition is not satisfied then there either exist $i,i'\in[1,k]$ with $i\neq i'$ such that $\nu_i=\nu_{i'}$ or there exists an $i\in[1,k]$ such that $\nu_i\geq n+k$. 
    
    In the first case we then have that there exists an $\ell\in[1,k-1]$ such that $\mathrm{sort}(\nu)_{k+1-\ell}=\mathrm{sort}(\nu)_{k-\ell}$. Therefore
    $$\lambda(\nu)'_{\ell}=n-\mathrm{sort}(\nu)_{k+1-\ell}+\ell-1=n-\mathrm{sort}(\nu)_{k-\ell}+\ell-1=\lambda(\nu)'_{\ell+1}-1,$$
    and thus $\lambda(\nu)'_\ell-\ell=\lambda(\nu)'_{\ell+1}-(\ell+1)$. Therefore the $\ell$-th and $(\ell+1)$-st rows of each of the three matrices agree. Hence the determinant vanishes and the left hand side is zero in each of the three cases.

    In the second case we get that $\mathrm{sort}(\nu)_1\geq n+k$. Therefore, we have $\lambda(\nu)'_k=n-\mathrm{sort}(\nu)_{1}+k-1\leq -1$. For $1\leq j\leq k$, we have
    \[
    \lambda(\nu)'_k-k+j\leq\lambda(\nu)'_k\leq-1,
    \qquad
    \lambda(\nu)'_k-k-j+2\leq\lambda(\nu)'_k-k+1\leq-1.
    \]
    Thus all indices in the $k$-th row of each of the three matrices are negative, and so this row vanishes since $\dot e_r=0$ for $r<0$. Then also the determinant vanishes and the left hand side is again zero in all three cases. 
\end{proof}
\section{Proofs of Theorem \ref{thm-b-char} and Corollary \ref{cor-b}}
\label{sec:type-b}
In this section, we provide proofs of all six formulae related to type $\mathrm{B}_n$ presented in Theorem \ref{thm-b-char} and Corollary \ref{cor-b}. We present the proofs of the formulae grouped by the parameter specialisations of $t,s$ in the $\mathrm{B}_n$ Hall--Littlewood polynomials. 
\subsection{The case \texorpdfstring{$s=0$}{s=0}.}
We begin with the first of three formulae \eqref{thm-type-b-0-char} of Theorem~\ref{thm-b-char}, namely
\[
    P_{(k^n)}^{\mathrm{B}_n}(X;t,0)=\sum_{y \in \mathbb{Z}^k}
(-1)^{|y|+\mathrm{inv}(\nu)}
t^{\sum_{j=1}^k \left( \frac{1}{2}Ky_j^2 + (j-\frac{1}{2})y_j \right)}
\oo_{\lambda(\nu)}(X)\,,
\]
where $K=2k+1$, $\nu_a = |Ky_a+a|-[y_a\ge0]$ for $a=1,\dots,k$ and $\lambda(\nu)$ as in Definition~\ref{de:lambda(nu)}.

\begin{proof}[Proof of \cref{thm-type-b-0-char}]
    Setting $s_1=0$ and $s_2=-1$ in the expression of Lemma~\ref{magical manipulation} we get
\begin{multline*}
    P^{\mathrm{B}_n}_{(k^n)}(X;t,0)= 
        \frac{\sqrt{t}^{\binom{n}{2}}}{2}
        \prod_{i=1}^nx_i^{n-1}\prod_{1\leq i<j\leq n}\frac{1}{(1-x_ix_j)(x_i-x_j)}\\\times\det_{1\le i,j\le n}\left(\frac{1}{1-x_i}\left((\sqrt{t})^{j-1}
        x_i^{j-1-k} +
        (\sqrt{t})^{-j+1}x_i^{-j+1-k}\right)\right.\\
        \left.-\frac{x_i}{1-x_i}\left((\sqrt{t})^{j-1}
        x_i^{-j+1+k}+(\sqrt{t})^{-j+1}x_i^{j-1+k}\right)\right)\\
        = \frac{\sqrt{t}^{\binom{n}{2}}}{2}\prod_{i=1}^n \frac{x_i^{n-1}}{x_i-1}
        \prod_{1\leq i<j\leq n}\frac{1}{(1-x_ix_j)(x_i-x_j)} 
        \\\times\det_{1\le i,j\le n} \left(
        (\sqrt{t})^{-j+1}\left(
        x_i^{k+(j-1)+1}
        - \x_i^{k+(j-1)}\right) +
        (\sqrt{t})^{j-1}\left(
        x_i^{k-(j-1)+1}
        - \x_i^{k-(j-1)}\right)
        \right).
\end{multline*}
We identify in the prefactor the Weyl denominator of type \ref{eq: Denom type B} up to a factor of $\prod_{i=1}^n x_i^{-1/2}$, which we can distribute over the determinant, and a sign $(-1)^{\binom{n}{2}}$ to obtain
\begin{multline*}
            \frac{(-\sqrt{t})^{\binom{n}{2}}}{2}\\\times 
        \frac{\det\limits_{1\leq i,j\leq n}\left(
        (\sqrt{t})^{-j+1}\left(
        x_i^{k+(j-1)+1/2}
        - \x_i^{k+(j-1)+1/2}\right)
        +
        (\sqrt{t})^{j-1}\left(
        x_i^{k-(j-1)+1/2}
        - \x_i^{k-(j-1)+1/2}\right)
        \right)}{\det\limits_{1\leq i,j\leq n}\left(x_i^{n-j+1/2}-\x_i^{n-j+1/2}\right)}\,.
\end{multline*}
Using the shorthand $\eta_j(m)$ for the expression $x_j^m-\x_j^m$, we observe that in the first column of the determinant, the $i$-th entry is $2\eta_i(k+1/2)$. We pull the $2$ out of the first column by linearity, and cancel it with the $1/2$ in the prefactor above to rewrite our expression as
\begin{multline*}
    P^{\mathrm{B}_n}_{(k^n)}(X;t,0)\\=(-\sqrt{t})^{\binom{n}{2}}
\frac{\det\limits_{\substack{1\le i\le n \\1\le j < n}}\left(
\matrixCols{
\eta_i(k+1/2)
}{
(\sqrt{t})^{-j}\eta_i(k+j+1/2) +
(\sqrt{t})^{j}\eta_i(k-j+1/2)
}\right)}{
\det\limits_{1\le i,j\le n }\left(\eta_i(n-j+1/2)\right)
}\,.
\end{multline*}
We invoke Lemma~\ref{lem: noodle-doodle expansion 2.1a} with $\eta_j(a)=x_j^a-\x_j^a$, $p_\pm(j) = (\sqrt{t})^{\mp j}$ and $K=2k+1$ to obtain
\begin{multline*}
    \det\limits_{\substack{1\le i\le n \\1\le j < n}}\left(
\matrixCols{
\eta_i(k+1/2)
}{
(\sqrt{t})^{-j}\eta_i(k+j+1/2) +
(\sqrt{t})^{j}\eta_i(k-j+1/2)
}\right)\\= \sum_{\substack{(y_0,\ldots,y_{k})\in\Z^{k+1}\\y_k\geq 0}}
(-1)^{|y|+\inv(\vec z_{0,k})}
t^{- \frac{1}{2}
 \binom{n}{2}
  + \sum_{i=0}^k \left( \frac{1}{2}Ky_i^2 + y_i (i+ \frac{1}{2}) \right)}\\\times\det_{\substack{1\le i\le n \\0 \le j\le k+n-1 \\ j \not=z(y_r)-1/2, r=0,\dots,k}}(\matrixCols{\eta_i(k+1/2)}{\eta_i(j+1/2)}),
\end{multline*}
where $z(y_j)=\alpha_j(Ky_j+j+1/2)$ and we denote $\vec{z}_{0,k}\coloneqq(z(y_0), z(y_1), \ldots, z(y_k))$. This is because $$(-1)^{\binom{k+1}{2}
        +\sum_{j=0}^{k}
        (y_j+b_j)}=(-1)^{|y|+\inv(\vec z_{0,k})},$$
as $ \inv(\vec z_{0,k})= |\{(i,j)\mid i>j,~ z(y_i) < z(y_j)\}|=\binom{k+1}{2} - \sum_{j=0}^{k} b_j$
is the number of inversions of $\vec z_{0,k}$ and we know that
\begin{equation*}
    \prod_{j=1}^{n-1}p_+(j) 
        \prod_{i=0}^{k}
        \prod_{\ell=0}^{\alpha_iy_i-1}
        \frac{p_-(K\ell+k+\alpha_i i+[y_i\geq 0])}{p_+(K\ell+k+\alpha_i i+[y_i\geq 0])}\\=t^{- \frac{1}{2}
 \binom{n}{2}
  + \sum_{i=0}^k \left( \frac{1}{2}Ky_i^2 + y_i (i+ \frac{1}{2}) \right)},
\end{equation*}
by the following simple calculation:
\begin{multline*}
-\sum_{j=1}^{n-1}
\frac{j}{2} +
\sum_{i=0}^k
\sum_{\ell=0}^{|y_i|-1}
\left(K\ell + k + \alpha_ii + [y_i\ge0] \right)
\\ =
- \frac{1}{2}\binom{n}{2} + 
 \sum_{i=0}^k
 \left(K\binom{|y_i|}{2} + y_ii + |y_i|\Big( k+[y_i\ge0]\Big)\right)
\\
=  - \frac{1}{2}
 \binom{n}{2}
  + \sum_{i=0}^k \left( \frac{1}{2}Ky_i^2 + y_i (i+{\textstyle \frac{1}{2}}) \right),
\end{multline*}
where the last equality holds summand-by-summand by distinguishing cases on the sign of~$y_i$. Therefore, inserting into our expression from before we have
\begin{multline*}
     P^{\mathrm{B}_n}_{(k^n)}(X;t,0) 
 =     
 \sum_{\substack{(y_0,\ldots,y_{k})\in\Z^{k+1}\\y_k\geq 0}}
(-1)^{\binom{n}{2}+|y|+\inv(\vec z_{0,k})}
t^{\sum_{i=0}^{k} \left( \frac{1}{2}Ky_i^2 + y_i (i+{\frac{1}{2}}) \right)}\\\times
\frac{\det\limits_{\substack{1\le i\le n \\0 \le j\le k+n-1 \\ j \not=z(y_r)-1/2, r=0,\dots,k}}(\matrixCols{\eta_i(k+1/2)}{\eta_i(j+1/2)})}{\det\limits_{1 \le i,j \le n}(\eta_i(n-j+1/2))}\,.
\end{multline*}

Now consider the determinant
$$\det_{\substack{1\le i\le n \\0 \le j\le k+n-1 \\ j \not=z(y_r)-1/2, r=0,\dots,k}}(\matrixCols{\eta_i(k+1/2)}{\eta_i(j+1/2)}),$$
in the numerator. The first column is $\eta_i(k+1/2)$. However, if $z(y_i)\ne k+1/2$ for all $i$, then there is a second column of the same form, and hence the determinant vanishes. Note that $z(y_i) = k+1/2$ if and only if $i=k$ and $y_k\in\{0,-1\}$. Since $y_k\ge0$, we have only one option and may thus assume $y_k=0$ from now on. Note that then $$ \inv (\vec z_{0,k})=\inv((z(y_0),\dots,z(y_{k-1}),k+1/2)=\inv(\vec z_{0,k-1})+|\{j\in[0,k-1]\mid k+1/2<z(y_j)\}|.$$
We reorder the columns of the determinant to be in increasing order with respect to the arguments of $\eta_i$. The only misplaced column is the first. This introduces a sign of $$(-1)^{k-|\{j\in[0,k-1]\mid z(y_j)< k+1/2\}|}=(-1)^{|\{j\in[0,k-1]\mid k+1/2<z(y_j)\}|},$$
as there are $k-|\{j\in[0,k-1]\mid z(y_j)< k+1/2\}|$ many columns after the first column of our matrix where the argument of $\eta_i$ is smaller than $k+1/2$. Thus, we have arrived at
\begin{multline*}
     P^{\mathrm{B}_n}_{(k^n)}(X;t,0) 
 =     
 \sum_{(y_0,\ldots,y_{k-1})\in\Z^{k}}
(-1)^{\binom{n}{2}+|y|+\inv(\vec z_{0,k-1})}
t^{\sum_{i=0}^{k-1} \left( \frac{1}{2}Ky_i^2 + y_i (i+{\frac{1}{2}}) \right)}\\\times
\frac{\det\limits_{\substack{1\le i\le n \\0 \le j\le k+n-1 \\ j \not=z(y_r)-1/2, r=0,\dots,k-1}}(\eta_i(j+1/2))}{\det\limits_{1 \le i,j \le n}(\eta_i(n-j+1/2))}\,.
\end{multline*}

To be able to use the Weyl character formula of type $\mathrm{B}_n$ \eqref{eq: Weyl character formula type B} and Corollary~\ref{cor:useful.complementary}, we define a vector $\nu\coloneqq (\nu_1,\dots, \nu_k)$ by
\[
\nu_i := z(y_{i-1})-\frac{1}{2}=\alpha_{i-1}(Ky_{i-1}+i-1+1/2)-1/2 = |Ky_{i-1}+i| - [y_{i-1}\ge0]
\]
for $i=1,\ldots,k$
and use $\inv(\vec z_{0,k-1}) = \inv(\nu),$
to reindex the sum by simultaneously setting $y_i\rightarrow y_{i+1}$ for $i=0,\dots k-1$ and write it as a sum over tuples $(y_1,\ldots,y_k)\in \Z^k$, to obtain
\[
 P^{\mathrm{B}_n}_{(k^n)}(X;t,0) 
 =     
 \sum_{(y_1,\ldots,y_k)\in\Z^{k}}
(-1)^{\binom{n}{2}+|y|+\inv(\nu)}
t^{\sum_{i=1}^k \left( \frac{1}{2}Ky_i^2 + y_i (i-{\frac{1}{2}}) \right)}
\frac{\det\limits_{\substack{1\le i\le n \\0 \le j\le k+n-1 \\ j \not={\nu}_r, r=1,\dots,k}}(\eta_i(j+1/2))}{\det\limits_{1 \le i,j \le n}(\eta_i(n-j+1/2))}\,.
\]

By the Weyl character formula of type $\mathrm{B}_n$ \eqref{eq: Weyl character formula type B} and Corollary~\ref{cor:useful.complementary}, we know that
\[
\frac{\det\limits_{\substack{1\le i\le n \\0 \le j\le k+n-1 \\ j \not={\nu}_r, r=1,\dots,k}}(\eta_i(j+1/2))}{\det\limits_{1 \le i,j \le n}(\eta_i(n-j+1/2))}= 
(-1)^{\binom{n}{2}}\oo_{\lambda(\nu)}(X).
\]
Note that the condition for the matrix in the numerator to be a square matrix is precisely the condition that the $y_a$'s corresponding $\nu_a$'s are inside $\{0,\dots,n+k-1\}$ and are all different. Plugging this into our expression for $P^{\mathrm{B}_n}_{(k^n)}(X;t,0)$ from above yields
\[
    P^{\mathrm{B}_n}_{(k^n)}(X;t,0)= 
 \sum_{(y_1,\ldots,y_k)\in\Z^{k}}
 (-1)^{|y|+\inv(\nu)}
t^{\sum_{i=1}^k \left( \frac{1}{2}Ky_i^2 + y_i (i-{\frac{1}{2}}) \right)}\oo_{\lambda(\nu)}(X),
\]
where we sum over all $y\in\Z^k$ such that the corresponding $\nu_a$'s are inside $\{0,\dots,n+k-1\}$ and are all different, as claimed. 
\end{proof}
Having finished the proof of \eqref{thm-type-b-0-char}, we are ready to prove \eqref{thm-b-0}.
Recall that the dual Jacobi--Trudi formula for $P^{\mathrm{B}_n}_{(k^n)}(X;t,0)$ stated in \eqref{thm-b-0} is
\[
\sum_{y \in \mathbb{Z}^k}
\det_{1 \le i,j \le k}
\left(
(-1)^{y_i}
t^{\frac{1}{2} K y_i^2 - (j-\frac{1}{2})y_i}
\left(
\dot{e}_{\,n-i+j-K y_i}
+
\dot{e}_{\,n+i+j- Ky_i-1}
\right)
\right)\,
\]
where $K=2k+1$.
\begin{proof}[Proof of ~\cref{thm-b-0}]
Let $\nu_a = |Ky_a + a| - [y_a\ge0]$ as in~\eqref{thm-type-b-0-char}. 
Unravelling the Definition~\ref{de:lambda(nu)}, we have
\begin{multline*}
\lambda(\nu)'_i = (\mathrm{sort}(\nu) - \rho_k)^{\square n^k}_i = 
n - (\mathrm{sort}(\nu) - \rho_k)_{k+1-i} =
n - \mathrm{sort}(\nu)_{k+1-i} + i -1\\
= n - |Ky_{\sigma(i)} + {\sigma(i)}| + [y_{\sigma(i)}\ge0] + i-1\,,
\end{multline*}
where $\sigma$ is the permutation that sorts $\nu$ into increasing order.

We apply the dual Jacobi--Trudi identity of type $\mathrm{B}_n$ \eqref{eq: dual JT type B} to the formula~\ref{thm-type-b-0-char} of Theorem \ref{thm-b-char} to get
\[
P_{(k^n)}^{\mathrm{B}_n}(X;t,0)=\sum_{y \in \mathbb{Z}^k}
(-1)^{|y|+\inv(\nu)}
t^{\sum_{j=1}^k \left( \frac{1}{2}K y_j^2 + (j-\frac{1}{2})y_j \right)}
\det_{1\leq i,j\leq k}(\dot{e}_{\lambda(\nu)'_i-i+j}
+ 
\dot{e}_{\lambda(\nu)'_i-i-j+1}).
\]
By the assumption in Theorem \ref{thm-b-char}, the sum ranges over all vectors $y\in\Z^k$ such that $\nu_1,\dots,\nu_k\in[0,n+k-1]$ are all different elements. By Lemma \ref{lem: handling the summation range} the determinant on the right hand side of this sum vanishes outside of this range of summation, hence we can extend the sum, to a sum over all $y\in\Z^k$. 

Plugging in the expression for $\lambda(\nu)'_i$ that we got before, the determinant is given by
\[
\det_{1\leq i,j\leq k}\left(
\dot{e}_{
n - |Ky_{\sigma(i)} + {\sigma(i)}| + [y_{\sigma(i)}\ge0]-1 + j
} +
\dot{e}_{
n - |Ky_{\sigma(i)} + {\sigma(i)}| + [y_{\sigma(i)}\ge0] - j 
}\right).
\]
Now we rearrange the rows of the matrix by the permutation $\sigma^{-1}$. This cancels the factor $\sgn(\sigma) = (-1)^{\mathrm{inv}(\nu)}$, since $\sigma$ is the permutation that sorts $\nu$ in increasing order and hence the sign is given by the number of inversions of $\nu$.  
The determinant becomes
\[
\det_{1\leq i,j\leq k}\left(
\dot{e}_{
n - |Ky_{i} + {i}| + [y_{i}\ge0] -1+ j
} +
\dot{e}_{
n - |Ky_{i} + {i}| + [y_{i}\ge0] - j 
}\right).
\]
We distinguish cases according to the sign of $y_i$ and fix some $1\leq j\leq k$.
If $y_i$ is non-negative, we use the formula
$\dot{e}_{n+r} = \dot{e}_{n-r}$ to get 
\begin{multline*}
\dot{e}_{
n - |Ky_{i} + {i}| + [y_{i}\ge0] -1+ j
} +
\dot{e}_{
n - |Ky_{i} + {i}| + [y_{i}\ge0] - j 
}\\=
\dot{e}_{
n - Ky_{i} - i + j 
} +
\dot{e}_{
n - Ky_{i} - i - j + 1
}=
\dot{e}_{
n + Ky_{i} + i - j 
} +
\dot{e}_{
n + Ky_{i} + i + j - 1
}.
\end{multline*}
If $y_i$ is negative, we immediately get that
\begin{equation*}
   \dot{e}_{
n - |Ky_{i} + {i}| + [y_{i}\ge0] -1+ j
} +
\dot{e}_{
n - |Ky_{i} + {i}| + [y_{i}\ge0] - j 
}=
\dot{e}_{
n + Ky_{i} + i + j-1
} +
\dot{e}_{
n + Ky_{i} + i - j}.
\end{equation*}
Replacing $y_i$ by $-y_i$ and transposing our matrix, we arrive at
    \[
\sum_{y \in \mathbb{Z}^k}
(-1)^{\sum_{j=1}^k y_j}
t^{\sum_{j=1}^k \left( \frac{1}{2} K y_j^2 - (j-\frac{1}{2})y_j \right)}
\det_{1 \le i,j \le k}
\left(
\dot{e}_{\,n-i+j-K y_j}
+
\dot{e}_{\,n+i+j- Ky_j-1}
\right).\]
By Lemma \ref{lem: handling the summation range}, the support of this sum is finite, so extending the sum to a finite $S_k$ invariant subset containing the support, by Lemma~\ref{lem:powery}, we get that this is equal to the right-hand side of~\eqref{thm-b-0}.
\end{proof}

\subsection{The case \texorpdfstring{$s=-t^{1/2}$}{s=-t^1/2}.}
We continue with the second of three formulae \eqref{thm-type-b-t-char} of Theorem~\ref{thm-b-char}, that is
\[
P_{(k^n)}^{\mathrm{B}_n}(X;t,-t^{1/2})=\sum_{(y_1, \ldots, y_k)\in\Z^k}
(-1)^{\inv(\nu)}
t^{\sum_{j=1}^{k} \left( \frac{1}{2}Ky_j^2 + (j-\frac{1}{2})y_j \right)}
\oo_{\lambda(\nu)}(X),\]
where $K=2k$, $\nu_a = |Ky_a + a| - [y_a\ge0]$ for $a=1,\dots,k$ and $\lambda(\nu)$ as in Definition~\ref{de:lambda(nu)}.
\begin{proof}[Proof of~\cref{thm-type-b-t-char}]
We set $s_1=-\sqrt{t}$ and $s_2=-1$ in the expression of Lemma~\ref{magical manipulation} to obtain
\begin{multline*}
    P_{(k^n)}^{\mathrm{B}_n}(X;t,-t^{1/2})= \frac{\sqrt{t}^{\binom{n}{2}}}{2}
    \prod_{i=1}^nx_i^{n-1}\prod_{1\leq i<j\leq n}\frac{1}{(1-x_ix_j)(x_i-x_j)}\\\times\det_{1\le i,j\le n}\left(\frac{(1+\sqrt{t}x_i)}{1-x_i}
    \left((\sqrt{t})^{j-1}x_i^{j-1-k}+
    (\sqrt{t})^{-j+1}x_i^{-j+1-k}\right)\right.\\\left.-\frac{(\sqrt{t}+x_i)}{1-x_i}
    \left((\sqrt{t})^{j-1}x_i^{-j+1+k}
    +(\sqrt{t})^{-j+1}x_i^{j-1+k}\right)\right)\\=
    \frac{\sqrt{t}^{\binom{n}{2}}}{2}
    \prod_{i=1}^n\frac{x_i^{n-1}}{x_i-1}\prod_{1\leq i<j\leq n}\frac{1}{(1-x_ix_j)(x_i-x_j)}\\\times\det_{1\le i,j\le n}\left(-(1+\sqrt{t}x_i)
    \left((\sqrt{t})^{j-1}x_i^{j-1-k}+
    (\sqrt{t})^{-j+1}x_i^{-j+1-k}\right)\right.\\\left.+(\sqrt{t}+x_i)
    \left((\sqrt{t})^{j-1}x_i^{-j+1+k}
    +(\sqrt{t})^{-j+1}x_i^{j-1+k}\right)\right).
\end{multline*}
Similarly as in the proof of \eqref{thm-type-b-0-char}, we identity the Weyl denominator of type \ref{eq: Denom type B} up to a factor of $\prod_{i=1}^nx_i^{-1/2}$, which we distribute over the determinant, and a sign $(-1)^{\binom{n}{2}}$ to obtain
\begin{multline*}
    \frac{(-\sqrt{t})^{\binom{n}{2}}}{2}
    \det_{1\le i,j\le n}\left(-(1+\sqrt{t}x_i)\left((\sqrt{t})^{j-1}x_i^{j-3/2-k}+(\sqrt{t})^{-j+1}x_i^{-j+1/2-k}\right)\right.\\\left.+
  {( \sqrt{t}+x_i)}
  \left((\sqrt{t})^{j-1}x_i^{-j+1/2+k}+(\sqrt{t})^{-j+1}x_i^{j-3/2+k}\right)\right)/\det_{1\leq i,j\leq n}\left(x_i^{n-j+1/2}-\x_i^{n-j+1/2}\right)
\end{multline*}
Using our shorthand notation $\eta_j(m)$ for the expression $x_j^m-\x_j^m$ again, we can rewrite this expression as
\begin{multline*}
    \frac{(-\sqrt{t})^{\binom{n}{2}}}{2}\det\limits_{1\leq i,j\leq n}\left((\sqrt{t})^{j-1} \eta_i(k-j+3/2)
 +(\sqrt{t})^{j} \eta_i(k-j+1/2) \right.\\\left.
 +(\sqrt{t})^{-j+1} \eta_i(k+j-1/2) +(\sqrt{t})^{-j+2}\eta_i(k+j-3/2)\right)/\det\limits_{1\leq i,j\leq n}(\eta_i(n-j+1/2)),
\end{multline*}
Now observe that the determinant in the numerator equals the following determinant, where we factored out the $2$ from the first column, since the first matrix can be obtained from the latter by applying elementary column operations $C_j \rightarrow C_j+C_{j-1}$ for $j=n,\dots, 2$, 
\begin{multline*}
2 \det\limits_{1\leq i,j\leq n}\left((\sqrt{t})^{-j+1} \eta_i(k+j-1/2)
+(\sqrt{t})^j \eta_i(k-j+1/2)\right)
\\ = 2(\sqrt{t})^{\frac{n}{2}}\cdot
\det\limits_{1\leq i,j\leq n}\left((\sqrt{t})^{-(j-\frac{1}{2})} \eta_i(k+j-1/2)
+
(\sqrt{t})^{j-\frac{1}{2}} \eta_i(k-j+1/2)\right)\,.
\end{multline*}
So our expression is equal to
\begin{multline*}
    P_{(k^n)}^{\mathrm{B}_n}(X;t,-t^{1/2})\\=(-1)^{\binom{n}{2}}(\sqrt{t})^{\binom{n}{2}+\frac{n}{2}}\frac{\det\limits_{1\leq i,j\leq n}\left((\sqrt{t})^{-(j-\frac{1}{2})} \eta_i(k+j-1/2)
+
(\sqrt{t})^{j-\frac{1}{2}} \eta_i(k-j+1/2)\right)}{\det\limits_{1\leq i,j\leq n}(\eta_i(n-j+1/2))}.
\end{multline*}
We use Lemma~\ref{lem: noodle-doodle expansion 2.1b} with $\eta_j(a)=x_j^a-\x_j^a$, $p_{\pm}(j) = (\sqrt{t})^{\mp(j-\frac{1}{2})}$ and $K=2k$ to obtain
\begin{multline*}
    \det\limits_{1\leq i,j\leq n}\left((\sqrt{t})^{-(j-\frac{1}{2})} \eta_i(k+j-1/2)
+
(\sqrt{t})^{j-\frac{1}{2}} \eta_i(k-j+1/2)\right)\\= \sum_{(y_0,\dots,y_{ k-1}) \in\Z^{k}}
        (-1)^{\inv(\vec z_{0,k-1})}t^{-\frac{n^2}{4} + \sum_{i=0}^{k-1} \left( \frac{1}{2}Ky_i^2 + (i+\frac{1}{2})y_i \right)}\\\times
         \det_{\substack{1\le i\le n \\0 \le j\le k+n-1 \\ j \not=z(y_r)-1/2, r=0,\dots,k-1}}({\eta_i(j+1/2)}),
\end{multline*}
where $z(y_j)=\alpha_j(Ky_j+j+1/2)$ and we denote $\vec z_{0,k-1}\coloneqq(z(y_0),\dots,z(y_{k-1}))$. This is because as in the proof of \eqref{thm-type-b-0-char}, we have $\binom{k}{2} - \sum_{j=0}^{k-1} b_j = \inv(\vec z_{0,k-1})$, and furthermore we get that
\begin{equation*}
    \prod_{j=1}^{n}p_+(j) 
        \prod_{i=0}^{k-1}
        \prod_{\ell=0}^{\alpha_iy_i-1}
        \frac{p_-(K\ell+k+\alpha_i i+[y_i\geq 0])}{p_+(K\ell+k+\alpha_i i+[y_i\geq 0])}=t^{-\frac{n^2}{4} + \sum_{i=0}^{k-1} \left( \frac{1}{2}Ky_i^2 + (i+\frac{1}{2})y_i \right)},
\end{equation*}
from the following simple calculation:
\begin{multline*}
-\frac{1}{2}\sum_{j=1}^n \left(j-\frac{1}{2}\right) +\sum_{i=0}^{k-1}
\sum_{\ell=0}^{|y_i|-1}
\left(K\ell+k + \alpha_ii + [y_i\ge0]-\frac{1}{2}\right)
\\
= -\frac{1}{2}\left(\binom{n+1}{2} - \frac{n}{2}\right) + \sum_{i=0}^{k-1}
\left( K\binom{|y_i|}{2} + |y_i|\left(k +\alpha_i i + [y_i\ge0]-\frac{1}{2}\right) \right)
\\
= 
-\frac{n^2}{4} + \sum_{i=0}^{k-1} \left( \frac{1}{2}Ky_i^2 + (i+{\textstyle\frac{1}{2}})y_i \right),  
\end{multline*}
where the last equality follows summand-by-summand by a case distinction on the sign of $y_i$. Inserting this into our expression from before we obtain
\begin{multline*}
    P_{(k^n)}^{\mathrm{B}_n}(X;t,-t^{1/2}) \\=
\sum_{(y_0, \ldots, y_{k-1})\in\Z^k}
(-1)^{\binom{n}{2}+\inv(\vec z_{0,k-1})}
t^{\sum_{i=0}^{k-1} \left( \frac{1}{2}Ky_i^2 + (i+\frac{1}{2})y_i \right)}
\frac{\det\limits_{\substack{1\le i\le n \\0 \le j\le k+n-1 \\ j \not=z(y_r)-1/2, r=0,\dots,k-1}}(\eta_i(j+1/2))}{\det\limits_{1 \le i,j \le n}(\eta_i(n-j+1/2))}   \,.
\end{multline*}

To able to use the Weyl character formula of type $\mathrm{B}_n$ \eqref{eq: Weyl character formula type B} and Corollary~\ref{cor:useful.complementary}, we define a vector $\nu\coloneqq(\nu_1,\dots,\nu_k)$ by
$$\nu_i = z(y_{i-1}) -\frac{1}{2} = \alpha_{i-1}\Big(Ky_{i-1} + i - 1 + \frac{1}{2}\Big)-\frac{1}{2}
= |Ky_{i-1} + i| - [y_{i-1}\ge0],$$
for $i=1,\ldots,k$
and use $\inv(\vec z_{0,k-1}) = \inv(\nu)$
and then simultaneously set $y_i\rightarrow y_{i+1}$ for $i=0,\dots k-1$ to write our expression as
\[
P_{(k^n)}^{\mathrm{B}_n}(X;t,-t^{1/2})=
\sum_{(y_1, \ldots, y_k)\in\Z^k}
(-1)^{\binom{n}{2}+\inv(\nu)}
t^{\sum_{i=1}^{k} \left( \frac{1}{2}Ky_i^2 + (i-\frac{1}{2})y_i \right)}
\frac{\det\limits_{\substack{1\le i\le n \\0 \le j\le k+n-1 \\ j \not=\nu_r, r=1,\dots,k}}(\eta_i(j+1/2))}{\det\limits_{1\le i,j\le n}(\eta_i(n-j+1/2))}.
\]

By the Weyl character formula of type $\mathrm{B}_n$ \eqref{eq: Weyl character formula type B} and Corollary~\ref{cor:useful.complementary} we find that
\[
(-1)^{\binom{n}{2}}
\frac{\det\limits_{\substack{1\le i\le n \\0 \le j\le k+n-1 \\ j \not=\nu_r, r=1,\dots,k}}(\eta_i(j+1/2))}{\det\limits_{1\le i,j\le n}(\eta_i(n-j+1/2))}
= \oo_{\lambda(\nu)}(X).
\]
Note that the condition for the matrix in the numerator to be a square matrix is precisely the condition that the $y_a$'s corresponding to the $\nu_a$'s are inside $\{0,\dots,n+k-1\}$ and are all different.
Plugging this into our expression from above we obtain
\[
P_{(k^n)}^{\mathrm{B}_n}(X;t,-t^{1/2})=
\sum_{(y_1, \ldots, y_k)\in\Z^k}
(-1)^{\inv(\nu)}
t^{\sum_{i=1}^{k} \left( \frac{1}{2}Ky_i^2 + (i-\frac{1}{2})y_i \right)}
\oo_{\lambda(\nu)}(X),
\]
where we sum over all $y\in\Z^k$ such that the corresponding $\nu_a$'s are inside $\{0,\dots,n+k-1\}$ and are all different, as claimed.
\end{proof}

Having finished the proof of \eqref{thm-type-b-t-char} we are ready to prove \eqref{thm-type-b-t}. Recall that the dual Jacobi--Trudi formula for $P^{\mathrm{B}_n}_{(k^n)}(X;t,-t^{1/2})$ stated in \eqref{thm-type-b-t} is
\[
\sum_{y \in \mathbb{Z}^k}
\det_{1 \le i,j \le k}
\left(
t^{\frac{1}{2} K y_i^2 - (j-\frac{1}{2})y_i}
\left(
\dot{e}_{\,n-i+j-K y_i}
+
\dot{e}_{\,n+i+j- Ky_i-1}
\right)
\right)
\]
where $K=2k$.
\begin{proof}[Proof of ~\cref{thm-type-b-t}]
Let $\nu_a = |Ky_a + a | - [y_a\ge0]$ as in~\eqref{thm-type-b-t-char}. 
Unravelling the Definition~\ref{de:lambda(nu)}, we have
\begin{multline*}
\lambda(\nu)'_i = (\mathrm{sort}(\nu) - \rho_k)^{\square n^k}_i = 
n - (\mathrm{sort}(\nu) - \rho_k)_{k+1-i} =
n - \mathrm{sort}(\nu)_{k+1-i} + i -1\\
= n - |Ky_{\sigma(i)} + {\sigma(i)}| + [y_{\sigma(i)}\ge0] + i-1\,,
\end{multline*}
where $\sigma$ is the permutation that sorts $\nu$ into increasing order.

We apply the dual Jacobi--Trudi identity of type $\mathrm{B}_n$ \eqref{eq: dual JT type B} to the formula \ref{thm-type-b-t-char} to get
\[
P_{(k^n)}^{\mathrm{B}_n}(X;t,-t^{1/2})=\sum_{(y_1, \ldots, y_k)\in\Z^k}
(-1)^{\mathrm{inv}(\nu)}t^{\sum_{j=1}^k \left( \frac{1}{2}Ky_j^2 + (j -
\frac{1}{2})y_j \right)
}
\det_{1\leq i,j\leq k}(\dot{e}_{\lambda(\nu)'_i-i+j}
+ 
\dot{e}_{\lambda(\nu)'_i-i-j+1}).
\]
By the assumption in Theorem \ref{thm-b-char}, the sum ranges over all vectors $y\in\Z^k$ such that $\nu_1,\dots,\nu_k\in[0,n+k-1]$ are all different elements. By Lemma \ref{lem: handling the summation range} the determinant on the right hand side of this sum vanishes outside of this range of summation, hence we can extend the sum, to a sum over all $y\in\Z^k$. 

Plugging in the expression for $\lambda(\nu)'_i$ that we got before, this determinant is
\[
\det_{1\leq i,j\leq k}\left(
\dot{e}_{
n - |Ky_{\sigma(i)} + {\sigma(i)} | + [y_{\sigma(i)}\ge0]-1 + j
} +
\dot{e}_{
n - |Ky_{\sigma(i)} + {\sigma(i)}| + [y_{\sigma(i)}\ge0] - j 
}\right).
\]
Now we rearrange the rows of the matrix by the permutation $\sigma^{-1}$. This cancels the factor $\sgn(\sigma) = (-1)^{\mathrm{inv}(\nu)}$, as $\sigma$ is the permutation that sorts $\nu$ in increasing order, hence the sign is given by the number of inversions of $\nu$. The matrix becomes
\[
\det_{1\leq i,j\leq k}\left(
\dot{e}_{
n - |Ky_{i} + {i}| + [y_{i}\ge0] -1+ j
} +
\dot{e}_{
n - |Ky_{i} + {i}| + [y_{i}\ge0] - j 
}\right).
\]
We distinguish cases according to the sign of $y_i$ and fix some $1\leq j\leq k$.
If $y_i$ is non-negative, we use the formula
$\dot{e}_{n+r} = \dot{e}_{n-r}$ to get 
\begin{multline*}
\dot{e}_{
n - |Ky_{i} + {i}| + [y_{i}\ge0] -1+ j
} +
\dot{e}_{
n - |Ky_{i} + {i}| + [y_{i}\ge0] - j 
}\\=\dot{e}_{
n - Ky_{i} - i + j 
} +
\dot{e}_{
n - Ky_{i} - i - j + 1
}=\dot{e}_{
n + Ky_{i} + i - j 
} +
\dot{e}_{
n + Ky_{i} + i + j - 1
}
\end{multline*}
If $y_i$ is negative, we immediately have that
\begin{equation*}
\dot{e}_{
n - |Ky_{i} + {i}| + [y_{i}\ge0] -1+ j
} +
\dot{e}_{
n - |Ky_{i} + {i}| + [y_{i}\ge0] - j 
}=\dot{e}_{
n + Ky_{i} + i + j-1
} +
\dot{e}_{
n + Ky_{i} + i - j \,.
}
\end{equation*}
Replacing $y_i$ by $-y_i$ and transposing our matrix,
we have arrived at
    \[\sum_{(y_1, \ldots, y_k)\in\Z^k}
t^{\sum_{j=1}^k \left( \frac{1}{2} K y_j^2 - (j-\frac{1}{2})y_j \right)}
\det_{1 \le i,j \le k}
\left(
\dot{e}_{\,n-i+j-K y_j}
+
\dot{e}_{\,n+i+j-Ky_j-1}
\right),\]
where we set $K=2k$. By Lemma \ref{lem: handling the summation range}, the support of this sum is finite, so extending the sum to a finite $S_k$ invariant subset containing the support, by Lemma~\ref{lem:powery}, we get that this is equal to the right-hand side of~\eqref{thm-type-b-t}.
\end{proof}

\subsection{The case \texorpdfstring{$s=1$}{s=1}.}
We conclude the section related to type $\mathrm{B}_n$ with the third of of three formulae \eqref{thm-type-b-1-char} of Theorem~\ref{thm-b-char}, that is
\[
P_{(k^n)}^{\mathrm{B}_n}(X;t,1)=\sum_{y \in \mathbb{Z}^k}
(-1)^{\inv(\nu)+|y|}
t^{\sum_{j=1}^k \left( \frac{1}{2}K y_j^2 + (j-1)y_j \right)}
\oe_{\lambda(\nu)}(X),\]
where $K=2k$ and $\nu_a = |Ky_a + a - 1|$ for $a=1,\dots,k$ and $\lambda(\nu)$ is defined in Definition~\ref{de:lambda(nu)}. 
\begin{proof}[Proof of ~\cref{thm-type-b-1-char}] 
Set $s_1=1$ and $s_2=-1$ in the expression of Lemma~\ref{magical manipulation} to get
    \begin{multline*}
    P_{(k^n)}^{\mathrm{B}_n}(X;t,1)= \frac{\sqrt{t}^{\binom{n}{2}}}{2}
    \prod_{i=1}^nx_i^{n-1}\prod_{1\leq i<j\leq n}\frac{1}{(1-x_ix_j)(x_i-x_j)}\\\times\det_{1\le i,j\le n}\left(
    (\sqrt{t})^{j-1}\left(x_i^{j-1-k}+x_i^{-j+1+k}\right)+
    (\sqrt{t})^{-j+1}\left(x_i^{-j+1-k}+x_i^{j-1+k}\right)\right).
\end{multline*}
We identify the Weyl denominator of type \ref{eq: Denom type D} up to a sign $(-1)^{\binom{n}{2}}$ to write this as
\begin{multline*}
    (-\sqrt{t})^{\binom{n}{2}}\det_{1\le i,j\le n}\left(
    (\sqrt{t})^{j-1}\left(x_i^{j-1-k}+x_i^{-j+1+k}\right)+
    (\sqrt{t})^{-j+1}\left(x_i^{-j+1-k}+x_i^{j-1+k}\right)\right)\\\times 1/\det_{1\leq i,j\leq n}\left(x_i^{n-j}+\x^{n-j}\right).
\end{multline*}
Using the shorthand notation $\zeta_i(m)$ for the expression $x_i^{m} + \x_i^{m}$, we observe that the first column of our matrix in the numerator is given by $2\zeta_i(k)$. Factoring out this $2$ by linearity, we can rewrite our expression as
\begin{equation*}
    P_{(k^n)}^{\mathrm{B}_n}(X;t,1)=2(-\sqrt{t})^{\binom{n}{2}}
  \frac{
  \det\limits_{\substack{1 \le i \le n\\ 1 \le j < n}}\left(\matrixCols{
  \zeta_i(k)
  }{
  (\sqrt{t})^{j}\zeta_i(k-j) +(\sqrt{t})^{-j}\zeta_i(k+j)
  }
  \right)
  }{\det\limits_{1 \le i,j \le n}(\zeta_i(n-j))}.
\end{equation*}

We use Lemma~\ref{lem: noodle-doodle expansion 2.1c} with $\zeta_i(a)=x_i^{a} + \x_i^{a}$, $p_{\pm}(j) = (\sqrt{t})^{\mp j}$ and $K=2k$ to obtain 
\begin{multline*}
    \det\limits_{\substack{1 \le i \le n\\ 1 \le j < n}}\left(\matrixCols{
  \zeta_i(k)
  }{
  (\sqrt{t})^{j}\zeta_i(k-j) +(\sqrt{t})^{-j}\zeta_i(k+j)
  }
  \right)\\=\sum_{\substack{(y_0,\dots,y_{ k}) \in\Z^{k +1} \\ y_0,y_k \geq 0}}
        (-1)^{|y|+\inv(\vec z_{0,k})} t^{-\frac{1}{2}\binom{n}{2} +\sum_{i=0}^{k} \left( \frac{1}{2}Ky_i^2 + y_ii \right)}
    \det_{\substack{1\le i\le n \\0 \le j\le k+n-1 \\ j \not=z(y_r), r=0,\dots,k}}(\matrixCols{\zeta_i(k)}{\zeta_i(j)}),
\end{multline*}
where $z(y_j)=\alpha_j(Ky_j+j)$ and where we denote $\vec z_{0,k}\coloneqq(z(y_0),\dots z(y_{k}))$. This is because similarly as in the previous proofs in this section we have $\binom{k+1}{2} - \sum_{j=0}^{k} b_j = \inv(\vec z_{0,k})$, and 
\begin{equation*}
     \prod_{j=1}^{n-1}p_+(j) 
        \prod_{i=0}^{k}
        \prod_{\ell=0}^{\alpha_iy_i-1}
        \frac{p_-(K\ell+k+\alpha_i i)}{p_+(K\ell+k+\alpha_i i)}=t^{-\frac{1}{2}\binom{n}{2} +\sum_{i=0}^{k} \left( \frac{1}{2}Ky_i^2 + y_ii \right)},
\end{equation*}
as the following simple calculations shows:
\begin{multline*}
-\sum_{j=1}^{n-1}\frac{j}{2}+\sum_{i=0}^k\sum_{\ell=0}^{\alpha_iy_i-1}\left(K\ell+k+\alpha_i i\right)=-\frac{1}{2}\binom{n}{2} + \sum_{i=0}^{k} K\binom{|y_i|}{2} + |y_i|k +y_ii \\= 
-\frac{1}{2}\binom{n}{2} +\sum_{i=0}^{k} \left( \frac{1}{2}Ky_i^2 + y_ii \right).
\end{multline*}
Inserting this into our expression from before yields that
\begin{equation*}
    P_{(k^n)}^{\mathrm{B}_n}(X;t,1)=2\sum_{\substack{(y_0,\dots,y_{ k}) \in\Z^{k +1} \\ y_0,y_k \geq 0}}
        (-1)^{\binom{n}{2}+\inv(\vec z_{0,k})+|y|} t^{\sum_{i=0}^{k} \left( \frac{1}{2}Ky_i^2 + y_ii \right)}
    \frac{\det\limits_{\substack{1\le i\le n \\0 \le j\le k+n-1 \\ j \not=z(y_r), r=0,\dots,k}}(\matrixCols{\zeta_i(k)}{\zeta_i(j)})}{\det\limits_{1 \le i,j \le n}(\zeta_i(n-j))}.
\end{equation*}

Now consider the determinant $$\det\limits_{\substack{1\le i\le n \\0 \le j\le k+n-1 \\ j \not=z(y_r), r=0,\dots,k}}(\matrixCols{\zeta_i(k)}{\zeta_i(j)}),$$
in the numerator. The first column of the determinant is $\zeta_i(k)$. However, if $z(y_i)\ne k$ for all $i$ then there is a second column of the same form, and hence the determinant vanishes. An elementary algebraic manipulation reveals that $z(y_i) = k$ if and only if $i=k$ and $y_k\in\{0,-1\}$. Since $y_k\ge0$, we have only one option and may assume, from now on, that $y_k=0$. Note that then
\begin{equation*}
    \inv(\vec z_{0,k})=\inv((z(y_0),\dots,z(y_{k-1}),k)=\inv(\vec z_{0,k-1})+|\{j\in[0,k-1]\mid k<z(y_j)\}|.
\end{equation*}
Now reorder the columns of the determinant to be in increasing order with respect to the argument of $\zeta_i$. The only misplaced column is the first. This introduces a sign of $$(-1)^{k-|\{j\in[0,k-1]\mid z(y_j)< k\}|}=(-1)^{|\{j\in[0,k-1]\mid k<z(y_j)\}|},$$
as there are $k-|\{j\in[0,k-1]\mid z(y_j)< k\}|$ many columns after the first column of our matrix where the argument of $\zeta_i$ is smaller than $k$. Thus we have arrived at
\begin{equation*}
    P_{(k^n)}^{\mathrm{B}_n}(X;t,1)=2\sum_{\substack{(y_0,\dots,y_{k-1}) \in\Z^{k } \\ y_0 \geq 0}}
        (-1)^{\binom{n}{2}+\inv(\vec z_{0,k-1})+|y|} t^{\sum_{i=0}^{k-1} \left( \frac{1}{2}Ky_i^2 + y_ii \right)}
    \frac{\det\limits_{\substack{1\le i\le n \\0 \le j\le k+n-1 \\ j \not=z(y_r), r=0,\dots,k-1}}(\zeta_i(j))}{\det\limits_{1 \le i,j \le n}(\zeta_i(n-j))}.
\end{equation*}

To be able to use the Weyl character formula of type $\mathrm{D}_n$ \eqref{eq: Weyl character formula type D} and Corollary~\ref{cor:useful.complementary}, we define a vector $\nu\coloneqq(\nu_1\dots,\nu_k)$ by
\[
\nu_i := z(y_{i-1})=\alpha_{i-1}(Ky_{i-1}+i-1) = |Ky_{i-1}+i-1|
\]
for $i=1,\ldots,k$
and use $\inv(\vec z_{0,k-1}) = \inv(\nu)$
to reindex the sum by simultaneously setting $y_i\rightarrow y_{i+1}$ for $i=0,\dots k-1$ and write it as a sum over tuples $(y_1,\ldots,y_k)\in \Z^k$. We obtain
\begin{equation*}
    P_{(k^n)}^{\mathrm{B}_n}(X;t,1)
  =  
  2\sum_{\substack{(y_1, \ldots, y_k)\in\Z^k\\ y_1\ge0}}
    (-1)^{\binom{n}{2} + \inv(\nu) + |y|}
    t^{\sum_{i=1}^{k} \left( \frac{1}{2}Ky_i^2 + y_i(i-1) \right)}
  \frac{\det\limits_{\substack{1\le i\le n \\0 \le j\le k+n-1 \\ j \not=\nu_r, r=1,\dots,k}}(\zeta_i(j))
  }{\det\limits_{1 \le i, j \le n}(\zeta_i(n-j))}.
\end{equation*}
By the Weyl character formula of type $\mathrm{D}_n$ \eqref{eq: Weyl character formula type D} and by Corollary~\ref{cor:useful.complementary}, we see that
$$(-1)^{\binom{n}{2}}\frac{\det\limits_{\substack{1\le i\le n \\0 \le j\le k+n-1 \\ j \not=\nu_r, r=1,\dots,k}}(\zeta_i(j))
  }{\det\limits_{1 \le i, j \le n}(\zeta_i(n-j))}=\frac{(1+[\lambda(\nu)_n=0])}{2}\oe_{\lambda(\nu)}(X).$$
  Note that the condition for the matrix in the numerator to be a square matrix is precisely the condition that the $y_a$'s corresponding $\nu_a$'s are inside $\{0,\dots,n+k-1\}$ and are all different.
Plugging this into our expression from before we obtain
\[
    P_{(k^n)}^{\mathrm{B}_n}(X;t,1)
  =  
  \sum_{\substack{(y_1, \ldots, y_k)\in\Z^k\\ y_1\ge0}}
  (1+[\lambda(\nu)_n=0])
    (-1)^{\inv(\nu) + |y|}
    t^{\sum_{i=1}^{k} \left( \frac{1}{2}Ky_i^2 + y_i(i-1) \right)}
  \oe_{\lambda(\nu)}(X),
\]
where we sum over all $y\in\Z^k$ such that the corresponding $\nu_a$'s are inside $\{0,\dots,n+k-1\}$ and are all different. 

To finish the proof, we need to make use of a symmetry property of our summands. We first claim that $\lambda(\nu)_n=0$ if $y_1\neq 0$. By unravelling Lemma~\ref{lem:flip}, we have
\[\lambda(\nu)_n = k - \ell(\mathrm{sort}(\nu)-\rho_k)
  = k - |\{i\in[k]\mid\mathrm{sort}(\nu)_i > k-i\}|.
\]
Now note that $\mathrm{sort}(\nu)$ is a strict partition, because $(\nu_1=z(y_0),\dots,\nu_k=z(y_{k-1}))$ are all located on different connected components, hence are all different. Therefore the only way how $\ell(\mathrm{sort}(\nu)-\rho_k)\neq k$ and subsequently $\lambda(\nu)_n\neq 0$ is when $\nu_1=0$, which, if we assume $y_1\geq 0$, is the case if and only if $y_1=0$. Thus $\lambda(\nu)_n=0$ unless $y_1=0$. 

Denoting $$S(y_1,\ldots,y_k) \coloneqq (-1)^{\inv(\nu) + |y|}
    t^{\sum_{i=1}^{k} \left( \frac{1}{2}Ky_i^2 + y_i(i-1) \right)}
  \oe_{\lambda(\nu)}(X),$$ we get the following symmetry property of our summands:
$$S(-y_1,y_2,\ldots,y_k) = S(y_1,y_2,\ldots,y_k).$$
This follows from the fact that $y_1$, the previous $y_0$, only contributes to $\nu_1=z(y_0)$ and that $z(y_0)=\alpha_0Ky_0=z(-y_0)$, thus $\lambda(\nu)$ is symmetric in $y_1$. The symmetry of the exponent of $t$ follows directly from the fact that the $i=1$ summand in the exponent of $t$ is given by $\frac{1}{2}Ky_1^2$ and clearly symmetric in $y_1$.

Combining these two observations we see that
\begin{align*}
    P_{(k^n)}^{\mathrm{B}_n}(X;t,1)
  &=  
  \sum_{\substack{(y_1, \ldots, y_k)\in\Z^k\\ y_1\ge0}}
  (1+[\lambda(\nu)_n=0]) S(y_1,y_2,\ldots,y_k)\\
  &=  
  \sum_{\substack{(y_1, \ldots, y_k)\in\Z^k\\ y_1\ge0}}
  (1+[y_1\ne 0]) S(y_1,y_2,\ldots,y_k)\\
  &=  
  \sum_{\substack{(y_1, \ldots, y_k)\in\Z^k\\ y_1\ge0}}
  S(y_1,y_2,\ldots,y_k)
+
  \sum_{\substack{(y_1, \ldots, y_k)\in\Z^k\\ y_1 > 0}}
  S(y_1,y_2,\ldots,y_k)\\  &=  
  \sum_{\substack{(y_1, \ldots, y_k)\in\Z^k\\ y_1\ge0}}
  S(y_1,y_2,\ldots,y_k)
+
  \sum_{\substack{(y_1, \ldots, y_k)\in\Z^k\\ y_1 < 0}}
  S(y_1,y_2,\ldots,y_k)\\  &=  
  \sum_{(y_1, \ldots, y_k)\in\Z^k}
  S(y_1,y_2,\ldots,y_k)\\&=\sum_{(y_1, \ldots, y_k)\in\Z^k}(-1)^{\inv(\nu) + |y|}
    t^{\sum_{i=1}^{k} \left( \frac{1}{2}Ky_i^2 + y_i(i-1) \right)}
  \oe_{\lambda(\nu)}(X),
\end{align*}
which is the desired expression.
\end{proof}

Having finished the proof of \eqref{thm-type-b-1-char} we are ready to prove \eqref{thm-type-b-1}.
Recall that the dual Jacobi--Trudi formula for $P^{\mathrm{B}_n}_{(k^n)}(X;t,1)$ stated in \eqref{thm-type-b-1} is
\[
     P_{(k^n)}^{\mathrm{B}_n}(X;t,1)=
     \frac{1}{2}
\sum_{y \in \mathbb{Z}^k}
\det_{1 \le i,j \le k}
\left(
(-1)^{y_i}
t^{\frac{1}{2}K y_i^2 - (j-1)y_i}
(
\dot{e}_{\,n-i+j-K y_i}
+
\dot{e}_{\,n+i+j- Ky_i-2})
\right)
\]
where $K=2k$.
\begin{proof}[Proof of ~\cref{thm-type-b-1}] 
Let $\nu_a = |Ky_a + a-1|$ as in~\eqref{thm-type-b-1-char}. 
Unravelling the definitions preceding Lemma~\ref{lem:flip}, we have
\begin{multline*}
\lambda(\nu)'_i = (\mathrm{sort}(\nu) - \rho_k)^{\square n^k}_i = 
n - (\mathrm{sort}(\nu) - \rho_k)_{k+1-i} =
n - \mathrm{sort}(\nu)_{k+1-i} + i -1\\
= n - |Ky_{\sigma(i)} + \sigma(i)-1| + i-1\,,
\end{multline*}
where $\sigma$ is the permutation that sorts $\nu$ into increasing order.

We apply the dual Jacobi--Trudi identity of type $\mathrm{D}_n$ \eqref{eq: dual JT type D} to the formula of ~\eqref{thm-type-b-1-char} to get
\begin{multline*}
P_{(k^n)}^{\mathrm{B}_n}(X;t,1)\\=\frac{1}{2} 
\sum_{{(y_1, \ldots, y_k)\in\Z^k}} 
(-1)^{|y|+\mathrm{inv}(\nu)}
t^{\sum_{j=1}^k \left( \frac{1}{2}K y_j^2 + (j-1)y_j \right)}
\det_{1\leq i,j\leq k}( \dot{e}_{\lambda(\nu)'_i-i+j}
+ 
\dot{e}_{\lambda(\nu)'_i-i-j+2}).
\end{multline*}
By the assumption in Theorem \ref{thm-b-char}, the sum ranges over all vectors $y\in\Z^k$ such that $\nu_1,\dots,\nu_k\in[0,n+k-1]$ are all different elements. By Lemma \ref{lem: handling the summation range} the determinant on the right hand side of this sum vanishes outside of this range of summation, hence we can extend the sum, to a sum over all $y\in\Z^k$. 

Plugging in the expression for $\lambda(\nu)'_i$ that we got before, this determinant is
\[
\det_{1\leq i,j\leq k}\left(
\dot{e}_{
n - |Ky_{\sigma(i)} + \sigma(i) - 1|  + j - 1
} +
\dot{e}_{
n - |Ky_{\sigma(i)} + \sigma(i) - 1| - j + 1
}\right).
\]
Now we rearrange the rows of the matrix by the permutation $\sigma^{-1}$. This cancels the factor $\sgn(\sigma) = (-1)^{\mathrm{inv}(\nu)}$, as $\sigma$ is the permutation that sorts $\nu$ in increasing order, hence the sign is given by the number of inversions of $\nu$. The matrix becomes
\[
\det_{1\leq i,j\leq k}\left(
\dot{e}_{
n - |Ky_{i} + i - 1| + j - 1
} +
\dot{e}_{
n - |Ky_{i} + i - 1| - j + 1
}\right).
\]
We distinguish cases according to the sign of $y_i$ and fix some $1\leq i,j\leq k$.
If $y_i$ is non-negative, we use the formula
$\dot{e}_{n+r} = \dot{e}_{n-r}$ to get 
\begin{multline*}
\dot{e}_{
n - |Ky_{i} + i - 1| + j - 1
} +
\dot{e}_{
n - |Ky_{i} + i - 1| - j + 1
}\\ = 
\dot{e}_{
n - Ky_{i} - i + j 
} +
\dot{e}_{
n - Ky_{i} - i - j + 2
} =
\dot{e}_{
n + Ky_{i} + i - j 
} +
\dot{e}_{
n + Ky_{i} + i + j - 2
}\,.
\end{multline*}
If $y_i$ is negative, we immediately obtain
\begin{equation*}
\dot{e}_{
n - |Ky_{i} + i - 1| + j - 1
} +
\dot{e}_{
n - |Ky_{i} + i - 1| - j + 1
} = 
\dot{e}_{
n + Ky_{i} + i + j - 2
} +
\dot{e}_{
n + Ky_{i} + i - j
}.
\end{equation*}
Replacing $y_i$ by $-y_i$, and transposing the matrix
we have arrived at
    \[
    \frac{1}{2}
\sum_{{(y_1, \ldots, y_k)\in\Z^k}}
(-1)^{|y|}
t^{\sum_{j=1}^k \left( \frac{1}{2}K y_j^2 - (j-1)y_j \right)}
\det_{1 \le i,j \le k}
\left(
\dot{e}_{\,n-i+j-K y_j}
+
\dot{e}_{\,n+i+j- Ky_j-2}
\right).\]
By Lemma \ref{lem: handling the summation range}, the support of this sum is finite, so extending the sum to a finite $S_k$ invariant subset containing the support, by Lemma~\ref{lem:powery}, we get that this is equal to the right-hand side of~\eqref{thm-type-b-1}.
\end{proof}

\section{Type \texorpdfstring{$\mathrm{C}_n$}{Cn}}
\label{sec:type-c}
In this section, we provide proof of all four formulae related to type $\mathrm{C}_n$ presented in Theorem \ref{main-type-c} and Corollary \ref{cor-type-c}. We present the proofs of the formulae grouped by the parameter specialisations of $t,s$ in the $\mathrm{C}_n$ Hall--Littlewood polynomials. 

\subsection{The case \texorpdfstring{$s=0$}{s=0}.}
We begin with the first of the two formulae \eqref{thm-type-c-0-char} presented in Theorem \ref{main-type-c}, that is
\[
P_{(k^n)}^{\mathrm{C}_n}(X;t,0)=
\sum_{y \in \mathbb{Z}^{k}}
(-1)^{\mathrm{inv}(\nu)+\sum_{j=1}^k [y_j<0]}
t^{\sum_{j=1}^k \left(\frac{1}{2}Ky_j^2 + jy_j \right)}
\sp_{\lambda(\nu)}(X),\]
where $K=2k+2$ and $\nu_a = |Ky_a + a|-1$ for $a=1,\dots,k$ and $\lambda(\nu)$ is defined in Definition~\ref{de:lambda(nu)}. 
\begin{proof}[Proof of ~\cref{thm-type-c-0-char}] 
Setting $s_1=0$ and $s_2=0$ in the expression of Lemma~\ref{magical manipulation} we get 

\begin{multline*}
    P_{(k^n)}^{\mathrm{C}_n}(X;t,0)= \frac{\sqrt{t}^{\binom{n}{2}}}{2}
    \prod_{i=1}^nx_i^{n-1}\prod_{1\leq i<j\leq n}\frac{1}{(1-x_ix_j)(x_i-x_j)}\\\times\det_{1\le i,j\le n}\left(\frac{1}{1-x_i^2}
    \left((\sqrt{t})^{j-1}x_i^{j-1-k}+
    (\sqrt{t})^{-j+1}x_i^{-j+1-k}\right)\right.\\\left.-\frac{x_i^2}{1-x_i^2}
    \left((\sqrt{t})^{j-1}x_i^{-j+1+k}
    +(\sqrt{t})^{-j+1}x_i^{j-1+k}\right)\right)\\=
    \frac{\sqrt{t}^{\binom{n}{2}}}{2}
    \prod_{i=1}^n\frac{x_i^{n-1}}{x_i^2-1}\prod_{1\leq i<j\leq n}\frac{1}{(1-x_ix_j)(x_i-x_j)}\\\times\det_{1\le i,j\le n}
    \left((\sqrt{t})^{j-1}\left(x_i^{-j+3+k}-x_i^{j-1-k}\right)+
    (\sqrt{t})^{-j+1}\left(x_i^{j+1+k}-x_i^{-j+1-k}\right)\right).
\end{multline*}
 We identify the Weyl denominator of type \ref{eq: Denom type C} up to a factor of $\prod_{i=1}^nx_i^{-1}$, which we distribute over the determinant, and a sign $(-1)^{\binom{n}{2}}$, to write our expression as
 \begin{multline*}
     \frac{(-\sqrt{t})^{\binom{n}{2}}}{2}
    \det_{1\le i,j\le n}
    \left((\sqrt{t})^{j-1}\left(x_i^{-j+2+k}-x_i^{j-2-k}\right)+
    (\sqrt{t})^{-j+1}\left(x_i^{j+k}-x_i^{-j-k}\right)\right)\\\times 1/\det_{1\leq i,j\leq n}\left(x_i^{n-j+1}-\x^{n-j+1}\right).
 \end{multline*}
Using our shorthand notation $\eta_i(m)$ for the expression $ x_i^{m} - \x_i^{m}$ from before, we observe that the first column is given by $2\eta_i(k+1)$. Factoring out this $2$ from the first column by the multilinearity of the determinant, we can write the above as
 \[
 P_{(k^n)}^{\mathrm{C}_n}(X;t,0)
  =  
  (-\sqrt{t})^{\binom{n}{2}}
  \frac{\det\limits_{\substack{1 \le i \le n\\ 1 \le j < n}}(
  \matrixCols{\eta_i(k+1)}{(\sqrt{t})^{-j}\eta_i(k+1+j) +(\sqrt{t})^{j}\eta_i(k+1-j)}
  )}{\det\limits_{1 \le i, j \le n}(\eta_i(n-j+1))}\,.
 \]
We invoke Lemma~\ref{lem: noodle-doodle expansion 2.3a} with $\eta_i(a)=x_i^{a} - \x_i^{a}$, $p_{\pm}(j) = (\sqrt{t})^{\mp j}$ and $K=2k+2$ to obtain that
\begin{multline*}
    \det\limits_{\substack{1 \le i \le n\\ 1 \le j < n}}(
  \matrixCols{\eta_i(k+1)}{(\sqrt{t})^{-j}\eta_i(k+1+j) +(\sqrt{t})^{j}\eta_i(k+1-j)}
  )\\=\sum_{\substack{(y_0,\dots,y_{k+1}) \in\Z^{k +2} \\ y_0,y_{k+1} \geq 0}}
        (-1)^ {\sum_{i=0}^{k+1} [y_i<0] + \inv(\vec z_{0,k+1})}t^{-\frac{1}{2}\binom{n}{2} +\sum_{i=0}^{k+1} \left( \frac{1}{2}Ky_i^2 + y_ii \right)}
        \\\times \det_{\substack{1\le i\le n \\0 \le j\le k+n \\ j \not=z(y_r), r=0,\dots,k+1}}(\matrixCols{\eta_i(k+1)}{\eta_i(j)}),
\end{multline*}
where $z(y_j)=\alpha_j(Ky_j+j)$ and we denote $\vec z_{0,k+1}\coloneqq(z(y_0),\dots,z(y_{k+1}))$. This is because as in the proofs of the previous section we have $\binom{k+2}{2} - \sum_{j=0}^{k+1} b_j = \inv(\vec z_{0,k+1})$ and furthermore
\begin{equation*}
    \prod_{j=1}^{n-1}p_+(j) 
        \prod_{i=0}^{k+1}
        \prod_{\ell=0}^{\alpha_iy_i-1}
        \frac{p_-(K\ell+k+1+\alpha_i i)}{p_+(K\ell+k+1+\alpha_i i)}=t^{-\frac{1}{2}\binom{n}{2} +\sum_{i=0}^{k+1} \left( \frac{1}{2}Ky_i^2 + y_ii \right)},
\end{equation*}
as the following simple calculation shows
\begin{multline*}
    -\sum_{j=1}^{n-1}\frac{j}{2}+\sum_{i=0}^{k+1}\sum_{\ell=0}^{\vert y_i\vert-1}(K\ell+k+1+\alpha_i i)=-\frac{1}{2}\binom{n}{2} + \sum_{i=0}^{k+1} \left(K\binom{|y_i|}{2} + |y_i|(k+1) +y_ii\right) \\= 
-\frac{1}{2}\binom{n}{2} +\sum_{i=0}^{k+1} \left( \frac{1}{2}Ky_i^2 + y_ii \right).
\end{multline*} Inserting this into our expression from before we get that
\begin{multline*}
    P_{(k^n)}^{\mathrm{C}_n}(X;t,0)=\sum_{\substack{(y_0,\dots,y_{k+1}) \in\Z^{k +2} \\ y_0,y_{k+1} \geq 0}}
        (-1)^ {\binom{n}{2}+\sum_{i=0}^{k+1} [y_i<0] + \inv(\vec z_{0,k+1})}t^{ \sum_{i=0}^{k+1} \left( \frac{1}{2}Ky_i^2 + y_ii \right)}
        \\\times \frac{\det\limits_{\substack{1\le i\le n \\0 \le j\le k+n \\ j \not=z(y_r), r=0,\dots,k+1}}(\matrixCols{\eta_i(k+1)}{\eta_i(j)})}{\det\limits_{1 \le i, j \le n}(\eta_i(n-j+1))}.
\end{multline*}

Now consider the determinant
$$\det\limits_{\substack{1\le i\le n \\0 \le j\le k+n \\ j \not=z(y_r), r=0,\dots,k+1}}(\matrixCols{\eta_i(k+1)}{\eta_i(j)}),$$
from the numerator and note that if $y_{k+1}\ne 0$, then the determinant has two equal columns and thus vanishes. Therefore, we can assume $y_{k+1}=0$ and so $z(y_{k+1})=k+1$ from now on. If $y_0\ne0$, then the matrix has a column of zeros and thus vanishes, so we can assume $y_0=0$ and $z(y_0)=0$ from now on too. Note that then
\begin{multline*}
    \inv(\vec z_{0,k+1})=\inv(0,z(y_1),\dots,z(y_k),k+1)\\=\inv(\vec z_{1,k})+\vert\{j\in[1,k]\mid z(y_j)<0\}+\vert\{j\in[1,k]\mid k+1<z(y_j)\}\\=\inv(\vec z_{1,k})+\vert\{j\in[1,k]\mid k+1<z(y_j)\},
\end{multline*}
as $z(y_j)\geq 0$ for all $1\leq j\leq k$. We reorder the columns of the matrix to be in increasing order with respect to the argument of $\eta_i$. The only misplaced column is the first one, and this introduces a sign of 
$$(-1)^{k-\vert\{j\in[1,k]\mid z(y_j)<k+1\}\vert}=(-1)^{\vert\{j\in[1,k]\mid k+1<z(y_j)\}},$$
as there are $k-\vert\{j\in[1,k]\mid z(y_j)<k+1\}\vert$ many columns after the first column of our matrix where the argument of $\eta_i$ is smaller than $k+1$. Thus we have arrived at
\begin{equation*}
    P_{(k^n)}^{\mathrm{C}_n}(X;t,0)=\sum_{(y_1,\dots,y_{k}) \in\Z^{k}}
        (-1)^ {\binom{n}{2}+ \inv(\vec z_{1,k})+\sum_{j=1}^k [y_j<0]}t^{ \sum_{i=1}^{k} \left( \frac{1}{2}Ky_i^2 + y_ii \right)} \frac{\det\limits_{\substack{1\le i\le n \\1 \le j\le k+n \\ j \not=z(y_r), r=1,\dots,k}}(\eta_i(j))}{\det\limits_{1 \le i, j \le n}(\eta_i(n-j+1))}.
\end{equation*}

To be able to use the Weyl character formula of type $\mathrm{C}_n$ \eqref{eq: Weyl character formula type C} and Corollary~\ref{cor:useful.complementary}, we define a vector $\nu\coloneqq (\nu_1,\dots,\nu_k)$ by
$$\nu_i\coloneqq z(y_i)-1=\alpha_i(Ky_i+i)-1=\vert Ky_i+i\vert-1$$
and observe that $\inv(\vec z_{1,k})=\inv(\nu)$. This allows us to rewrite our expression as

\[
    P_{(k^n)}^{\mathrm{C}_n}(X;t,0)=\sum_{(y_1,\ldots,y_{k})\in\Z^{k}}
    (-1)^{\binom{n}{2}+ \inv(\nu)+\sum_{j=1}^k [y_j<0] }
    t
    ^{\sum_{i=1}^{k} \left(\frac{1}{2}Ky_i^2 + y_ii \right)}
    \frac{\det\limits_{\substack{1\le i\le n \\1 \le j\le k+n \\ j \not=\nu_r+1, r=1,\dots,k}}(\eta_i(j))}{\det\limits_{1 \le i,j \le n}(\eta_i(n-j+1))}  
\]

By the Weyl character formula of type $\mathrm{C}_n$ \eqref{eq: Weyl character formula type C} and Corollary~\ref{cor:useful.complementary} we have that
$$(-1)^{\binom{n}{2}}\frac{\det\limits_{\substack{1\le i\le n \\1 \le j\le k+n \\ j \not=\nu_r+1, r=1,\dots,k}}(\eta_i(j))}{\det\limits_{1 \le i,j \le n}(\eta_i(n-j+1))}=\sp_{\lambda(\nu)}(X).$$
Note that the condition for the matrix in the numerator to be a square matrix is precisely the condition that the $y_a$'s corresponding $\nu_a$'s are inside $\{0,\dots,n+k-1\}$ and are all different. Inserting into our expression from above yields
\[
P_{(k^n)}^{\mathrm{C}_n}(X;t,0)
= 
    \sum_{(y_1,\ldots,y_{k})\in\Z^{k}}
    (-1)^{ \inv(\nu)+\sum_{j=1}^k [y_j<0]}
    t
    ^{\sum_{i=1}^{k} \left( \frac{1}{2}Ky_i^2 + y_ii \right)}\sp_{\lambda(\nu)}(X), 
\]
where we sum over all $y\in\Z^k$ such that the corresponding $\nu_a$'s are inside $\{0,\dots,n+k-1\}$ and are all different, as claimed.
\end{proof}

Having finished the proof of \eqref{thm-type-c-0-char}, we can now show the dual Jacobi--Trudi formula~\eqref{cor-type-c-0} for this specialisation, 
    \[
P_{(k^n)}^{\mathrm{C}_n}(X;t,0)=\sum_{y \in \mathbb{Z}^k}
\det_{1 \le i,j \le k}
\left(
t^{\frac{1}{2} K y_i^2 - jy_i}
\left(
\dot{e}_{\,n-i+j-K y_i}
-
\dot{e}_{\,n+i+j- Ky_i}
\right)
\right),
\]
where $K = 2k+2$.
\begin{proof}[Proof of ~\cref{cor-type-c-0}]
  Let $\nu_a = |Ky_a + a| - 1$ as in~\eqref{thm-type-c-0-char}. 
Unravelling the definitions preceding Lemma~\ref{lem:flip} we have
\begin{multline*}
\lambda(\nu)'_i = (\mathrm{sort}(\nu) - \rho_k)^{\square n^k}_i = 
n - (\mathrm{sort}(\nu) - \rho_k)_{k+1-i} =
n - \mathrm{sort}(\nu)_{k+1-i} + i -1\\
= n - |Ky_{\sigma(i)} + \sigma(i)| + i\,,
\end{multline*}
where $\sigma$ is the permutation that sorts $\nu$ into increasing order.

We apply the dual Jacobi--Trudi identity of type $\mathrm{C}_n$ \eqref{eq: dual JT type C} to the formula of ~\eqref{thm-type-c-0-char} to get
\[
P_{(k^n)}^{\mathrm{C}_n}(X;t,0)=
\sum_{y \in \mathbb{Z}^k}
(-1)^{\mathrm{inv}(\nu)+\sum_{j=1}^k [y_j<0]}
t^{\sum_{j=1}^k \left( \frac{1}{2}Ky_j^2 + jy_j \right)}
\det_{1\leq i,j\leq k}(\dot{e}_{\lambda(\nu)'_i-i+j}
- 
\dot{e}_{\lambda(\nu)'_i-i-j}).
\]
By the assumption in Theorem \ref{main-type-c}, the sum ranges over all vectors $y\in\Z^k$ such that $\nu_1,\dots,\nu_k\in[0,n+k-1]$ are all different elements. By Lemma \ref{lem: handling the summation range} the determinant on the right hand side of this sum vanishes outside of this range of summation, hence we can extend the sum, to a sum over all $y\in\Z^k$. 

Plugging in the expression for $\lambda(\nu)'_i$, from before this determinant is equal to
\[
\det_{1\leq i,j\leq k}(
\dot{e}_{
n - |Ky_{\sigma(i)} + \sigma(i)|  + j
} -
\dot{e}_{
n - |Ky_{\sigma(i)} + \sigma(i)| - j
}).
\]
Now we rearrange the rows of the matrix by the permutation $\sigma^{-1}$. This cancels the factor $\sgn(\sigma) = (-1)^{\mathrm{inv}(\nu)}$, as $\sigma$ is the permutation that sorts $\nu$ in increasing order, hence the sign is given by the number of inversions of $\nu$. The matrix becomes
\[
\det_{1\leq i,j\leq k}(
\dot{e}_{
n - |Ky_{i} + i| + j 
} -
\dot{e}_{
n - |Ky_{i} + i| - j 
}).
\]
We distinguish cases according to the sign of $y_i$ and fix some $1\leq i,j\leq k$.
If $y_i$ is non-negative, we use the formula $\dot{e}_{n+r}=\dot{e}_{n-r}$  to get
\begin{equation*}\dot{e}_{
n - |Ky_{i} + i| + j 
} -
\dot{e}_{
n - |Ky_{i} + i| - j 
}= 
\dot{e}_{
n - Ky_{i} - i + j
} -
\dot{e}_{
n - Ky_{i} - i - j
} =
\dot{e}_{
n + Ky_{i} + i - j
} -
\dot{e}_{
n + Ky_{i} + i + j 
}\,.
\end{equation*}
If $y_i$ is negative, one directly gets
\begin{equation*}
\dot{e}_{
n - |Ky_{i} + i| + j 
} -
\dot{e}_{
n - |Ky_{i} + i| - j 
}=
-(
\dot{e}_{
n + Ky_{i} + i - j
} -
\dot{e}_{
n + Ky_{i} + i + j 
})\,.
\end{equation*}
Pulling this sign out of the determinant using the multilinearity we obtain an additional factor of $(-1)^{\sum_{j=1}^k [y_j<0]}$. Replacing $y_i$ by $-y_i$ and then transposing the matrix yields
\[
P_{(k^n)}^{\mathrm{C}_n}(X;t,0)=
\sum_{y \in \Z^k}
t^{\sum_{j=1}^k \left( \frac{1}{2}K y_j^2 - jy_j \right)}
\det_{1 \le i,j \le k}
\left(
\dot{e}_{\,n-i+j-K y_j}
-
\dot{e}_{\,n+i+j- Ky_j}
\right).\]
By Lemma \ref{lem: handling the summation range}, the support of this sum is finite, so extending the sum to a finite $S_k$ invariant subset containing the support, by Lemma~\ref{lem:powery}, we get that this is equal to the right-hand side of the sought expression.
\end{proof}

\subsection{The case $s=t$.}
We conclude the section related to type $\mathrm{C}_n$ by proving the second of the two formulae \eqref{thm-type-c-t-char} presented in Theorem \ref{main-type-c}, that is
\[
P_{(k^n)}^{\mathrm{C}_n}(X;t,t)=
\sum_{\substack{y \in \mathbb{Z}^{k}\\ y_k\ge0}}
(-1)^{|y|+\mathrm{inv}(\nu)+ \sum_{j=1}^k [y_j<0]}
t^{\sum_{j=1}^k \left( \frac{1}{2}Ky_j^2 + jy_j \right)}
\sp_{\lambda(\nu)}(X),\]
where $K=2k$ and $\nu_a = |Ky_a + a|-1$ for $a=1,\dots,k$  and $\lambda(\nu)$ is defined in Definition~\ref{de:lambda(nu)}. 
\begin{proof}[Proof of ~\cref{thm-type-c-t-char}] Set $s_1=\sqrt{t}$ and $s_2=-\sqrt{t}$ in the expression of Lemma~\ref{magical manipulation} to get
    \begin{multline*}
    P_{(k^n)}^{\mathrm{C}_n}(X;t,t)= \frac{(\sqrt{t})^{\binom{n}{2}}}{2}
    \prod_{i=1}^nx_i^{n-1}\prod_{1\leq i<j\leq n}\frac{1}{(1-x_ix_j)(x_i-x_j)}\\\times\det_{1\le i,j\le n}\left(\frac{(1-\sqrt{t}x_i)(1+\sqrt{t}x_i)}{1-x_i^2}
    \left((\sqrt{t})^{j-1}x_i^{j-1-k}+
    (\sqrt{t})^{-j+1}x_i^{-j+1-k}\right)\right.\\\left.-\frac{(\sqrt{t}-x_i)(-\sqrt{t}-x_i)}{1-x_i^2}
    \left((\sqrt{t})^{j-1}x_i^{-j+1+k}
    +(\sqrt{t})^{-j+1}x_i^{j-1+k}\right)\right)\\
    =\frac{(\sqrt{t})^{\binom{n}{2}}}{2}
    \prod_{i=1}^n\frac{x_i^{n-1}}{x_i^2-1}\prod_{1\leq i<j\leq n}\frac{1}{(1-x_ix_j)(x_i-x_j)}\\\times\det_{1\le i,j\le n}\left((tx_i^2-1)\left((\sqrt{t})^{j-1}x_i^{j-1-k}+
    (\sqrt{t})^{-j+1}x_i^{-j+1-k}\right)\right.\\\left.+(x_i^2-t)
    \left((\sqrt{t})^{j-1}x_i^{-j+1+k}
    +(\sqrt{t})^{-j+1}x_i^{j-1+k}\right)\right).
\end{multline*}
We identify the Weyl denominator of type \ref{eq: Denom type C} up to a factor of $\prod_{i=1}^nx_i^{-1}$, which we distribute over the determinant, and a sign $(-1)^{\binom{n}{2}}$, to write our expression as
\begin{multline*}
    \frac{(-\sqrt{t})^{\binom{n}{2}}}{2}\det_{1\le i,j\le n}\left((tx_i-\x_i)\left((\sqrt{t})^{j-1}x_i^{j-1-k}+
    (\sqrt{t})^{-j+1}x_i^{-j+1-k}\right)\right.\\\left.+(x_i-t\x_i)
    \left((\sqrt{t})^{j-1}x_i^{-j+1+k}
    +(\sqrt{t})^{-j+1}x_i^{j-1+k}\right)\right)/\det_{1\leq i,j\leq n}\left(x_i^{n-j+1}-\x^{n-j+1}\right).
\end{multline*}
Using our shorthand notation $\eta_j(m)$ for the expression $x_j^m-\x_j^m$ again, we can rewrite this expression as
\begin{multline*}
    \frac{(-\sqrt{t})^{\binom{n}{2}}}{2}\det_{1\le i,j\le n}\left((\sqrt{t})^{-j+1}\eta_i(k+j)
    -(\sqrt{t})^{j+1}\eta_i(k-j)
    \right.\\\left.-(\sqrt{t})^{-(j-2)+1}\eta_i(k+j-2)
    +(\sqrt{t})^{(j-2)+1}\eta_i(k-(j-2))\right)/\det_{1\leq i,j\leq n}\left(\eta_i(n-j+1)\right).
\end{multline*}
Now observe that the determinant in the numerator equals the following determinant, where we factored out the $2$ from the first column, since the first matrix can be obtained from the latter by applying elementary column operations $C_j \rightarrow C_j-C_{j-2}$ for $j=n,\dots, 3$,
\begin{multline*}
    2\det_{1\leq i,j\leq n}\left((\sqrt{t})^{-j+1}\eta_i(k+j)
    -(\sqrt{t})^{j+1}\eta_i(k-j)\right)\\=2(\sqrt{t})^n\det_{1\leq i,j\leq n}\left((\sqrt{t})^{-j}\eta_i(k+j)
    -(\sqrt{t})^{j}\eta_i(k-j)\right).
\end{multline*}
Therefore, our expression is equal to
\begin{equation*}
    P_{(k^n)}^{\mathrm{C}_n}(X;t,t)=(-1)^{\binom{n}{2}}(\sqrt{t})^{\binom{n}{2}+n}\frac{\det\limits_{1\leq i,j\leq n}\left((\sqrt{t})^{-j}\eta_i(k+j)
    -(\sqrt{t})^{j}\eta_i(k-j)\right)}{\det\limits_{1\leq i,j\leq n}\left(\eta_i(n-j+1)\right)}.
\end{equation*}
  We use Lemma~\ref{lem: noodle-doodle expansion 2.3b} with $\eta_j(a)=x_j^a-\x_j^a$, $p_{\pm}(j) = \pm(\sqrt{t})^{\mp j}$ and $K=2k$ to obtain
\begin{multline*}
    \det\limits_{1\leq i,j\leq n}\left((\sqrt{t})^{-j}\eta_i(k+j)
    -(\sqrt{t})^{j}\eta_i(k-j)\right)\\=\sum_{\substack{(y_0,\dots,y_{ k}) \in\Z^{k + 1} \\ y_0,y_k \geq 0}}
        (-1)^ {|y|+\sum_{i=0}^{k} [y_i<0] + \inv(\vec z_{0,k})} 
        t^{-\frac{1}{2}\binom{n+1}{2}+\sum_{i=0}^k\left(\frac{1}{2}Ky_i^2+y_ii\right)}\det_{\substack{1\le i\le n \\0 \le j\le k+n \\ j \not=z(y_r), r=0,\dots,k}}({\eta_i(j)}),
\end{multline*}
where $z(y_j)=\alpha_j(Ky_j+j)$ and we denote $\vec z_{0,k}\coloneqq (z(y_0),\dots,z(y_k))$. This is because as in the previous proofs we have that $\binom{k+1}{2} - \sum_{j=0}^{k} b_j = \inv(\vec z)$ and furthermore we know that 
\begin{equation*}
    \prod_{j=1}^np_+(j) 
        \prod_{i=0}^{k}
        \prod_{\ell=0}^{\alpha_iy_i-1}
        \frac{p_-(K\ell+k+\alpha_i i)}{p_+(K\ell+k+\alpha_i i)}=(-1)^{|y|}t^{-\frac{1}{2}\binom{n+1}{2}+\sum_{i=0}^k\left(\frac{1}{2}Ky_i^2+y_ii\right)},
\end{equation*}
as every $p_-$ contributes a factor of $-1$ and by the following simple calculation:
\begin{multline*}
    -\sum_{j=1}^n\frac{j}{2}+\sum_{i=0}^k\sum_{\ell=0}^{\vert y_i\vert-1}(K\ell+k+\alpha_i i)=-\frac{1}{2}\binom{n+1}{2}+\sum_{i=0}^k\left(K\binom{\vert y_i\vert}{2}+\vert y_i\vert k+y_ii\right)\\=-\frac{1}{2}\binom{n+1}{2}+\sum_{i=0}^k\left(\frac{1}{2}Ky_i^2+y_ii\right).
\end{multline*}
Inserting this into our expression from before we get that
\begin{multline*}
    P_{(k^n)}^{\mathrm{C}_n}(X;t,t)\\=\sum_{\substack{(y_0,\dots,y_{ k}) \in\Z^{k + 1} \\ y_0,y_k \geq 0}}
        (-1)^ {|y|+\binom{n}{2}+\sum_{i=0}^{k} [y_i<0] + \inv(\vec z_{0,k})} 
        t^{\sum_{i=0}^k\left(\frac{1}{2}Ky_i^2+y_ii\right)}\frac{\det\limits_{\substack{1\le i\le n \\0 \le j\le k+n \\ j \not=z(y_r), r=0,\dots,k}}({\eta_i(j)})}{\det\limits_{1\leq i,j\leq n}\left(\eta_i(n-j+1)\right)}.
\end{multline*}
Now consider the determinant. If $y_0\ne0$, then the determinant has a column of zeros and vanishes. Therefore, we may assume $y_0=0$ from now on. Note that then
\begin{equation*}
    \inv(\vec z_{0,k})=\inv(0,z(y_1),\dots,z(y_k))=\inv(\vec z_{1,k})+\vert\{j\in[1,k]\mid z(y_j)<0\}=\inv(\vec z_{1,k}),
\end{equation*}
as $z(y_j)\geq 0$ for all $1\leq j\leq k$. We have arrived at the following expression
\begin{equation*}
    P_{(k^n)}^{\mathrm{C}_n}(X;t,t)
  =  
  \sum_{\substack{(y_1,\ldots,y_k)\in\Z^k\\ y_k\ge0}}
  (-1)^{|y|+\binom{n}{2} + \sum_{i=1}^k[y_i<0] + \inv(\vec z_{1,k})}
  t^{\sum_i \left( \frac{1}{2}Ky_i^2 +y_ii \right)}
  \frac{\det\limits_{\substack{1\le i\le n \\1 \le j\le k+n \\ j \not=z(y_r), r=1,\dots,k}}({\eta_i(j)})}{\det\limits_{1 \le i,j \le n}(\eta_i(n-j+1))}.
\end{equation*}
To be able to use the Weyl character formula of type $\mathrm{C}_n$ \eqref{eq: Weyl character formula type C} and Corollary~\ref{cor:useful.complementary}, we define a vector $\nu\coloneqq (\nu_1,\dots,\nu_k)$ by
$$\nu_i\coloneqq z(y_i)-1=\alpha_i(Ky_i+i)-1=\vert Ky_i+i\vert-1$$
and observe that $\inv(\vec z)=\inv(\nu).$
Therefore, we can rewrite our expression as
\begin{equation*}
    P_{(k^n)}^{\mathrm{C}_n}(X;t,t)
  =  
  \sum_{\substack{(y_1,\ldots,y_k)\in\Z^k\\ y_k\ge0}}
  (-1)^{|y|+\binom{n}{2} + \sum_i[y_i<0] + \inv(\nu)}
  t^{\sum_{i=1}^k \left( \frac{1}{2}Ky_i^2 +y_ii \right)}
  \frac{\det\limits_{\substack{1\le i\le n \\1 \le j\le k+n \\ j \not=\nu_r+1, r=1,\dots,k}}({\eta_i(j)})}{\det\limits_{1 \le i,j \le n}(\eta_i(n-j+1))}.
\end{equation*}
By the Weyl character formula of type $\mathrm{C}_n$ \eqref{eq: Weyl character formula type C} and Corollary~\ref{cor:useful.complementary} we have that
$$(-1)^{\binom{n}{2}}\frac{\det\limits_{\substack{1\le i\le n \\1 \le j\le k+n \\ j \not=\nu_r+1, r=1,\dots,k}}(\eta_i(j))}{\det\limits_{1 \le i,j \le n}(\eta_i(n-j+1))}=\sp_{\lambda(\nu)}(X). $$
Note that the condition for the matrix in the numerator to be a square matrix is precisely the condition that the $y_a$'s corresponding $\nu_a$'s are inside $\{0,\dots,n+k-1\}$ and are all different. Inserting into our expression from above yields
\begin{equation*}
    P_{(k^n)}^{\mathrm{C}_n}(X;t,t)
  =  
  \sum_{\substack{(y_1,\ldots,y_k)\in\Z^k\\ y_k\ge0}}
  (-1)^{|y|+ \sum_i[y_i<0] + \inv(\nu)}
  t^{\sum_i \left( \frac{1}{2}Ky_i^2 +y_ii \right)}
  \sp_{\lambda(\nu)}(X),
\end{equation*}
where we sum over all $y\in\Z^k$ such that the corresponding $\nu_a$'s are inside $\{0,\dots,n+k-1\}$ and are all different, as desired.
\end{proof}

Having finished the proof of \eqref{thm-type-c-t-char} we are ready to prove the corresponding Jacobi--Trudi formula \eqref{cor-type-c-t}, that is
    \[
    P_{(k^n)}^{\mathrm{C}_n}(X;t,t)=
    \sum_{\substack{y \in \mathbb{Z}^k,\\ |y| \text{ even }}}
\det_{1 \le i,j \le k}
\left(
t^{\frac{1}{2} K y_i^2 - jy_i}
\left(
\dot{e}_{\,n-i+j-K y_i}
-
\dot{e}_{\,n+i+j- Ky_i}
\right)
\right),
    \]
    where $K=2k.$

\begin{proof}[Proof of ~\cref{cor-type-c-t}]
Let $\nu_a = |Ky_a + a| - 1$ as in~\eqref{thm-type-c-t-char}. 
Unravelling the definitions preceding Lemma~\ref{lem:flip} we have
\begin{multline*}
\lambda(\nu)'_i = (\mathrm{sort}(\nu) - \rho_k)^{\square n^k}_i = 
n - (\mathrm{sort}(\nu) - \rho_k)_{k+1-i} =
n - \mathrm{sort}(\nu)_{k+1-i} + i -1\\
= n - |Ky_{\sigma(i)} + \sigma(i)| + i\,,
\end{multline*}
where $\sigma$ is the permutation that sorts $\nu$ into increasing order.
Apply the dual Jacobi--Trudi identity of type $\mathrm{C}_n$ \eqref{eq: dual JT type C} to the formula of~\cref{thm-type-c-t-char} to get
\[
P_{(k^n)}^{\mathrm{C}_n}(X;t,t)=
\sum_{\substack{y \in \mathbb{Z}^k\\y_k\geq 0}}
(-1)^{|y|+\mathrm{inv}(\nu) + \sum_{j=1}^k [y_j<0]}
t^{\sum_{j=1}^k \left( \frac{1}{2}Ky_j^2 + jy_j \right)}
\det_{1\leq i,j\leq k}(\dot{e}_{\lambda(\nu)'_i-i+j}
- 
\dot{e}_{\lambda(\nu)'_i-i-j}).
\]
By the assumption in Theorem \ref{main-type-c}, the sum ranges over all vectors $y\in\Z^k$ such that $\nu_1,\dots,\nu_k\in[0,n+k-1]$ are all different elements and $y_k\geq 0$. By Lemma \ref{lem: handling the summation range} the determinant on the right hand side of this sum vanishes outside of this range of summation, hence we can extend the sum, to a sum over all $y\in\Z^k$ with $y_k\geq 0$. 

Plugging in the expression for $\lambda(\nu)'_i$, from before this determinant is equal to
\[
\det_{1\leq i,j\leq k}(
\dot{e}_{
n - |Ky_{\sigma(i)} + \sigma(i)|  + j
} -
\dot{e}_{
n - |Ky_{\sigma(i)} + \sigma(i)| - j
}).
\]
Now rearrange the rows of the matrix by the permutation $\sigma^{-1}$. This cancels the factor $\sgn(\sigma) = (-1)^{\mathrm{inv}(\nu)}$, as $\sigma$ is the permutation that sorts $\nu$ in increasing order, hence the sign is given by the number of inversions of $\nu$. The matrix becomes
\[
\det_{1\leq i,j\leq k}(
\dot{e}_{
n - |Ky_{i} + i| + j 
} -
\dot{e}_{
n - |Ky_{i} + i| - j 
}).
\]
We distinguish cases according to the sign of $y_i$ and fix some $1\leq i,j\leq k$.
If $y_i$ is non-negative, we use the formula $\dot{e}_{n+r}=\dot{e}_{n-r}$  to get
\begin{equation*}\dot{e}_{
n - |Ky_{i} + i| + j 
} -
\dot{e}_{
n - |Ky_{i} + i| - j 
}= 
\dot{e}_{
n - Ky_{i} - i + j
} -
\dot{e}_{
n - Ky_{i} - i - j
} =
\dot{e}_{
n + Ky_{i} + i - j
} -
\dot{e}_{
n + Ky_{i} + i + j 
}\,.
\end{equation*}
If $y_i$ is negative, one directly gets
\begin{equation*}
\dot{e}_{
n - |Ky_{i} + i| + j 
} -
\dot{e}_{
n - |Ky_{i} + i| - j 
}=
-(
\dot{e}_{
n + Ky_{i} + i - j
} -
\dot{e}_{
n + Ky_{i} + i + j 
})\,.
\end{equation*}
Pulling this sign out of the determinant using the multilinearity we obtain an additional factor of $(-1)^{\sum_j [y_j<0]}$. Replacing $y_i$ by $-y_i$ and then transposing the matrix yields
\[
P_{(k^n)}^{\mathrm{C}_n}(X;t,t)=
\sum_{\substack{y \in \Z^k\\y_k\leq 0}}
(-1)^{|y|}t^{\sum_{j=1}^k \left( \frac{1}{2}K y_j^2 - jy_j \right)}
\det_{1 \le i,j \le k}
\left(
\dot{e}_{\,n-i+j-K y_j}
-
\dot{e}_{\,n+i+j- Ky_j}
\right)\]
To conclude the proof using Lemma~\ref{lem:powery}, we need one more symmetry argument. Let us denote one summand of this sum by 
$$S(y_1,\dots,y_k)\coloneqq (-1)^{|y|}t^{\sum_{j=1}^k \left( \frac{1}{2}K y_j^2 - jy_j \right)}
\det_{1 \le i,j \le k}
\left(
\dot{e}_{n-i+j-K y_j}
-
\dot{e}_{n+i+j- Ky_j}
\right).$$ 
By using $\dot{e}_{n+r} = \dot{e}_{n-r}$ in the third equality and recalling that $K=2k$ we get the following symmetry of the $k$-th column of the matrix:
\begin{multline*}
    \left(\dot{e}_{\,n-i+k-K y_k}
-
\dot{e}_{\,n+i+k- Ky_k}\right)|_{y_k=1-y_k}=\dot{e}_{\,n-i+k-K(1- y_k)}
-
\dot{e}_{\,n+i+k- K(1-y_k)}\\=\dot{e}_{\,n-i-k+K y_k}
-
\dot{e}_{\,n+i-k+Ky_k}=\dot{e}_{\,n+i+k-K y_k}
-
\dot{e}_{\,n-i+k-Ky_k}\\=-\left(\dot{e}_{\,n-i+k-K y_k}
-
\dot{e}_{\,n+i+k- Ky_k}\right).
\end{multline*}
Combining this with the observations that
$$\left.\frac{1}{2}Ky_k^2-ky_k\right|_{y_k=(1-y_k)}=k(1-y_k)^2-k(1-y_k)=\frac{1}{2}Ky_k^2-ky_k,$$ 
and 
$$\left.(-1)^{|y|}\right|_{y_k=1-y_k}=(-1)^{\sum_{i=1}^{k-1}y_i+1-y_k}=-(-1)^{|y|}$$
we get that
\[
S(y_1,\dots,y_{k-1},y_k) = S(y_1,\dots,y_{k-1},1-y_k).
\]
Now denote by $R = \{y\in\Z^k\mid\sum_j y_j~\text{even}\}$ and decompose the sum in our expression as follows:
\begin{multline*}
P_{(k^n)}^{\mathrm{C}_n}(X;t,t)=
\sum_{\substack{y\in R\\y_k\le0}} S(y_1,\dots,y_{k-1},y_k)
+
\sum_{\substack{y\in\Z^k\setminus R\\y_k\le0}} S(y_1,\dots,y_{k-1},y_k)\\
=\sum_{\substack{y\in R\\y_k\le0}} S(y_1,\dots,y_{k-1},y_k)
+
\sum_{\substack{y\in\Z^k\setminus R\\y_k\le0}} S(y_1,\dots,y_{k-1},1-y_k)\\=
\sum_{\substack{y\in R\\y_k\le0}} S(y_1,\dots,y_{k-1},y_k)
+
\sum_{\substack{y\in R\\y_k\ge1}} S(y_1,\dots,y_{k-1},y_k)
\\=
\sum_{\substack{y\in\Z^k\\\vert y\vert~\text{even}}}
t^{\sum_{j=1}^k \left( \frac{1}{2}K y_j^2 - jy_j \right)}
\det_{1 \le i,j \le k}
\left(
\dot{e}_{\,n-i+j-K y_j}
-
\dot{e}_{\,n+i+j- Ky_j}
\right),
\end{multline*}
where we can omit the sign $(-1)^{|y|}$ as $|y|$ is always even in the final sum. By Lemma \ref{lem: handling the summation range}, the support of this sum is finite, so extending the sum to a finite $S_k$ invariant subset containing the support and retaining the additional parameter restriction $|y|$ even, by Lemma~\ref{lem:powery}, we get that this yields the desired statement.
\end{proof}

\section{Type \texorpdfstring{$\mathrm{BC}_n$}{BCn}}
\label{sec:type-bc}
We conclude the paper by establishing the two formulae related to type $\mathrm{BC}_n$ presented in Theorem \ref{thm-type-bc-0-char-all} and Corollary \ref{cor-type-bc}. We start with the character expansion from \eqref{thm-type-bc-0-char}, that is
\[
P_{(k^n)}^{\mathrm{BC}_n}(X;t,-\sqrt{t},0)=
\sum_{y \in \mathbb{Z}^{k}}
(-1)^{|y| + \mathrm{inv}(\nu)- \sum_{j=1}^k [y_j<0]}
t^{\sum_{j=1}^k \left( \frac{1}{2}Ky_j^2 + jy_j\right)}
\sp_{\lambda(\nu)}(X),\]
where $K=2k+1$ and $\nu_a = |Ky_a + a|-1$ for $a=1,\dots,k$  and $\lambda(\nu)$ is defined in Definition~\ref{de:lambda(nu)}. 
\begin{proof}[Proof of Equation~\ref{thm-type-bc-0-char}] 
Set $s_1=-\sqrt{t}$ and $s_2=0$ in the expression of Lemma~\ref{magical manipulation} to get
   \begin{multline*}
    P_{(k^n)}^{\mathrm{BC}_n}(X;t,-\sqrt{t},0)= \frac{(\sqrt{t})^{\binom{n}{2}}}{2}
    \prod_{i=1}^nx_i^{n-1}\prod_{1\leq i<j\leq n}\frac{1}{(1-x_ix_j)(x_i-x_j)}\\\times\det_{1\le i,j\le n}\left(\frac{(1+\sqrt{t}x_i)}{1-x_i^2}
    \left((\sqrt{t})^{j-1}x_i^{j-1-k}+
    (\sqrt{t})^{-j+1}x_i^{-j+1-k}\right)\right.\\\left.-\frac{(\sqrt{t}+x_i)x_i}{1-x_i^2}
    \left((\sqrt{t})^{j-1}x_i^{-j+1+k}
    +(\sqrt{t})^{-j+1}x_i^{j-1+k}\right)\right)\\=\frac{(\sqrt{t})^{\binom{n}{2}}}{2}
    \prod_{i=1}^n\frac{x_i^{n-1}}{x_i^2-1}\prod_{1\leq i<j\leq n}\frac{1}{(1-x_ix_j)(x_i-x_j)}\\\times\det_{1\le i,j\le n}\left((\sqrt{t}+x_i)x_i
    \left((\sqrt{t})^{j-1}x_i^{-j+1+k}
    +(\sqrt{t})^{-j+1}x_i^{j-1+k}\right)\right.\\\left.-(1+\sqrt{t}x_i)
    \left((\sqrt{t})^{j-1}x_i^{j-1-k}+
    (\sqrt{t})^{-j+1}x_i^{-j+1-k}\right)\right).
\end{multline*}
We identify the Weyl denominator of type \ref{eq: Denom type C} up to a factor of $\prod_{i=1}^nx_i^{-1}$, which we distribute over the determinant, and a sign $(-1)^{\binom{n}{2}}$, to write our expression as
\begin{multline*}
    \frac{(-\sqrt{t})^{\binom{n}{2}}}{2}\det_{1\le i,j\le n}\left((\sqrt{t}+x_i)
    \left((\sqrt{t})^{j-1}x_i^{-j+1+k}
    +(\sqrt{t})^{-j+1}x_i^{j-1+k}\right)\right.\\\left.-(\sqrt{t}+\x_i)
    \left((\sqrt{t})^{j-1}x_i^{j-1-k}+
    (\sqrt{t})^{-j+1}x_i^{-j+1-k}\right)\right)/\det_{1\leq i,j\leq n}\left(x_i^{n-j+1}-\x_i^{n-j+1}\right).
\end{multline*}
Using our shorthand notation $\eta_j(m)$ for the expression $x_j^m-\x_j^m$ again, we can rewrite this expression as
\begin{multline*}
    \frac{(-\sqrt{t})^{\binom{n}{2}}}{2}\det_{1\le i,j\le n}\left( (\sqrt{t})^{j}\eta_i(k+1/2-(j-1/2))
    \right.\\\left.+(\sqrt{t})^{1-j}\eta_i(k+1/2+j-1/2)
    +(\sqrt{t})^{2-j}\eta_i(k+1/2+j-3/2)
    \right.\\\left.+(\sqrt{t})^{j-1}\eta_i(k+1/2-(j-3/2))\right)/\det_{1\leq i,j\leq n}\left(\eta_i(n-j+1)\right).
\end{multline*}
Now observe that the determinant in the numerator equals the following determinant, where we factored out the $2$ from the first column, since the first matrix can be obtained from the latter by applying elementary column operations $C_j \rightarrow C_j+C_{j-1}$ for $j=n,\dots, 2$,
\begin{multline*}
    2\det_{1\leq i,j\leq n}\left(
(\sqrt{t})^{1-j}\eta_i(k+1/2+j-1/2)
    +(\sqrt{t})^{j}\eta_i(k+1/2-(j-1/2)\right)\\=2(\sqrt{t})^{n/2}\det_{1\leq i,j\leq n}\left(
(\sqrt{t})^{-(j-1/2)}\eta_i(k+1/2+j-1/2)
    +(\sqrt{t})^{j-1/2}\eta_i(k+1/2-(j-1/2)\right).
\end{multline*}
Therefore, our expression is given by
\begin{multline*}
    P_{(k^n)}^{\mathrm{BC}_n}(X;t,-\sqrt{t},0)=(-1)^{\binom{n}{2}}(\sqrt{t})^{\binom{n}{2}+n/2} \\\times \frac{\det_{1\leq i,j\leq n}\left(
(\sqrt{t})^{-(j-1/2)}\eta_i(k+1/2+j-1/2)
 +(\sqrt{t})^{j-1/2}\eta_i(k+1/2-(j-1/2)\right)}{\det_{1\leq i,j\leq n}\left(\eta_i(n-j+1)\right)}.
\end{multline*}
We use Lemma~\ref{lem: noodle-doodle expansion 2.5a} with $\eta_j(a)=x_j^a-\x_j^a$, $p_{\pm}(j) = (\sqrt{t})^{\mp (j-\frac{1}{2})}$ and $K=2k+1$ to obtain
\begin{multline*}
    \det_{1\leq i,j\leq n}\left(
(\sqrt{t})^{-(j-1/2)}\eta_i(k+1/2+j-1/2)
 +(\sqrt{t})^{j-1/2}\eta_i(k+1/2-(j-1/2)\right)\\=\sum_{\substack{(y_0,\dots,y_{ k}) \in\Z^{k+1}\\ y_0\geq 0}}
        (-1)^{|y| + \sum_{j=0}^k [y_j<0] + \inv(\vec z_{0,k})} t^{-\frac{n^2}{4}+\sum_{i=0}^k\left(\frac{1}{2}Ky_i^2+y_ii\right)}
         \det_{\substack{1\le i\le n \\0 \le j\le k+n \\ j \not=z(y_r), r=0,\dots,k}}({\eta_i(j)}),
\end{multline*}
where $z(y_j)=\alpha_j(Ky_j+j)$ and we denote $\vec z_{0,k}\coloneqq(z(y_o),\dots,z(y_k))$. This is because as in the previous proofs we have that $\binom{k+1}{2} - \sum_{j=0}^{k} b_j = \inv(\vec z_{0,k})$ and furthermore 
\begin{equation*}
    \prod_{j=1}^{n}p_+(j) 
        \prod_{i=0}^{k}
        \prod_{\ell=0}^{\alpha_iy_i-1}
        \frac{p_-(K\ell+k+\alpha_i i+1)}{p_+(K\ell+k+\alpha_i i+1)}=t^{-\frac{n^2}{4}+\sum_{i=0}^k\left(\frac{1}{2}Ky_i^2+y_ii\right)},
\end{equation*}
as the following simple calculations shows:
\begin{multline*}
    -\sum_{j=1}^n\left(\frac{j}{2}-\frac{1}{4}\right)+\sum_{i=0}^k\sum_{\ell=0}^{\vert y_i\vert-1}\left(K\ell+k+\alpha_ii+1-1/2\right)\\=-\frac{1}{2}\binom{n+1}{2}+\frac{n}{4}+\sum_{i=0}^k\left(K\binom{\vert y_i\vert }{2}+\left(k+\frac{1}{2}\right)\vert y_i\vert +y_ii\right)=-\frac{n^2}{4}+\sum_{i=0}^k\left(\frac{1}{2}Ky_i^2+y_ii\right).
\end{multline*}
Inserting this into our expression from before we have that
\begin{multline*}
    P_{(k^n)}^{\mathrm{BC}_n}(X;t,-\sqrt{t},0)\\=\sum_{\substack{(y_0,\dots,y_{ k}) \in\Z^{k+1}\\ y_0\geq 0}}
        (-1)^{\binom{n}{2}+|y| + \sum_{j=0}^k [y_j<0] + \inv(\vec z_{0,k})} t^{\sum_{i=0}^k\left(\frac{1}{2}Ky_i^2+y_ii\right)}
         \frac{\det\limits_{\substack{1\le i\le n \\0 \le j\le k+n \\ j \not=z(y_r), r=0,\dots,k}}({\eta_i(j)})}{\det_{1\leq i,j\leq n}\left(\eta_i(n-j+1)\right)}.
\end{multline*}
Now consider the determinant
$$\det\limits_{\substack{1\le i\le n \\0 \le j\le k+n \\ j \not=z(y_r), r=0,\dots,k}}({\eta_i(j)}),$$
from the numerator and note that if $y_0\ne0$, the determinant vanishes as one column is all zeros. Therefore, we may assume $y_0=0$ from now on. Note that then
\begin{equation*}
    \inv(\vec z_{0,k})=\inv(0,z(y_1),\dots,z(y_k))=\inv(\vec z_{1,k})+\vert\{j\in[1,k]\mid z(y_j)<0\}=\inv(\vec z_{1,k}).
\end{equation*}
We have arrived at the following expression
\begin{multline*}
    P_{(k^n)}^{\mathrm{BC}_n}(X;t,-\sqrt{t},0)\\=\sum_{(y_1,\dots,y_{ k}) \in\Z^{k}}(-1)^{\binom{n}{2}+|y| + \sum_{j=1}^k [y_j<0] + \inv(\vec z_{1,k})}t^{\sum_{i=1}^k\left(\frac{1}{2}Ky_i^2+y_ii\right)}\frac{\det\limits_{\substack{1\le i\le n \\1 \le j\le k+n \\ j \not=z(y_r), r=1,\dots,k}}({\eta_i(j)})}{\det_{1\leq i,j\leq n}\left(\eta_i(n-j+1)\right)}.
\end{multline*}
To be able to use the Weyl character formula of type $\mathrm{C}_n$ \eqref{eq: Weyl character formula type C} and Corollary~\ref{cor:useful.complementary}, we define a vector $\nu\coloneqq (\nu_1,\dots,\nu_k)$ by
$$\nu_i\coloneqq z(y_i)-1=\alpha_i(Ky_i+i)-1=\vert Ky_i+i\vert-1$$
and observe that $\inv(\vec z)=\inv(\nu).$
Therefore, we can rewrite our expression as
\begin{multline*}
    P_{(k^n)}^{\mathrm{BC}_n}(X;t,-\sqrt{t},0)\\=\sum_{(y_1,\dots,y_{ k}) \in\Z^{k}}(-1)^{\binom{n}{2}+|y| + \sum_{j=1}^k [y_j<0] + \inv(\nu)}t^{\sum_{i=1}^k\left(\frac{1}{2}Ky_i^2+y_ii\right)}\frac{\det\limits_{\substack{1\le i\le n \\1 \le j\le k+n \\ j \not=\nu_r+1, r=1,\dots,k}}({\eta_i(j)})}{\det_{1\leq i,j\leq n}\left(\eta_i(n-j+1)\right)}.
\end{multline*}
By the Weyl character formula of type $\mathrm{C}_n$ \eqref{eq: Weyl character formula type C} and Corollary~\ref{cor:useful.complementary} we have that
$$(-1)^{\binom{n}{2}}\frac{\det\limits_{\substack{1\le i\le n \\1 \le j\le k+n \\ j \not=\nu_r+1, r=1,\dots,k}}(\eta_i(j))}{\det\limits_{1 \le i,j \le n}(\eta_i(n-j+1))}=\sp_{\lambda(\nu)}(X). $$
Note that the condition for the matrix in the numerator to be a square matrix is precisely the condition that the $y_a$'s corresponding to the $\nu_a$'s are inside $\{0,\dots,n+k-1\}$ and are all different. Inserting into our expression from above yields
\begin{equation*}
    P_{(k^n)}^{\mathrm{BC}_n}(X;t,-\sqrt{t},0)=\sum_{(y_1,\dots,y_{ k}) \in\Z^{k}}(-1)^{|y| + \sum_{j=1}^k [y_j<0] + \inv(\nu)}t^{\sum_{i=1}^k\left(\frac{1}{2}Ky_i^2+y_ii\right)}\sp_{\lambda(\nu)}(X)
\end{equation*}
where we sum over all $y\in\Z^k$ such that the corresponding $\nu_a$'s are inside $\{0,\dots,n+k-1\}$ and are all different, as desired.

\end{proof}

Having established \eqref{thm-type-bc-0-char} we are ready to prove~\eqref{cor-type-bc-0}:
\[       
P_{(k^n)}^{\mathrm{BC}_n}(X;t,-t^{1/2},0)=
    \sum_{y \in \mathbb{Z}^k}
\det_{1 \le i,j \le k} 
\left(
(-1)^{y_i}
t^{\frac{1}{2} K y_i^2 - jy_i}
\left(
\dot{e}_{\,n-i+j-K y_i}
-
\dot{e}_{\,n+i+j- Ky_i}
\right)
\right),
\]
    where $K=2k+1.$

\begin{proof}[Proof of Corollary~\ref{cor-type-bc}]
  Let $\nu_a = |Ky_a + a| - 1$ as in Theorem~\ref{thm-type-bc-0-char-all}. 
Unravelling the definitions preceding Lemma~\ref{lem:flip} we have
\begin{multline*}
\lambda(\nu)'_i = (\mathrm{sort}(\nu) - \rho_k)^{\square n^k}_i = 
n - (\mathrm{sort}(\nu) - \rho_k)_{k+1-i} =
n - \mathrm{sort}(\nu)_{k+1-i} + i -1\\
= n - |Ky_{\sigma(i)} + \sigma(i)| + i\,,
\end{multline*}
where $\sigma$ is the permutation that sorts $\nu$ into increasing order.
Apply the dual Jacobi--Trudi identity of type $\mathrm{C}_n$ \eqref{eq: dual JT type C} to the formula of~\cref{thm-type-bc-0-char} to get
\begin{multline*}
    P_{(k^n)}^{\mathrm{BC}_n}(X;t,-\sqrt{t},0)\\=\sum_{y \in\Z^{k}}(-1)^{|y| + \sum_{j=1}^k [y_j<0] + \inv(\nu)}t^{\sum_{j=1}^k\left(\frac{1}{2}Ky_j^2+y_jj\right)}\det_{1\leq i,j\leq k}(\dot{e}_{\lambda(\nu)'_i-i+j}
- 
\dot{e}_{\lambda(\nu)'_i-i-j})
\end{multline*}
By the assumption in Theorem \ref{thm-type-bc-0-char-all}, the sum ranges over all vectors $y\in\Z^k$ such that $\nu_1,\dots,\nu_k\in[0,n+k-1]$ are all different elements. By Lemma \ref{lem: handling the summation range} the determinant on the right hand side of this sum vanishes outside of this range of summation, hence we can extend the sum, to a sum over all $y\in\Z^k$. 

Plugging in the expression for $\lambda(\nu)'_i$, from before this determinant is equal to
\[
\det_{1\leq i,j\leq k}(
\dot{e}_{
n - |Ky_{\sigma(i)} + \sigma(i)|  + j
} -
\dot{e}_{
n - |Ky_{\sigma(i)} + \sigma(i)| - j
}).
\]
Now rearrange the rows of the matrix by the permutation $\sigma^{-1}$. This cancels the factor $\sgn(\sigma) = (-1)^{\mathrm{inv}(\nu)}$, as $\sigma$ is the permutation that sorts $\nu$ in increasing order, hence the sign is given by the number of inversions of $\nu$. The matrix becomes
\[
\det_{1\leq i,j\leq k}(
\dot{e}_{
n - |Ky_{i} + i| + j 
} -
\dot{e}_{
n - |Ky_{i} + i| - j 
}).
\]
We distinguish cases according to the sign of $y_i$ and fix some $1\leq i,j\leq k$.
If $y_i$ is non-negative, we use the formula $\dot{e}_{n+r}=\dot{e}_{n-r}$  to get
\begin{equation*}\dot{e}_{
n - |Ky_{i} + i| + j 
} -
\dot{e}_{
n - |Ky_{i} + i| - j 
}= 
\dot{e}_{
n - Ky_{i} - i + j
} -
\dot{e}_{
n - Ky_{i} - i - j
} =
\dot{e}_{
n + Ky_{i} + i - j
} -
\dot{e}_{
n + Ky_{i} + i + j 
}\,.
\end{equation*}
If $y_i$ is negative, one directly gets
\begin{equation*}
\dot{e}_{
n - |Ky_{i} + i| + j 
} -
\dot{e}_{
n - |Ky_{i} + i| - j 
}=
-(
\dot{e}_{
n + Ky_{i} + i - j
} -
\dot{e}_{
n + Ky_{i} + i + j 
})\,.
\end{equation*}
Pulling this sign out of the determinant using the multilinearity, we obtain an additional factor of $(-1)^{\sum_{j=1}^k [y_j<0]}$. Replacing $y_i$ by $-y_i$ and then transposing the matrix yields
\[
P_{(k^n)}^{\mathrm{BC}_n}(X;t,-\sqrt{t},0)=
\sum_{y \in \Z^k}(-1)^{\vert y\vert}
t^{\sum_{j=1}^k \left( \frac{1}{2}K y_j^2 - jy_j \right)}
\det_{1 \le i,j \le k}
\left(
\dot{e}_{\,n-i+j-K y_j}
-
\dot{e}_{\,n+i+j- Ky_j}
\right).\]
By Lemma \ref{lem: handling the summation range}, the support of this sum is finite, so extending the sum to a finite $S_k$ invariant subset containing the support, by Lemma~\ref{lem:powery}, we conclude the proof.
\end{proof}

\subsection*{Acknowledgements}
AI was used exclusively for proofreading the final text.

\bibliography{Bibliography}
\bibliographystyle{alpha}
\end{document}